\documentclass[10pt]{amsart}

\makeatletter
\def\l@subsection{\@tocline{2}{0pt}{2.5pc}{5pc}{}}
\renewcommand\tocchapter[3]{%
  \indentlabel{\@ifnotempty{#2}{\ignorespaces#2.\quad}}#3%
}
\newcommand\@dotsep{4.5}
\def\@tocline#1#2#3#4#5#6#7{\relax
  \ifnum #1>\c@tocdepth 
  \else
    \par \addpenalty\@secpenalty\addvspace{#2}%
    \begingroup \hyphenpenalty\@M
    \@ifempty{#4}{%
      \@tempdima\csname r@tocindent\number#1\endcsname\relax
    }{%
      \@tempdima#4\relax
    }%
    \parindent\z@ \leftskip#3\relax \advance\leftskip\@tempdima\relax
    \rightskip\@pnumwidth plus1em \parfillskip-\@pnumwidth
    #5\leavevmode\hskip-\@tempdima{#6}\nobreak
    \leaders\hbox{$\m@th\mkern \@dotsep mu\hbox{.}\mkern \@dotsep mu$}\hfill
    \nobreak
    \hbox to\@pnumwidth{\@tocpagenum{#7}}\par
    \nobreak
    \endgroup
  \fi}
\makeatother
\AtBeginDocument{%
\makeatletter
\expandafter\renewcommand\csname r@tocindent0\endcsname{0pt}
\makeatother
}
\def\l@subsection{\@tocline{2}{0pt}{2.5pc}{5pc}{}}

\usepackage[english]{babel}
\usepackage[utf8x]{inputenc}
\usepackage[T1]{fontenc}
\usepackage[a4paper,top=3cm,bottom=3cm,left=2.7cm,right=2.7cm,marginparwidth=2cm]{geometry}
\usepackage{indentfirst}
\usepackage[dvipsnames]{xcolor}

\usepackage{amsmath}
\usepackage{amsthm}
\usepackage[color=pink]{todonotes}
\usepackage[colorlinks=true, linktocpage=true, linkcolor=blue, citecolor=blue]{hyperref}
\usepackage{cleveref}
\usepackage{tikz-cd}
\usepackage{amsfonts}
\usepackage{amssymb}
\usepackage{graphicx}  
\usepackage{mathrsfs}
\usepackage{mathtools}

\theoremstyle{plain}
\newtheorem{theorem}{Theorem}[section]
\newtheorem{lemma}[theorem]{Lemma}
\newtheorem{corollary}[theorem]{Corollary}
\newtheorem{conjecture}[theorem]{Conjecture}

\newtheorem{proposition}[theorem]{Proposition}

\theoremstyle{definition}
\newtheorem{definition}[theorem]{Definition}
\newtheorem{construction}[theorem]{Construction}
\newtheorem{remark}[theorem]{Remark}
\newtheorem{example}[theorem]{Example}
\numberwithin{equation}{section}

\DeclareMathOperator{\Hom}{Hom}

\DeclareMathOperator{\Spec}{Spec}

\DeclareMathOperator{\Aut}{Aut}

\usepackage{caption} 
\title{Scattering diagrams for generalized cluster algebras}

\author{Lang Mou}
\address{Lang Mou\newline
Mathematisches Institut, Universit\"at zu K\"oln, Weyertal 86-90, 50931 K\"oln, Germany}
\email{langmou@math.uni-koeln.de}

\date{}

\begin{document}

\begin{abstract}
	We construct scattering diagrams for Chekhov--Shapiro's generalized cluster algebras where exchange polynomials are factorized into binomials, generalizing the cluster scattering diagrams of Gross, Hacking, Keel and Kontsevich. They turn out to be natural objects arising in Fock and Goncharov's cluster duality. Analogous features and structures (such as positivity and the cluster complex structure) in the ordinary case also appear in the generalized situation. With the help of these scattering diagrams, we show that generalized cluster variables are theta functions and hence have certain positivity property with respect to the coefficients in the binomial factors.
\end{abstract}

\maketitle

\tableofcontents

\section{Introduction}\label{introduction}

We study generalized cluster algebras in the sense of Chekhov and Shapiro \cite{chekhov2011teichmller}. These algebras are generalizations of the (ordinary) cluster algebras introduced by Fomin and Zelevinsky \cite{fomin2002cluster}, allowing more general exchange polynomials (as opposed to only binomials) in mutations.

We will see that generalized cluster algebras can not only be studied in a similar way as cluster algebras \cite{fomin2002cluster, fomin2003cluster, berenstein2005cluster, fomin2007cluster}, but that they also naturally appear in the context of the \emph{cluster duality} proposed by Fock and Goncharov \cite{fock2009cluster}. A modified version of Fock and Goncharov's cluster duality was formulated and proved by Gross, Hacking, Keel and Kontsevich in \cite{gross2018canonical}. In this paper, we extend the scheme therein to study generalized cluster algebras.

Generalized cluster algebras come in a family containing ordinary cluster algebras. Each algebra in this family can be viewed as (a subalgebra of) the algebra of regular functions of a generalized $\mathcal A$-cluster variety. The (generalized version of) cluster duality says this family is in a sense dual to another family of generalized $\mathcal X$-cluster varieties. In this paper, we demonstrate this duality by reconstructing a family of generalized cluster algebras with principal coefficients $\mathscr A^\mathrm{prin}$ from a general fibre of the corresponding dual family of $\mathcal X$-cluster varieties.

In the ordinary case, the reconstruction is done through a \emph{cluster scattering diagram}, the main technical tool developed in \cite{gross2018canonical}, which is a mathematical structure associated to the dual $\mathcal X$-cluster variety. For our purpose of studying generalized cluster algebras, we construct \emph{generalized cluster scattering diagrams}. This is done by allowing more general wall-crossing functions on the initial incoming walls. It turns out that many features (such as the positivity property of wall-crossings and the cluster complex structure) in the ordinary case still hold in the generalized situation; see \Cref{thm: generalized positivity} and \Cref{thm: wall crossing on cluster chamber}.

Using the techniques of scattering diagrams (and related objects such as broken lines) transplanted from \cite{gross2018canonical}, we are able to prove that generalized cluster monomials are theta functions. As a result, they have certain positivity property coming from that of the scattering diagram. We remark that this positivity is with respect to the coefficients appearing in the binomial factors of exchange polynomials, thus weaker than a conjectural positivity of Chekhov and Shapiro (\Cref{conj: chekhov shapiro}) with respect to the coefficients of exchange polynomials themselves; See \Cref{thm: cluster variable as theta function} and \Cref{subsection: further on positivity} for the precise statements.

We next describe the contents of the paper in more detail.

\subsection{Generalized cluster algebras}

Our way of generalizing cluster algebras is slightly different from \cite{chekhov2011teichmller}, in the way we deal with coefficients. In a sense, one can go from one formulation to the other, in particular when the coefficients are evaluated in some algebraically closed field; see \Cref{subsection: chekhov-shapiro}, \Cref{subsection: specialized coefficients} and also \Cref{subsection: further on positivity}. We replace Fomin and Zelevinsky's binomial exchange relation
\[
x_k'x_k = p_k^+ \prod_{i=1}^n x_i^{[b_{ik}]_+} + p_k^- \prod_{i=1}^n x_i^{[-b_{ik}]_+}
\]
with a more general polynomial exchange relation
\[
x_k'x_k = \prod_{j=1}^{r_k} \left(p_{k,j}^+ \prod_{i=1}^n x_i^{[b_{ik}/r_k]_+} + p_{k,j}^- \prod_{i=1}^n x_i^{[-b_{ik}/r_k]_+} \right).
\]
We require the coefficients $p_{k,j}^\pm$ (in some semifield $(\mathbb P, \oplus, \cdot)$) to satisfy the normalized condition $p_{k,j}^+ \oplus p_{k,j}^- = 1$. The normalization makes mutations deterministic and a particular choice of coefficients named \emph{principal coefficients} (as in \cite{fomin2007cluster}) available in the generalized situation.

It turns out many algebraic and combinatorial features of cluster algebras are also inherited by generalized cluster algebras. The same finite type classification as for cluster algebras \cite{fomin2003cluster} and the generalized Laurent phenomenon have already been obtained in \cite{chekhov2011teichmller}. We show that the dependence on coefficients in the generalized case behaves very much like the ordinary case \cite{fomin2007cluster}. In particular, a generalized version of the separation formula, \Cref{thm: separation formula}, is made available through an analogous notion of principal coefficients. The well-known sign coherence of $c$-vectors (see \Cref{subsection: principal coefficients}) is also extended to the generalized case in \Cref{prop: generalized sign coherence}. We note that Nakanishi has a rather different version of normalized generalized cluster algebras with a certain reciprocal restriction in \cite{nakanishi2015structure} where some results on the structures of seeds parallel to \cite{fomin2007cluster} were also established.

Another remarkable feature of an ordinary cluster algebra is the positivity of its cluster variables, i.e. they are all Laurent polynomials in the initial variables $x_i$ and coefficients $p_i^\pm$ with non-negative integer coefficients. This was proved by Lee and Schiffler \cite{lee2015positivity} for skew-symmetric types and by Gross, Hacking, Keel, and Kontsevich \cite{gross2018canonical} for the more general skew-symmetrizable types. We extend the positivity to our generalized case (see \Cref{thm: positivity introduction}), showing that the Laurent expression of any cluster variable in $x_i$ and $p_{k,j}^\pm$ has non-negative integer coefficients. We note that the positivity obtained here is (in the reciprocal case) a weak form of a positivity conjecture of Chekhov and Shapiro (which we reformulated in \Cref{conj: chekhov shapiro}); see \Cref{rmk: a stronger positivity} and \Cref{subsection: further on positivity}.

\subsection{Generalized cluster varieties}
Let $L$ be a lattice of finite rank. Fix an algebraically closed field $\Bbbk$ of characteristic zero. The (ordinary) cluster varieties of Fock and Goncharov \cite{fock2009cluster} are schemes of the form
\[
V = \bigcup_{\mathbf s} T_{L,\mathbf s}
\]
where each $T_{L, \mathbf s}$ is a copy of the torus $L \otimes \Bbbk^*$ and they are glued together via birational maps called \emph{cluster mutations}. Here $\mathbf s$ runs over a set of \emph{seeds} (a seed roughly being a labeled basis of $L$) iteratively generated by mutations. A cluster mutation is give by the following birational map
\[
\mu_{(m,n)} \colon T_L \dashrightarrow T_L, \quad \mu_{(m,n)}^*(z^\ell) = z^\ell(1 + z^m)^{\langle \ell, n\rangle},\quad \ell\in L^*,
\]
for a pair of vectors $(n,m)\in L\times L^*$, where $\langle \cdot, \cdot \rangle$ denotes the natural paring between $L^*$ and $L$. It has a natural dual by switching the roles of $m$ and $n$, $\mu_{(-n, m)}\colon T_{L^*} \rightarrow T_{L^*}$. Glueing $T_{L^*}$ via these maps gives the \emph{dual cluster variety} $V^\vee \coloneqq \bigcup_\mathbf s T_{L^*, \mathbf s}$.

Depending on the types of seeds and mutations chosen, one obtains either Fock--Goncharov's $\mathcal A$-cluster varieties or $\mathcal X$-cluster varieties, which are dual constructions as above. A cluster algebra $\mathscr A$ can be embedded into the upper cluster algebra $\overline{\mathscr A}$, defined to be the algebra of regular functions on the corresponding $\mathcal A$-variety, while the dual $\mathcal X$-variety encodes the so-called $Y$-variables; see \Cref{section: cluster varieties}.

One can actually encode coefficients in each cluster mutation, the above construction thus leading to families of cluster varieties. They mutate along with seeds under certain rules. In the $\mathcal A$-case, they mutate as $Y$-variables (see \cite{fomin2007cluster} and \cite{fock2009cluster}). In the $\mathcal X$-case, the mutation rule of the coeffcients has been worked out by Bossinger, Fr{\'\i}as-Medina, Magee and N\'ajera Ch\'avez in \cite{bossinger2020toric}.

We extend these dynamics of coefficients to the generalized situation for both the $\mathcal A$- and $\mathcal X$-cases. We define a \emph{generalized cluster mutation} as
\[
\mu^*(z^\ell) = z^\ell \prod_{j=1}^r \left( t_j^- + t_j^+ z^m \right)^{\langle \ell, n\rangle},
\]  
which depends on some coefficients $t_j^\pm$ in $\Bbbk^*$; see \Cref{section: cluster varieties}. Thus an ordinary cluster mutation can be viewed as a specialization of a generalized one. Generalized cluster varieties are then defined by glueing tori via the generalized mutations. We obtain two families of generalized cluster varieties
\[
\pi_\mathcal A \colon \mathcal A \rightarrow \mathrm{Spec}(R),\quad \pi_\mathcal X \colon \mathcal X \rightarrow \mathrm{Spec}(R),
\]
where the coefficients vary in some torus $\mathrm{Spec}(R) = (\mathbb G_m)^d$.

One key observation of Gross, Hacking and Keel in \cite{gross2013birational} is that a cluster variety can be investigated through its toric models, and mutations between seeds are basically switching between neighboring toric models. A toric model is a construction of a log Calabi--Yau variety by blowing up a hypersurface on the toric boundary of some toric variety. In the cluster situation, the toric variety depends on the choice of a seed $\mathbf s$ which also tells us where on the toric boundary to blow up. The resulting log Calabi--Yau variety is shown in \cite{gross2013birational} (under certain nice conditions) to be isomorphic to the corresponding cluster variety up to codimension two subsets. We extend this description to the generalized case, for both $\mathcal A$- and $\mathcal X$-type varieties; see \Cref{thm: mutation of toric model} and \Cref{subsection: toric model for cluster}.

\subsection{Scattering diagrams}
Cluster scattering diagrams are the main technical tool in \cite{gross2018canonical}. They have their origin in \cite{kontsevich2006affine} and \cite{gross2011real} in mirror symmetry. Roughly speaking, in the cluster case, a scattering diagram is a collection of walls in a real vector space with attached \emph{wall-crossing functions} (some of them giving information on mutations). Similar to a cluster algebra which starts with one cluster with information to perform mutations in $n$ directions iteratively, its scattering diagram can be constructed by initially setting up $n$ incoming walls and letting them propagate, generating only outgoing walls.

To get a generalized cluster scattering diagram, we replace ordinary wall-crossings (which correspond to ordinary cluster mutations) on the initial incoming walls by the generalized ones of the form
\[
f = \prod_{j=1}^r (1 + t_jz^m)
\]
where the $t_j$ are treated as formal parameters. Given a seed $\mathbf s$ (in the generalized sense), the associated data of incoming walls uniquely determines a consistent scattering diagram $\mathfrak D_\mathbf s$, which we call \emph{the generalized cluster scattering diagram}.

We show that the behavior of $\mathfrak D_\mathbf s$ under mutations is analogous to that of the ordinary case, in a way it is canonically associated to a mutation equivalence class of seeds. This is called the \emph{mutation invariance} in \cite[Theorem 1.24]{gross2018canonical}. See \Cref{thm: mutation invariance} for the precise description of the following theorem.

\begin{theorem}[\Cref{thm: mutation invariance}]
	There is a piecewise linear operation $T_k$ such that $T_k(\mathfrak D_\mathbf s)$ is equivalent to $\mathfrak D_{\mu_k(\mathbf s)}$ where $\mu_k(\mathbf s)$ denotes the mutation in direction $k$ of the seed $\mathbf s$.
\end{theorem}

In analogy with the ordinary case, the mutation invariance leads to the \emph{cluster complex structure} of $\mathfrak D_\mathbf s$.

\begin{theorem}[{\Cref{thm: wall crossing on cluster chamber}}]
	There is the cluster cone complex $\Delta_\mathbf s^+$ such that $\mathfrak D_\mathbf s$ is a union of codimension one cones of $\Delta_\mathbf s^+$ (with explicit attached wall-crossing functions) and walls outside $\Delta_\mathbf s^+$.
\end{theorem}

We observe in \Cref{lemma: tropical vertex sd equals cluster sd} that $\mathfrak D_\mathbf s$ is equivalent to the \emph{tropical vertex scattering diagram} $\mathfrak D_{(X_\Sigma, H)}$ of Arg\"uz and Gross \cite{arguz2020higher} associated to the corresponding $\mathcal X$-type toric model associated to $\mathbf s$. It is shown in \cite[Theorem 6.1]{arguz2020higher} that $\mathfrak D_{(X_\Sigma, H)}$ is further equivalent (after a certain piecewise linear operation) to the \emph{canonical scattering diagram} $\mathfrak D_{(X,D)}$ (see \cite{gross2021canonical} and \cite[Section 2]{arguz2020higher}) of the log Calabi--Yau pair $(X,D)$ from the toric model. We thus see that $\mathfrak D_\mathbf s$ is canonically associated to the corresponding $\mathcal X$-cluster variety, with a different seed $\mathbf s'$ simply giving another representative $\mathfrak D_{\mathbf s'}$.

\subsection{Cluster dualities}
The cluster duality of Fock and Goncharov predicts that, in the ordinary case, the varieties $\mathcal A$ and $\mathcal X$ (see \Cref{section: cluster varieties} for our convention as we do not need the Langlands dual data) are dual in the sense that the upper cluster algebra $\overline{\mathscr A}$ has a basis parametrized by the tropical set $\mathcal X^\mathrm{trop}(\mathbb Z)$ (see \cite[Section 2]{gross2018canonical} for a definition) and vice versa. A modified version of this duality (and when it is true) is the main subject of study of \cite{gross2018canonical}.

The strategy of \cite{gross2018canonical} to get the desired basis is as follows. First the tropical set $\mathcal X^\mathrm{trop}(\mathbb Z)$ (resp. $\mathcal X^\mathrm{trop}(\mathbb R)$) can be identified with the cocharacter lattice $M$ (resp. $M_\mathbb R \coloneqq M\otimes \mathbb R$) of a chosen seed torus $T_{M,\mathbf s} = M \otimes \Bbbk^*$ contained in the variety $\mathcal X$. By the mutation invariance, the ordinary cluster scattering diagram $\mathfrak D_\mathbf s^\mathrm{ord}$ (see \Cref{subsection: ordinary cluster sd}) naturally lives in $\mathcal X^\mathrm{trop}(\mathbb R)$. Denote by $\Delta^+$ the set of integral points inside the cluster complex (which is again a canonical subset of $\mathcal X^\mathrm{trop}(\mathbb Z)$ by mutation invariance). 

For any integral point $m \in \mathcal X^\mathrm{trop}(\mathbb Z)$, using the scattering diagram $\mathfrak D_\mathbf s^\mathrm{ord}$, one can construct the \emph{theta function} $\vartheta_m$, which in general is only a formal power series in a completion $\widehat{\Bbbk[M]}_\mathbf s$ which depends on $\mathbf s$. However, it is shown in \cite[Theorem 4.9]{gross2018canonical} that for $m\in \Delta^+$, $\vartheta_m$ is indeed a Laurent polynomial in $\Bbbk[M]$ and corresponds to a cluster monomial. Furthermore, there is a canonically defined (i.e. independent of $\mathbf s$) subset $\Theta$ of $\mathcal X^\mathrm{trop}(\mathbb Z)$ containing $\Delta^+$ such that for any $m\in \Theta$, $\vartheta_m$ is a Laurent polynomial on every seed torus. It is also shown in \cite{gross2018canonical} that the vector space
\[
\mathrm{mid}(\mathcal A) \coloneqq \bigoplus_{m\in \Theta} \vartheta_m
\]
has an associative algebra structure whose structure constants are defined through \emph{broken lines}. This algebra $\mathrm{mid}(\mathcal A)$ can be embedded in $\overline{\mathscr A}$ so that for $m\in \Delta^+$, $\vartheta_m$ is sent to the corresponding cluster monomial. While we do not know in general when $\mathrm{mid}(\mathcal A)$ equals $\overline{\mathscr A}$ (see \cite[Theorem 0.3]{gross2018canonical}), we do have a basis of $\mathrm{mid}(\mathcal A)$ parametrized by the subset $\Theta$. Strictly speaking, this process is done through the principal coefficients case.

Our insight is that it is natural to consider the above cluster duality for generalized cluster varieties. In the principal coefficients case, we take a general fiber $\mathcal X^\mathrm{prin}_\lambda \coloneqq \pi_\mathcal X^{-1}(\lambda)$ of the family
\[
    \pi_\mathcal X \colon \mathcal X^\mathrm{prin} \rightarrow \mathrm{Spec}(R).
\]
The generalized cluster scattering diagram $\mathfrak D_\mathbf s$ then lives in the tropical set $(\mathcal X^\mathrm{prin}_{\lambda})^\mathrm{trop}(\mathbb R)$ which is identified with $M_\mathbb R$ given a chosen seed $\mathbf s$. Towards a generalized version of the cluster duality, we show

\begin{theorem}[{\Cref{thm: cluster variable as theta function}}]
	For any $m\in \Delta_\mathbf s^+$, the theta function $\vartheta_m$ constructed from the generalized cluster scattering diagram $\mathfrak D_\mathbf s$ corresponds to the cluster monomial of the generalized cluster algebra $\mathscr A^\mathrm{prin}(\mathbf s)$ whose $g$-vector is $m$. Moreover, it is a Laurent polynomial in the initial cluster variables $x_i$ and coefficients $p_{i,j}$ with non-negative integer coefficients.
\end{theorem}

It follows from the above theorem that the family
\[
    \pi_\mathcal A\colon \mathcal A^\mathrm{prin} \rightarrow \mathrm{Spec}(R)
\]
can be reconstructed from a general fibre $\mathcal X_\lambda^\mathrm{prin}$ (through any of its toric models); see \Cref{prop: reconstruct cluster variety}.

In principle, in the generalized case, one could consider the subset $\Theta$ as in \cite{gross2018canonical} and the corresponding generalized middle cluster algebra $\mathrm{mid}(\mathcal A^\mathrm{prin})$. This would lead to a formulation of generalized cluster duality similar to the ordinary case in \cite[Theorem 0.3]{gross2018canonical}. Then the usual problem on when the full Fock--Goncharov conjecture is true remains and can be discussed as in \cite[Section 8]{gross2018canonical}.

\subsection{Relations to other works}

There are examples of generalized cluster scattering diagrams from representation theory, where they are realized as Bridgeland's stability scattering diagrams \cite{bridgeland2016scattering} for quivers (with loops) with potentials; see \cite{labardini2019} of Labardini-Fragoso and the author for such examples arising from surfaces with orbifold points.

In rank two, the scattering diagram $\mathfrak D_{\mathbf s}$ has already appeared in \cite{gross2010tropical} and \cite{gross2010quivers} from origins other than cluster algebras. There the wall-crossing functions are shown to encode relative Gromov--Witten invariants on certain log Calabi--Yau surfaces. Some conjectural wall-crossing functions in \cite{gross2010quivers} are later verified in \cite{reineke2013refined} using techniques from quiver representations; see \Cref{ex: example for k2}.

The recent paper \cite{cheung2023cluster} of Cheung, Kelley and Musiker (outlined in \cite{cheungcluster}) and some part of Kelley's PhD thesis \cite{kelley2021structural} have significant overlaps with this paper and the author's PhD thesis \cite[Chapter 8]{mou2020wall}. We in the following highlight the differences and relationships concerning scattering diagrams.

In \cite[Chapter 8]{mou2020wall}, a class of generalized cluster scattering diagrams were constructed and properties including mutation invariance and cluster complex structure were proved. In that work, a palindromic and monic restriction (termed \emph{reciprocal} in \cite{chekhov2011teichmller}) on the coefficients was imposed. Such a scattering diagram can be obtained from applying to $\mathfrak D_{\mathbf s}$ of the current paper an evaluation $\lambda$ such that the initial wall-crossings are specialized to reciprocal polynomials, i.e., of the form
\[
    f = 1 + a_1 z^w + \cdots + a_{r-1} z^{(r-1)w} + z^{rw}
\]
where $r\in \mathbb Z_{\geq 0}$, $w\in M$, and $a_{k} = a_{r - k}$ in some ground field $\Bbbk$; see \Cref{subsection: sd with special coefficients}. Scattering diagrams amost identical to these (with the reciprocal restriction) were later considered by Cheung, Kelley and Musiker in the announcement \cite{cheungcluster}, with more details provided in \cite{kelley2021structural}. The authors treat the coefficients $a_i$ as formal variables. They also outlined the construction of theta functions, following \cite{gross2018canonical}.

The current paper aims to fill in gaps and missing details in \cite{mou2020wall}, enhance the setup therein to include more general coefficients, and discuss the positivity of generalized cluster algebras using scattering diagram techniques. Shortly after this paper was posted on the arXiv, \cite{cheung2023cluster} appeared on the arXiv, completing the program \cite{cheungcluster}. Despite many similarities between the current paper and \cite{cheung2023cluster}, our approaches of treating coefficients differ somewhat. In \cite{cheung2023cluster}, the $\mathbf y$-variables in a generalized seed and the coefficients $\mathbf a = (a_i)$ in a generalized exchange polynomial are treated separately. The coefficients $\mathbf a$ are assumed to be reciprocal and remain unchanged under mutations. In contrast, we view the coefficients $\mathbf a$ as deriving from the $\mathbf y$-variables (denoted as $\mathbf p$ in our notation) by factorizing an exchange polynomial into binomials, with each binomial governed by one coefficient in the style of Fomin and Zelevinsky. This approach allows us to realize more general exchange polynomials (beyond just reciprocal ones), at least for an algebraically closed ground field, by specialization from principal coefficients (see \Cref{subsection: specialized coefficients} and \Cref{subsection: further on positivity}). This setup also enables us to formulate and prove a form of positivity for generalized cluster algebras, a topic not extensively discussed in \cite{cheung2023cluster}.

\section*{Acknowledgements}

I would like to thank Daniel Labardini-Fragoso for showing me an example of generalized cluster scattering diagram in the beginning; Fang Li and Siyang Liu for helpful discussions and the hospitality during my visit at Zhejiang University. During the preparation of the manuscript, I was supported by the Royal Society through the Newton International Fellowship NIF\textbackslash R1\textbackslash 201849.

\section{Preliminaries}\label{section: preliminaries}

In this section, we review some preliminaries commonly used in the theory of cluster algebras \cite{fomin2002cluster}.

\subsection{Semifields}\label{subsection: semifield}

\begin{definition}\label{def: semifield}
	A semifield $(\mathbb P, \oplus, \cdot)$ is a torsion free (multiplicative) abelian group $\mathbb P$ with a binary operation \emph{addition} $\oplus$ which is commutative, associative and distributive.
\end{definition}

We denote by $\mathbb {ZP}$ the group ring of $\mathbb P$ and by $\mathbb {NP}\subset \mathbb {ZP}$ the subset of linear combinations with coefficients in $\mathbb N$. Denote by $\mathbb {QP}$ the field of fractions of $\mathbb {ZP}$.

For an element $p\in \mathbb P$, we define in $\mathbb P$ two elements:
\[
    p^+ \coloneqq \frac{p}{p\oplus 1}\quad \text{and} \quad p^- \coloneqq \frac{1}{p\oplus 1}.
\]

\begin{definition}\label{def: tropical semifield}
    Let $I$ be a finite set. We define $\mathrm{Trop}(s_i\mid i\in I)$ to be the (multiplicative) abelian group with free generators $s_i$ indexed by $I$, with the operation addition $\oplus$:
\[
    \prod_{i\in I}s_i^{a_i} \oplus \prod_{i\in I}s_i^{b_i} \coloneqq \prod_{i\in I}s_i^{\min\{a_i, b_i\}}.
\]
It is clear that $\mathrm{Trop}(s_i\mid i\in I)$ is a semifield. Such a semifield is called a \emph{tropical semifield}.
\end{definition}

For $n\in \mathbb Z$, we write $[n]_+ \coloneqq \max \{n, 0\}$. The elements $s^\pm$ for $$s = \prod\limits_{i\in I}s_i^{a_i}\in \mathrm{Trop}(s_i\mid i\in I)$$ has the following simple expressions:
\[
	s^+ = \prod_{i\in I} s_i^{[a_i]_+}\quad \text{and} \quad s^{-} = \prod_{i\in I} s_i^{[-a_i]_+}.
\]

\begin{definition}\label{def: universal semifield}
Denote by $\mathbb Q_{\mathrm{sf}}(u_1, \dots, u_l)$ the set of all rational functions in $l$ independent variables which can be written as subtraction-free rational expressions in $u_1, \dots, u_l$. Then the set $\mathbb Q_{\mathrm{sf}}(u_1, \dots, u_l)$ is a semifield with respect to the usual addition and multiplication.
\end{definition}

We call such $\mathbb Q_{\mathrm{sf}}(u_1, \dots, u_l)$ a \emph{universal semifield} since for another arbitrary semifield $\mathbb P$ and its elements $p_1, \dots, p_l$, there exists a unique map of semifields from $\mathbb Q_{\mathrm{sf}}(u_1, \dots, u_l)$ to $\mathbb P$ sending $u_i$ to $p_i$; see \cite[Lemma 2.1.6]{berenstein1996parametrizations}.

\subsection{Mutations of matrices}

\begin{definition}\label{def: skew-sym}
	A matrix $B\in \mathrm{Mat}_{n\times n}(\mathbb Z)$ is called (left) \emph{skew-symmetrizable} if there exists a diagonal matrix $D = \mathrm{diag}(d_i\mid 1\leq i \leq n)$ with $d_i\in \mathbb Z_{>0}$ such that
	\[
		D B + (D B)^T = 0.
	\]
	Such a matrix $D$ is called a (left) \emph{skew-symmetrizer} of $B$.
\end{definition}

\begin{definition}[{\cite[Definition 4.2]{fomin2002cluster}}]
	Let $B = (b_{ij})\in \mathrm{Mat}_{n \times n}(\mathbb Z)$ be a skew-symmetrizable matrix. For $k = 1, \dots, n$, we define $\mu_k(B) = (b_{ij}')\in \mathrm{Mat}_{n \times n}(\mathbb Z)$ the \emph{mutation of $B$ in direction $k$} by setting
	\begin{enumerate}
		\item $b'_{ik} = -b_{ik}$ and $b'_{kj} = -b_{kj}$ for $1\leq i, j \leq n$;
		\item For $i\neq k$ and $j\neq k$, \[ b'_{ij} = \begin{cases}
			b_{ij} + b_{ik} b_{kj} \quad &\text{if $b_{ik}>0$ and $b_{jk}<0$};\\
			b_{ij} - b_{ik} b_{kj} \quad &\text{if $b_{ik}<0$ and $b_{jk}>0$};\\
			b_{ij} \quad &\text{otherwise.}
		\end{cases} \] 
	\end{enumerate}
\end{definition}

It is clear that the matrix $\mu_k(B)$ is again skew-symmetrizable with the same set of skew-symmetrizers of $B$. One can easily check that a mutation is involutive in the same direction, i.e. $\mu_k\circ \mu_k (B) = B$.

\section{Generalized cluster algebras}

\subsection{Generalized cluster algebras}
Cluster algebras were originally invented by Fomin and Zelevinsky in \cite{fomin2002cluster}, which we later refer to as \emph{ordinary cluster algebras}. A generalization of cluster algebras has been provided by Chekhov and Shapiro in \cite{chekhov2011teichmller}. Our definition of generalized cluster algebras below may be considered as a special case (of a slight generalization) of theirs but with a normalization analogous to the one in \cite[Definition 5.3]{fomin2002cluster} for ordinary cluster algebras. The relation and difference will be explained in \Cref{subsection: chekhov-shapiro}. 

We follow the pattern of \cite{fomin2007cluster} to define generalized cluster algebras. Most of the key notions here are the generalized versions of their correspondents in the ordinary case.

\begin{definition}\label{def: seed}
	A \emph{(generalized) labeled seed} $\Sigma$ of rank $n\in \mathbb N$ is a triple $(\mathbf x, \mathbf p, B)$, where
	\begin{enumerate}
	\item[$\bullet$] $\mathbf p = (\mathbf p_1, \dots, \mathbf p_n)$ is an $n$-tuple of collections of elements, where each $\mathbf p_i = \left ( p_{i,1},\dots, p_{i,r_i} \right)$ is a $r_i$-tuple of elements in a semifield $(\mathbb P, \oplus, \cdot)$ for some positive integer $r_i$.
	
	\item[$\bullet$] $\mathbf x = \{ x_1, \dots, x_n\}$ is a collection of algebraically independent rational functions of $n$ variables over $\mathbb {QP}$. In other words, the $x_1, \dots x_n$ are elements in some field of rational functions $\mathcal F = \mathbb {QP}(u_1, \dots, u_n)$ such that $\mathcal F = \mathbb {QP}(x_1, \dots, x_n)$.
	
	\item[$\bullet$] $B\in \mathrm{Mat}_{n\times n}(\mathbb Z)$ is skew-symmetrizable such that for any $i = 1, \dots, n$, its $i$-th column is divisible by $r_i$. The diagonal matrix $D = \mathrm{diag}(r_i)$ is not necessarily a skew-symmetrizer of $B$.
	\end{enumerate}
\end{definition}

For convenience, let $I$ be the index set $\{1, \dots, n\}$. For an arbitrary positive integer $k$, we use the interval $[1,k]$ to represent the set $\{1,\dots, k\}$. We will often call a labeled seed simply a \emph{seed} if there is no confusion.

Associated to a labeled seed $\Sigma = (\mathbf x, \mathbf p, B)$, for each $i\in I$, there is the \emph{exchange polynomial}
\[
\theta_i(u,v) = \theta[\mathbf p_i](u,v) \coloneqq \prod_{l = 1}^{r_i} \left( p_{i,l}^+ u + p_{i,l}^- v \right)\in \mathbb {ZP}[u,v].
\]
Write $\beta_{ij} = b_{ij}/r_j\in \mathbb Z$. We put
\[
u_{j;+} \coloneqq \prod_{i\colon b_{ij}>0}x_i^{\beta_{ij}},\quad u_{j;-} \coloneqq \prod_{i\colon b_{ij}<0} x_i^{-\beta_{ij}}
\]
and
\[
p_{i;+} \coloneqq \prod_{l = 1}^{r_i} p_{i,l}^+,\quad p_{i;-} \coloneqq \prod_{l = 1}^{r_i} p_{i,l}^- \in \mathbb P.
\]
Note that all the above notions are with respect to $\Sigma$.

\begin{definition}\label{def: mutation of seed}
	For any $k\in I$, we define the \emph{mutation of a seed $\Sigma = (\mathbf x, \mathbf p, B)$ in direction $k$} as a new labeled seed $\mu_k(\mathbf x, \mathbf p, B) \coloneqq ((x_i'), (\mathbf p_i'), B')$ where $\mathbf p_i' = \left(p_{i,j}'\mid j\in [1,r_i]\right)$ in the following way:
	\begin{enumerate}
		\item 
	\begin{equation*}
		B' = \mu_k(B);
	\end{equation*}
		\item
	$$p_{k,j}' = p_{k,j}^{-1}\quad  \text{for $j \in [1,r_k]$};$$
	\item 
\begin{equation*}
	\text{for $i\neq k$, $j\in [1, r_i]$}\quad p'_{i,j} = \begin{dcases}
	p_{i,j}\cdot \left( p_{k;-}\right)^{\beta_{ki}}\quad &\text{if $\beta_{ik}>0$}\\
	p_{i,j}\cdot \left( p_{k;+}\right)^{\beta_{ki}} \quad &\text{if $\beta_{ik}\leq 0$},
	\end{dcases}
\end{equation*}
or equivalently
\begin{equation*}
	\text{for $i\neq k$,}\quad p'_{i,j} = 
	p_{i,j}\left( \prod_{l=1}^{r_k} \left( 1 \oplus p_{k,l}^{\mathrm{sgn}(\beta_{ik})} \right) \right)^{-\beta_{ki}};	
	\end{equation*}
	
	\item
	\begin{equation*}
		x_i' = \begin{dcases}
			x_i \quad &\text{if $i\neq k$}\\
			x_k^{-1}\theta[\mathbf p_k](u_{k;+}, u_{k;-}) \quad &\text{if $i = k$}.
		\end{dcases}
	\end{equation*}
		\end{enumerate}
\end{definition}

\begin{lemma}\label{lemma: mutation is involutive}
The mutation $\mu_k$ is involutive, i.e.	 $\mu_k\circ\mu_k(\Sigma) = \Sigma$.
\end{lemma}

\begin{proof}
We check that $\mu_k$ is involutive on each component of a seed. We denote $$\mu_k\circ\mu_k(\Sigma) = \left((x_i''), (p_{i,j}''\mid j \in [1, r_i])_{i\in I}, B''\right).$$
For this seed, we simply denote the relevant objects appearing in \Cref{def: mutation of seed} by adding a double prime to the old symbols, while for $\mu_k(\Sigma)$, we add a single prime.

	\begin{enumerate}
		\item First of all, the matrix mutation $\mu_k$ is an involution, which was already shown in \cite{fomin2002cluster}.
		\item We have for $j\in [1, r_k]$, $$p''_{k,j} = \left(p'_{k,j}\right)^{-1} = p_{k,j}.$$
		\item For $i\neq k$, we have for $j\in [1, r_i]$, \begin{align*}
			\quad p''_{i,j} &= \begin{dcases}
	p_{i,j}'\cdot \left( p_{k;-}'\right)^{\beta_{ki}'}\quad &\text{if $\beta_{ik}'>0$}\\
	p_{i,j}'\cdot \left( p_{k;+}'\right)^{\beta_{ki}'} \quad &\text{if $\beta_{ik}'\leq 0$}
	\end{dcases}\\
			&= \begin{dcases}
	p_{i,j}\cdot \left( p_{k;+}\right)^{\beta_{ki}}\cdot \left(p_{k;-}' \right)^{-\beta_{ki}}\quad &\text{if $\beta_{ik}<0$}\\
	p_{i,j}\cdot \left( p_{k;-}\right)^{\beta_{ki}}\cdot \left(p_{k;+}' \right)^{-\beta_{ki}} \quad &\text{if $\beta_{ik}\geq 0$}
	\end{dcases}\\
			&= p_{i,j}.
		\end{align*}
		The last equality is because $p'_{k;+} = p_{k;-}$ and $p'_{k;-} = p_{k;+}$.
	\item Finally for the $\mathbf x$ part, we have \begin{align*}
		x_i'' & = \begin{cases}
			x_i' \quad &\text{if $i\neq k$}\\
			(x_k')^{-1}\theta[\mathbf p_k']\left(u'_{k;+},u'_{k;-}\right)\quad &\text{if $i = k$}
		\end{cases}\\
		& = \begin{cases}
			x_i \quad &\text{if $i\neq k$}\\
			x_k\cdot \theta[\mathbf p_k]\left(u_{k;+}, u_{k;-}\right)^{-1}\theta[\mathbf p_k']\left(u'_{k;+},u'_{k;-}\right)\quad &\text{if $i = k$}
		\end{cases}\\
		& = x_i.
	\end{align*}
	The last equality is because that $\theta[\mathbf p_k'](u,v) = \theta[\mathbf p_k](v,u)$ and $u_{k;\pm}' = u_{k;\mp}$.	
	\end{enumerate}
	So overall we have proven that $\mu_k\circ \mu_k(\Sigma) = \Sigma$.
\end{proof}

Fix a positive integer $n$. We consider the (non-oriented) \emph{$n$-regular tree} $\mathbb T_n$ whose edges are labeled by the numbers $1,\dots,n$ as in \cite{fomin2002cluster}. \Cref{lemma: mutation is involutive} makes the following definition well-defined.

\begin{definition}
	A \emph{(generalized) cluster pattern} is an assignment of a labeled seed $\Sigma_t = \left ( \mathbf x_t, \mathbf p_t, B^t\right)$ to every vertex $t\in \mathbb T_n$, such that for any $k$-labeled edge with endpoints $t$ and $t'$, the seed $\Sigma_{t'}$ is the mutation in direction $k$ of $\Sigma_t$, i.e. $\Sigma_{t'} = \mu_k(\Sigma_t)$.
	\end{definition}

The elements in $\Sigma_t$ are written as follows:
	\[
	\mathbf x_t = (x_{i;t}\mid i\in I),\quad \mathbf p_{i;t} = \left (p_{i,j;t}\mid j\in [1, r_i]\right), \quad B^t = (b_{ij}^t).
	\]
The part $\mathbf x$ of a labeled seed is called a (labeled) \emph{cluster}, elements $x_i$ are called \emph{cluster variables}, elements $p_{i,j}$ are called \emph{coefficients} and $B$ is called \emph{exchange matrix}.

Two seeds are \emph{mutation-equivalent} if one is obtained from the other by a sequence of mutations.  If a seed $\Sigma$ appears in a cluster pattern, then by definition any seeds mutation-equivalent to $\Sigma$ must also appear. On the other hand, assigning a seed of rank $n$ to any vertex of $\mathbb T_n$ would uniquely determine a cluster pattern.

By definition, all cluster variables live in some ambient field $\mathcal F$ of rational functions of $n$ variables. One may identify $\mathcal F$ with $\mathbb {QP}(x_1, \cdots, x_n)$ where $(x_1,\dots, x_n)$ is a cluster.

\begin{definition}
	Given a generalized cluster pattern, the \emph{(generalized) cluster algebra} $\mathscr A$ is defined to be the $\mathbb {ZP}$-subalgebra of the ambient field $\mathcal F$ generated by all cluster variables $x_{i;t}$ in all seeds that appear in the cluster pattern. Since a cluster pattern is determined by any of its seed, we denote $\mathscr A = \mathscr A(\Sigma)$ where $\Sigma = (\mathbf x, \mathbf p, B)$ is any seed in this cluster pattern.
\end{definition}

\begin{remark}
	In the case where $r_i = 1$ for any $i\in I$, all the above generalized notions recover the original versions of Fomin and Zelevinsky \cite{fomin2007cluster}.
\end{remark}

For convenience, one can specify one vertex $t_0\in \mathbb T_n$ to be \emph{initial}, thus the associated seed being called the \emph{initial seed} with the \emph{initial} cluster, cluster variables, coefficients and exchange matrix. All other seeds are obtained by applying mutations iteratively to the initial one. For the following two theorems, we denote by $(x_1, \dots, x_n)$ the initial cluster. 

\begin{theorem}[Generalized Laurent phenomenon, cf. \cite{fomin2002cluster} and \cite{chekhov2011teichmller}]\label{thm: generalized laurent}
 Let $(\mathbf x, \mathbf p, B)$ be a generalized labeled seed. Then any cluster variable of $\mathscr A(\mathbf x, \mathbf p, B)$ is a Laurent polynomial over $\mathbb {ZP}$ in the initial cluster variables $x_i$, i.e. an element in $\mathbb {ZP}[x_1^\pm, \dots, x_n^\pm]$.
\end{theorem}

\begin{proof}
	The generalized Laurent phenomenon was already obtained in \cite[Theorem 2.5]{chekhov2011teichmller}. Since our setting is slightly different, we give a proof for completeness.
	
	The proof completely follows from the discussion in \cite[Section 3]{fomin2002cluster}. The generalized Laurent property is a direct corollary of \cite[Theorem 3.2]{fomin2002cluster}. One only needs to check the following hypothesis required by \cite[Theorem 3.2]{fomin2002cluster}: 
for any subgraph
	\[
		t_0 \overset{i}{\text{ ------ }} t_1 \overset{j}{\text{ ------ }} t_2 \overset{i}{\text{ ------ }} t_3
	\]
	in the tree $\mathbb T_n$, if we define the following three \emph{exchange polynomial} in $n$ variables $x_1,\dots, x_n$ by writing
	\[
	P(\mathbf x_{t_0}) = \theta[\mathbf p_{i;t_0}](u_{i;+}^{t_0}, u_{i;-}^{t_0}),\quad Q(\mathbf x_{t_1}) = \theta[\mathbf p_{j;t_1}](u_{j;+}^{t_1}, u_{j;-}^{t_1}), \quad R(\mathbf x_{t_2}) = \theta[\mathbf p_{i;t_2}](u_{i;+}^{t_2}, u_{i;-}^{t_2}),
	\]
	then they satisfy $R = C \cdot \left( P \vert_{x_j\leftarrow Q_0/x_j} \right)$ where $Q_0 = Q\vert_{x_i = 0}$ for some $C\in \mathbb {NP}[x_1, \dots x_n]$.	
	
	Notice that since $t_0 \frac{i}{\quad\quad} t_1$, we have
	\[
	P = \prod_{l=1}^{r_i} \left( p_{i,l;t_1}^+ \prod_{k} x_k^{[\beta_{ki}^{t_1}]_+} + p_{i,l;t_1}^- \prod_{k} x_k^{[-\beta_{ki}^{t_1}]_+} \right).
	\]
	
	When $\beta_{ij}^{t_1} = 0$, $\beta_{ji}^{t_0} = -\beta_{ij}^{t_1} = 0$. So $x_j$ does not appear in $P$, implying $P = P \vert_{x_j\leftarrow Q_0/x_j}$. In this case, we have for any $l\in [1, r_i]$
	\[
	p_{i,l;t_0} = p_{i,l;t_2}^{-1},\quad \beta_{li}^{t_0} = -\beta_{li}^{t_2}.
	\]
	Thus we have $R = P$.
	
	When $\beta_{ij}^{t_1} < 0$ (implying $\beta_{ji}^{t_1} >0$), then
	\[
	Q_0/x_j = p_{j;+;t_1} x_j^{-1} \prod_{k} x_k^{[b_{kj}^{t_1}]_+}
	\]
	and
	\[
		P \vert_{x_j\leftarrow Q_0/x_j} = 
				\prod_{l=1}^{r_i} \left(  p_{i,l;t_1}^+ p_{j;+;t_1}^{\beta_{ji}^{t_1}} x_j^{-\beta_{ji}^{t_1}}\prod_{k\neq j} x_k^{[\beta_{ki}^{t_1}]_+ + \beta_{ji}^{t_1} \cdot [b_{kj}^{t_1}]_+} +p_{i,l;t_1}^- \prod_{k} x_k^{[-\beta_{ki}^{t_1}]_+}  \right)
	\]
	We take the ratio of the two monomials in each factor of the above product to obtain monomials
	\[
	p_{i,l;t_1}\cdot p_{j;+;t_1}^{\beta_{ji}^{t_1}}\cdot \prod_{k} x_k^{\beta_{ki}^{t_2}}.
	\]
	We get exactly the same monomials if we do the same for $R$. So $R$ and $P \vert_{x_j\leftarrow Q_0/x_j}$ only differ by a monomial factor in $\mathbb {NP}[x_1, \dots, x_n]$. The case when $\beta_{ij}^{t_1}>0$ is analogous. Hence the hypothesis is checked and the Laurent phenomenon follows from \cite[Theorem 3.2]{fomin2002cluster}.
\end{proof}

The following \Cref{thm: positivity introduction} is a generalization of the well-known positivity for ordinary cluster algebras. In the case of ordinary cluster algebras, the positivity was originally conjectured by Fomin and Zelevinsky \cite{fomin2002cluster}. It has been proved by Lee and Schiffler \cite{lee2015positivity} when the exchange matrix $B$ is skew-symmetric and by Gross, Hacking, Keel, and Kontsevich \cite{gross2018canonical} when $B$ is more generally skew-symmetrizable.

\begin{theorem}[Positivity]\label{thm: positivity introduction}
	In a generalized cluster algebra, each of the coefficients in the Laurent polynomial corresponding to any cluster variable from \Cref{thm: generalized laurent} is a non-negative integer linear combination of elements in $\mathbb P$. In other words, any cluster variable is an element in $\mathbb {NP}[x_1^\pm, \dots, x_n^\pm ]$.
\end{theorem}

\begin{proof}
	By the separation formula \Cref{thm: separation formula} and \Cref{rmk: positivity for general coefficients}, we only need to show the positivity in the principal coefficients case (to be defined in \Cref{def: principal coefficients}). In this case, we prove the positivity in \Cref{thm: cluster variable as theta function}. 
\end{proof}

\begin{remark}\label{rmk: a stronger positivity}
	Chekhov and Shapiro conjectured \cite[Conjecture 5.1]{chekhov2011teichmller} a positivity for generalized cluster algebras under a reciprocal condition; see also the formulation in \Cref{conj: chekhov shapiro}. In the reciprocal case, this positivity implies \Cref{thm: positivity introduction}. We do not know how to show this stronger positivity in general; see the discussion in \Cref{subsection: further on positivity}.
\end{remark}

\subsection{Relation to Chekhov--Shapiro's definition}\label{subsection: chekhov-shapiro}

In \cite{chekhov2011teichmller}, Chekhov and Shapiro defined generalized cluster algebras by considering more general exchange polynomials. Precisely, a labeled seed in that setting is a triple $(\mathbf x, (\overline {\alpha_i} \mid i\in I), B)$ where $\mathbf x$ and $B$ are the same as in \Cref{def: seed} and $\overline \alpha_i = (\alpha_{i,k}\in \mathbb {P} \mid 0\leq k \leq r_i)$ for $i\in I$. Here we only take $\mathbb P$ as a multiplicative abelian group. The coefficients $\alpha_{i,k}$ are responsible for expressing the exchange polynomial defined as $$\theta_i(u,v)  \coloneqq \sum_{k = 0}^{r_i} \alpha_{i,k}u^{r_i-k}v^{k}\in \mathbb{ZP}[u,v].$$

The mutation $(\mathbf x', (\overline \alpha_i'), B') = \mu_k(\mathbf x, (\overline {\alpha_i}), B)$ is defined in the following way.
\begin{enumerate}
	\item $B' = \mu_k(B)$;
	\item $\alpha_{k,j}' = \alpha_{k, r_k-j}$ and for $i\neq k$, the coefficients satisfy
	\[
		\alpha_{i,j}'/\alpha_{i,0}' = \begin{cases}
			\alpha_{k,0}^{j\beta_{ki}} \cdot \alpha_{i,j}/\alpha_{i,0} \quad \text{if $\beta_{ik}>0$}\\
			\alpha_{k,r_k}^{j\beta_{ki}} \cdot \alpha_{i,j}/\alpha_{i,0} \quad \text{if $\beta_{ik}\leq 0$};
		\end{cases}
	\]
	\item $x_i' = x_i$ for $i\neq k$ and
	\[
	x_k'x_k = \theta_i\left( u_{k;+}, u_{k;-} \right).
	\]
\end{enumerate}

\begin{remark}
	In this setting, it does no harm to allow the coefficients $\alpha_{i,k}$ to be elements of $\mathbb {ZP}$, as long as the above rule (2) is satisfied. For example, one may check that the Laurent phenomenon still holds for cluster variables. 
\end{remark}

Now assume the multiplicative abelian group $\mathbb P$ has an addition $\oplus$ so that it is a semifield. In our setting the exchange polynomials are given by $\theta[\mathbf p_i](u,v)$, thus the coefficients $\alpha_{i,j}$ corresponding to the coefficients of $\theta[\mathbf p_i](u,v)$ when expanded as polynomial of $u$ and $v$. Under \Cref{def: mutation of seed}, the polynomials $\theta[\mathbf p_i](u,v)$ mutate in a way satisfying the rule (2) above. In fact, when talking about coefficients $\alpha_{i,j}/\alpha_{i,0}$, we can normalize our polynomial
\[
\tilde \theta[\mathbf p_i](u,v) = \prod_{j\in [1, r_i]}(p_{i,j}u + v).
\]
So when expanded as a sum of monomials in $u$ and $v$, the coefficients of $\tilde \theta[\mathbf p_i]$ are $\prod_{j\in J} p_{i,j}$ for a subset $J\subset [1,r_i]$. According to the mutation formula in \Cref{def: mutation of seed}, under $\mu_k$, we have
\[
\prod_{j\in J}p'_{i,j} = p_{k;\pm}^{|J|\beta_{ki}} \prod_{j\in J} p_{i,j},
\]
which satisfies the rule $(2)$. So our definition of generalized cluster algebras can be viewed as a special case of \cite{chekhov2011teichmller} if we ease the condition $\alpha_{i,k}\in \mathbb P$ into $\alpha_{i,k}\in \mathbb {ZP}$.

We note that the above rule (2) in \cite{chekhov2011teichmller} is not enough to uniquely determine the coefficients $(\overline \alpha_i')$ after mutation, whereas the mutation formula in \Cref{def: mutation of seed} is deterministic because of the normalization condition $p_{i,j}^+ \oplus p_{i,j}^- = 1$.

One advantage of our definition is the availability of principal coefficients analogous to \cite[Definition 3.1]{fomin2007cluster}, to be discussed in the next section.

\subsection{Principal coefficients}\label{subsection: principal coefficients}
 As in \cite{fomin2007cluster} for ordinary cluster algebras, we have the notion of principal coefficients for generalized cluster algebras.

\begin{definition}\label{def: geometric type}
	We say a generalized cluster algebra $\mathscr A$ is of geometric type if $\mathbb P$ is a tropical semifield $$\mathrm{Trop}(s_a\mid a\in I')$$ where $I'$ is a finite index set. 
\end{definition}

\begin{proposition}\label{prop: matrix of coefficient mutation}
	Let $\mathscr A$ be a generalized cluster algebra of geometric type. For each seed $\Sigma$ in $\mathscr A$'s cluster pattern and $i\in I$, we introduce matrices $$C^{(i)} = C^{(i)}_\Sigma =  \left( c^{(i)}_{aj} \right) \in \mathrm{Mat}_{|I'|\times r_i} (\mathbb Z)$$ to encode the coefficients $p_{i,j}$ by columns of $C^{(i)}$:
	\[
	p_{i,j} = \prod_{a\in I'} s_a^{c_{aj}^{(i)}}\in \mathbb P.
	\]
	Denote by $\left(\bar c^{(i)}_{aj}\right)$ the matrices corresponding to the seed $\mu_k(\Sigma)$ for some $k\in I$. Then we have

	\[
	\bar c^{(i)}_{aj} = \begin{dcases}
		-c^{(i)}_{aj} \quad &\text{if $i = k$};\\
		c^{(i)}_{aj} + \left ( \sum_{j = 1}^{r_k}\left[-c^{(k)}_{aj}\right]_+\right)\beta_{ki} &\text{if $i\neq k$ and $\beta_{ik}>0$};\\
		c^{(i)}_{aj} + \left ( \sum_{j = 1}^{r_k}\left[c^{(k)}_{aj}\right]_+\right)\beta_{ki} &\text{if $i\neq k$ and $\beta_{ik}\leq 0$}.
	\end{dcases}
	\]
\end{proposition}

\begin{proof}
In the tropical semifield $\mathrm{Trop}(s_a\mid a\in I')$, we have the following expressions:
    \[
	p_{i,j}^+ = \prod_{a\in I'} s_a^{\left[c^{(i)}_{aj}\right]_+}\quad \text{and} \quad p_{i,j}^- = \prod_{a\in I'} s_a^{\left[-c^{(i)}_{aj}\right]_+}.
	\]
Then the result follows by spelling out the mutation formula of coefficients ((3) of \Cref{def: mutation of seed}).
\end{proof}

The matrices and their dynamics in \Cref{prop: matrix of coefficient mutation} have led to a remarkable combinatorial phenomenon in cluster theory known as \emph{the sign coherence of $c$-vectors}. We shall explain it below.

\begin{definition}\label{def: principal coefficients}
	We say a generalized cluster algebra $\mathscr A$ has \emph{principal coefficients} at a vertex $t_0\in \mathbb T_n$ if $\mathbb P$ is the tropical semifield
	\[
	\mathrm{Trop}(\mathbf p) \coloneqq \mathrm{Trop} \left( p_{i,j}\mid i\in I, j \in [1, r_i]  \right),
	\]
	and the seed $\Sigma_{t_0}$ has coefficients $\mathbf p_i = (p_{i,1}, \dots p_{i,r_i})$. In this case, the cluster algebra, denoted as $\mathscr A^\mathrm{prin}(\Sigma_{t_0})$, is by definition a subalgebra of $$\mathbb Z[x_{i;t_0}^\pm, p_{i,j}^\pm \mid i\in I, j\in [1, r_i]].$$
\end{definition}

In the case of principal coefficients, the columns of the matrices $C^{(i)}_{\Sigma_t}$ are called \emph{generalized $c$-vectors}. At the initial seed $\Sigma = \Sigma_{t_0}$ with principal coefficients, each $C^{(i)}_\Sigma$ is an identity matrix $I_{r_i}$ extended (vertically) by zeros.

\begin{theorem}[Sign coherence of generalized $c$-vectors]\label{thm: generalized sign coherence}
	In the principal coefficients case, for any $t\in \mathbb T_n$, for any $i\in I$ and any $j\in [1, r_i]$, the entries of the $j$-th column of $C^{(i)}_{\Sigma_t}$ are either all non-negative or all non-positive. 
\end{theorem}

When $r_i = 1$ for each $i\in I$, i.e. in the case of ordinary cluster algebras, each $C^{(i)} = C^{(i)}_{\Sigma_t}$ is just a column vector with $n$ entries, altogether forming a matrix $C = (C^{(1)}, \dots, C^{(n)})$. They are the so-called \emph{$C$-matrices} in \cite{fomin2007cluster} whose columns are \emph{$c$-vectors}. In this case, \Cref{thm: generalized sign coherence} then says that each column of any $C$ is either non-negative or non-positive. This is well-known in the theory of cluster algebras as \emph{the sign coherence of $c$-vectors}, which has already been proved by Derksen, Weyman and Zelevinsky \cite{derksen2010quivers} for skew-symmetric exchange matrices and by Gross, Hacking, Keel and Kontsevich \cite{gross2018canonical} for skew-symmetrizable ones. We will see in \Cref{prop: generalized sign coherence} that \Cref{thm: generalized sign coherence} follows from the already established sign coherence of $c$-vectors.

We set the index set
\[
I' = \bigsqcup_{i\in I} I'_i,\quad I_i'\coloneqq \{(i,j)\mid j\in [1, r_i]\}.
\]

\begin{lemma}\label{lemma: property generalized c matrix} 
Let $\Sigma = \Sigma_{t_0}$ be a seed with principal coefficients. We have the following properties for the matrices $C^{(i)}_{\Sigma_t}$ for any seed $\Sigma_t$, $t\in \mathbb T_n$.

	\begin{enumerate}
		\item Let $i, k\in I$ such that $k\neq i$. Then for any $a, a'\in I'_k$ and any $1\leq j,j'\leq r_i$, we have
	\[
	c_{a,j}^{(i)} = c_{a',j'}^{(i)}.
	\]
		\item Let $i\in I$. We have
	\[
	c_{(i,1),1}^{(i)} = c_{(i,2),2}^{(i)} = \cdots = c_{(i,r_i),r_i}^{(i)} = c\pm 1
	\]
	and
	\[
	c_{(i,k),j}^{(i)} = c \quad \text{for $k\neq j$}
	\]
	for some integer $c$.
	\end{enumerate}
\end{lemma}

\begin{proof}
	We prove this lemma by induction on the distance from $t$ to $t_0$ in $\mathbb T_n$. The base case is for $C_\Sigma^{(i)}$ where the entries in $(1)$ are all zeroes and the ones in $(2)$ are $1$ when $k = j$ and $0$ otherwise. Then the properties stated in the lemma are preserved under the mutation formula given in \Cref{prop: matrix of coefficient mutation}.
\end{proof}

Let $\overline {\mathbb P}$ be the tropical semifield $\mathrm{Trop}(\bar p_i\mid i\in I)$. There is a group homomorphism
\[
\pi\colon \mathbb P \rightarrow \overline{\mathbb P}, \quad p_{i,j}\mapsto \bar p_i.
\]
For $t\in \mathbb T_n$, denote the image of $p_{i,j;t}$ in $\overline {\mathbb P}$ by $\bar p_{i;t}$ (which is independent of $j$ by \Cref{lemma: property generalized c matrix}) and the image of $\prod\limits_{j=1}^{r_i}p_{i,j;t}$ in $\overline {\mathbb P}$ by $ p_{i;t} =  \bar p_{i;t}^{r_i}$. 

\begin{lemma}\label{lemma: mutation of p_i}
	The elements $p_{i;t}$ behave in the following way under the mutation $\mu_k$. If $t' \frac{k}{\quad\quad} t$ and we write $p'_i = p_{i;t'}$ and $p_i = p_{i;t}$, then we have
	\[
	p'_i = \begin{cases}
		p_i^{-1} \quad & \text{if $i=k$};\\
		p_i \cdot (p_{k}^-)^{b_{ki}} \quad &\text{if $i\neq k$ and $\beta_{ik}>0$};\\
		p_i \cdot (p_k^+) ^{b_{ki}} \quad &\text{if $i\neq k$ and $\beta_{ik} \leq 0$}.
	\end{cases}
	\]
	So they behave under mutations in the same way as $p_{i,1;t}$ in the case where $r_i=1, i\in I$, i.e. the case of ordinary cluster algebras.
\end{lemma}

\begin{proof}
	By the generalized mutation formula of coefficients, we have
	\[
	\prod_{j=1}^{r_i}p_{i,j}' = \begin{cases}
		\prod\limits_{j=1}^{r_i}p_{i,j}^{-1} \quad & \text{if $i=k$};\\
		\prod\limits_{j=1}^{r_i}p_{i,j} \cdot \left (\prod\limits_{j=1}^{r_k} p_{k,j}^- \right)^{b_{ki}} \quad &\text{if $i\neq k$ and $\beta_{ik}>0$};\\
		\prod\limits_{j=1}^{r_i}p_{i,j} \cdot \left ( \prod\limits_{j=1}^{r_k} p_{k,j}^+ \right)^{b_{ki}} \quad &\text{if $i\neq k$ and $\beta_{ik} \leq 0$}.
	\end{cases}
	\]
	By the matrix description of the elements $p_{k,j}$ in \Cref{lemma: property generalized c matrix}, we have that
	\[
	\prod\limits_{j=1}^{r_k} p_{k,j}^\pm  =  \left(\prod\limits_{j=1}^{r_k} p_{k,j} \right)^\pm \in \mathbb P,\quad 
	\pi \left( \prod\limits_{j=1}^{r_k} p_{k,j} ^\pm \right) = p_k^\pm \in \overline{\mathbb P}.
	\] 
	The result then follows.
\end{proof}

\begin{proposition}\label{prop: generalized sign coherence}
	The sign coherence of $c$-vectors implies the sign coherence of generalized $c$-vectors.
\end{proposition}

\begin{proof}
	In the case where all $r_i = 1$, the sign coherence then says each column of the matrix $C = (C^{(1)}, \dots, C^{(n)})$ is either non-negative or non-positive.
	
	On the other hand, by \Cref{lemma: mutation of p_i}, the elements $p_i$ behave under mutations in the exact same way as the coefficients in seeds when all $r_i = 1$ (thus we only have one $p_i$ for each $i$). Thus the column $C^{(i)}$ of $\Sigma_t$ serves as the coordinates of $p_{i;t}$ in terms of the initial coefficients $p_i$. Then the sign coherence tells that one of $p_i^+$ and $p_i^-$ is $1$. It then follows from \Cref{lemma: property generalized c matrix} that the corresponding $p_{i,j}^+$ or $p_{i,j}^-$ for each $j\in [1, r_i]$ is also $1$, hence the generalized sign coherence.
\end{proof}

The following lemma will be useful  later.

\begin{lemma}\label{lemma: generalized c vectors form basis}
	In the principal coefficient case, for any $t\in \mathbb T_n$, the set of coefficients in seed $\Sigma_t$ $$\{p_{i,j;t}\mid i\in I, j\in [1, r_i]\}$$ form a $\mathbb Z$-basis of $\mathbb P \cong \mathbb Z^{d}$ where $d = \sum_{i\in I} r_i$.
\end{lemma}

\begin{proof}
	It follows directly from the mutation formula \Cref{prop: matrix of coefficient mutation} and \Cref{lemma: property generalized c matrix}.
\end{proof}

\subsection{Separation formula}
In this section, we describe the separation formula for generalized cluster variables, which can be derived in the exact same way as \cite[Theorem 3.7]{fomin2007cluster}, so we omit the proof.

\begin{definition}
	Let $\mathscr A^\mathrm{prin}(\Sigma_{t_0})$ be a generalized cluster algebra with principal coefficients at $\Sigma_{t_0} = (\mathbf x, \mathbf p, B)$. We define the rational function
	\[
	X_{l;t}\in \mathbb {Q}_\mathrm{sf}(\mathbf x, \mathbf p)
	\]
	corresponding to the subtraction-free rational expression of the cluster variable $x_{l;t}$ by iterating exchange relations. Here $(\mathbf x, \mathbf p)$ denote the set of all variables in $\mathbf x$ and $\mathbf p$.
	
	Define the rational function
	\[
	F_{l;t}(\mathbf p) = X_{l;t}((1,\dots, 1), \mathbf p)\in \mathbb Q_\mathrm{sf}(\mathbf p).
	\]
\end{definition}

In general, for a subtraction free expression $F$ in $\mathbb Q_\mathrm{sf}(x_1, \dots, x_n)$ and an arbitrary semifield $\mathbb P$, we use the notation
\[
	F\mid_{\mathbb P}(y_1,\dots y_n)\in \mathbb P
\] 
for the evaluation at $x_i = y_i$. This evaluation is well-defined (i.e. independent of the expression used) because of the universal property of the semifield $\mathbb Q_\mathrm{sf}(x_1, \dots, x_n)$; see \Cref{subsection: semifield}.

\begin{theorem}[cf. {\cite[Proposition 3.6, Theorem 3.7]{fomin2007cluster}}]\label{thm: separation formula}
\
\begin{enumerate}
	\item We have
	\[
	X_{l;t}\in \mathbb Z[x_i^\pm; p_{i,j}\mid i\in I, j\in [1,r_i]],\quad F_{l;t}\in \mathbb Z[p_{i,j}\mid i\in I, j\in [1,r_i]].
	\]
	\item Let $\mathscr A$ be a generalized cluster algebra over an arbitrary semifield $\mathbb P'$, with an initial seed $\Sigma_{t_0} = (\mathbf x, \mathbf p, B)$. Then the cluster variables in $\mathscr A$ can be expressed in the initial cluster as
	\[
		x_{l;t} = \frac{X_{l;t}\mid_{\mathcal F}(\mathbf x, \mathbf p)}{F_{l;t}\mid_{\mathbb P'}(\mathbf p)}
	\]
	where $\mathcal F$ is the ambient field for $\mathscr A$.
\end{enumerate}
\end{theorem}

\begin{remark}\label{rmk: positivity for general coefficients}
Suppose the positivity for $x_{l;t}$ in the principal coefficients case (where we denote the semifield by $\mathbb P$) has been established. This means $X_{l;t}$ has a subtraction free expression as a Laurent polynomial (i.e. whose coefficients are in $\mathbb {NP}$). Then the evaluation $X_{l;t}\mid_{\mathcal F}(\mathbf x, \mathbf p)$ also has positive coefficients in $\mathbb{NP}'$, while $F_{l;t}\mid_{\mathbb P'}(\mathbf p)$ is an element in $\mathbb P'$. Thus it follows by \Cref{thm: separation formula} that $x_{l;t}$ has positive coefficients in the case of arbitrary $\mathbb P'$.
\end{remark}

\subsection{Cluster algebras with specialized coefficients} \label{subsection: specialized coefficients}
We fix a field $\Bbbk$ of characteristic $0$ and consider the case of geometric coefficients. In this case, the generalized cluster algebra $\mathscr A(\Sigma)$ for $\Sigma = (\mathbf x, \mathbf p, B)$ can be viewed as a subring of $\Bbbk \mathbb P[x_1^\pm, \dots, x_n^\pm]$ where $\Bbbk \mathbb P$ is the group algebra of $\mathbb P$ over $\Bbbk$.

Let $\lambda\colon \mathbb P \rightarrow \Bbbk^*$ be a group homomorphism (which we will later refer to as an \emph{evaluation}). It extends to a $\Bbbk$-algebra homomorphism $$\lambda\colon \Bbbk \mathbb P[x_1^\pm, \dots, x_n^\pm] \rightarrow \Bbbk[x_1^\pm, \dots, x_n^\pm].$$ We denote the image of $\mathscr A(\Sigma)\otimes \Bbbk$ by $\mathscr A(\Sigma, \lambda)$. So we have a family of $\Bbbk$-algebras parametrized by $(\Bbbk^*)^l$ if the free abelian group $\mathbb P$ is of rank $l$. Each $\mathscr A(\Sigma, \lambda)$ is in fact the $\Bbbk$-subalgebra generated by cluster variables (with coefficients specialized by $\lambda$) within $\Bbbk[x_1^\pm, \dots, x_n^\pm]$. These are what we call (generalized) \emph{cluster algebras with specialized coefficients}.

We point out that an ordinary cluster algebra with trivial coefficients (i.e. when $\mathbb P$ is trivial) is actually a {generalized cluster algebra with specialized coefficients}. Suppose $B$ is a skew-symmetrizable matrix and let $r_i$ be the gcd of the $i$-th column (if that column is non-zero). Let $\mathscr A^\mathrm{prin}(\Sigma)$ be the generalized cluster algebra with principal coefficients where $\Sigma$ has exchange matrix $B$. Choose a group homomorphism $\lambda\colon \mathrm{Trop}(\mathbf p)\rightarrow \Bbbk^*$ such that the specialized exchange polynomials equals the usual cluster exchange binomial, i.e.
\[
    \prod_{j = 1}^{r_i}\left( \lambda(p_{i,j})u+v\right) = u^{r_i} + v^{r_i}.
\]
Of course such $\lambda$ always exists assuming $\Bbbk$ is algebraically closed. Then it is easy to check that every generalized mutation becomes an ordinary mutation: if $t\frac{k}{\quad\quad} t'$,
\[
    x_{k;t'} = x_{k;t}^{-1} \left( \prod_{i\in I} x_i^{[b_{ik}^t]_+}+ \prod_{i\in I} x_i^{[-b_{ik}^t]_+} \right).
\]
Thus the algebra $\mathscr A^\mathrm{prin}(\Sigma, \lambda)$ has the exact same cluster variables as the ordinary cluster algebra with trivial coefficients, and can thus be viewed as an ordinary cluster algebra.

\subsection{An example in type $B_2$ with principal coefficients}

We consider $\mathscr A^\mathrm{prin}(\mathbf x, \mathbf p, B)$ with principal coefficients for $B = \begin{bmatrix}
	0 & -2\\
	1 & 0
\end{bmatrix}$ which is of type $B_2$ in the finite type classification \cite{fomin2003cluster} and \cite[Theorem 2.7]{chekhov2011teichmller}. We write $x_{i;t_0} = A_i$, and $p_{i,j;t_0} = t_{ij}$. For the subgraph
\[
t_0 \frac{_1}{\quad\quad}t_1 \frac{_2}{\quad\quad}t_2 \frac{_1}{\quad\quad}t_3 \frac{_2}{\quad\quad}t_4 \frac{_1}{\quad\quad}t_5 \frac{_2}{\quad\quad}t_6
\]
of $\mathbb T_2$, we have the associated labeled seeds calculated in \Cref{tab: labeled seeds example}

\begin{table}[ht]
    \centering
\begin{tabular}{|c |c |c |c |c |c |c |} 
 \hline
 $t$ & $B^t$ & $p_{1,1;t}$ & $p_{2,1;t}$ & $p_{2,2;t}$ & $x_{1;t}$ & $x_{2;t}$ \\ 
\hline
&&&&&&\\

 0&$\begin{bmatrix}
 	0 & -2\\
 	1 & 0
 \end{bmatrix}$ & $t_{11}$ & $t_{21}$ & $t_{22}$ & $A_1$ & $A_2$  \\

&&&&&&\\

1&$\begin{bmatrix}
 	0 & 2\\
 	-1 & 0
 \end{bmatrix}$ & $t_{11}^{-1}$ & $t_{21}$ & $t_{22}$ & $A_1^{-1}(1 + t_{11}A_2)$ & $A_2$  \\

&&&&&&\\
 
 2&$\begin{bmatrix}
 	0 & -2\\
 	1 & 0
 \end{bmatrix}$ & $t_{11}^{-1}$ & $t_{21}^{-1}$ & $t_{22}^{-1}$ & $A_1^{-1}(1 + t_{11}A_2)$ & $
 	\substack{A_2^{-1}\left(1+t_{21}A_1^{-1}(1+t_{11}A_2)\right)
 	\\\left(1+t_{22}A_1^{-1}(1+t_{11}A_2)\right)}$\\

&&&&&&\\

 3&$\begin{bmatrix}
 	0 & 2\\
 	-1 & 0
 \end{bmatrix}$ & $t_{11}$ & $t_{11}^{-1}t_{21}^{-1}$ & $t_{11}^{-1}t_{22}^{-1}$ & $\substack{A_1A_2^{-1} \big( \left(1+t_{21}A_1^{-1} \right) \left( 1+t_{22}A_1^{-1} \right) \\
 + t_{11}t_{21}t_{22}A_1^{-2}A_2 \big)}$ & $
 	\substack{A_2^{-1}\left(1+t_{21}A_1^{-1}(1+t_{11}A_2)\right)
 	\\\left(1+t_{22}A_1^{-1}(1+t_{11}A_2)\right)}$\\

&&&&&&\\

 4&$\begin{bmatrix}
 	0 & -2\\
 	1 & 0
 \end{bmatrix}$ & $t_{11}^{-1}t_{21}^{-1}t_{22}^{-1}$ & $t_{11}t_{21}$ & $t_{11}t_{22}$ & $\substack{A_1A_2^{-1} \big( \left(1+t_{21}A_1^{-1} \right) \left( 1+t_{22}A_1^{-1} \right) \\
 + t_{11}t_{21}t_{22}A_1^{-2}A_2 \big)}$ & $
 	A_2^{-1}\left(t_{21} + A_1\right)
 	\left(t_{22} + A_1\right)$\\

&&&&&&\\

5&$\begin{bmatrix}
 	0 & 2\\
 	-1 & 0
 \end{bmatrix}$ & $t_{11}t_{21}t_{22}$ & $t_{22}^{-1}$ & $t_{21}^{-1}$ & $A_1$ & $
 	A_2^{-1}\left(t_{21}+A_1\right)
 	\left(t_{22}+A_1\right)$\\

&&&&&&\\

6&$\begin{bmatrix}
 	0 & -2\\
 	1 & 0
 \end{bmatrix}$ & $t_{11}$ & $t_{22}$ & $t_{21}$ & $A_1$ & $
 	A_2$\\
&&&&&&\\
\hline
\end{tabular}
    \caption{Labeled seeds of $\mathscr A^\mathrm{prin}$}
    \label{tab: labeled seeds example}
\end{table}

We note that the $\Sigma_{t_6}$ is not exactly the same as the $\Sigma_{t_0}$ but up to a switch of $p_{2,1}$ and $p_{2,2}$.

\subsection{Generalized $Y$-seeds}
We define generalized $Y$-seeds (with coefficients) and their mutations. The formulation to including coefficients in $Y$-seeds comes from \cite{bossinger2020toric}. The following definition is a generalization of \cite[Definition 2.15]{bossinger2020toric}, which is an enhancement of a $Y$-seed of \cite{fomin2007cluster}.

\begin{definition}\label{def: y-seed}
	A \emph{generalized labeled $Y$-seed} (with coefficients) $\Delta$ is a triple $(\mathbf y, \mathbf q, B)$, where
	\begin{enumerate}
	\item[$\bullet$] $\mathbf q = (\mathbf q_1, \dots, \mathbf q_n)$ is an $n$-tuple of $r_i$-tuples $\mathbf q_i = (q_{i,1},\dots q_{i,r_i})$ of elements in a semifield $\mathbb P$ for positive integers $r_i$, $1\leq i\leq n$.
	
	\item[$\bullet$] $\mathbf y = \{ y_1, \dots, y_n\}$ is a collection of elements in some universal semifield $\mathbb {QP}_\mathrm{sf}(u_1, \cdots, u_l)$.
	
	\item[$\bullet$] $B$ is a left skew-symmetrizable integer matrix such that the $i$-th column is divisible by $r_i$ for every $i$.
	\end{enumerate}
\end{definition}

\begin{definition}
For $k\in \{1,\dots, n\}$, we define the \emph{mutation} of a $Y$-seed $(\mathbf y, \mathbf q, B)$ in direction $k$ as a new $Y$-seed $\mu_k(\mathbf y, \mathbf q, B) \coloneqq ((y_i'), (\mathbf q_i'), B')$ in the following way:
	\begin{equation}
		B' = \mu_k(B);
	\end{equation}
	
	$$q_{k,j}' = q_{k,j}^{-1}\quad  \text{for $j\in [1, r_k]$};$$
\begin{equation}
	\text{for $i\neq k$,  $j\in [1, r_i]$,} \quad q'_{i,j} = \begin{dcases}
	q_{i,j}\cdot \left( \prod_{l=1}^{r_k}q_{k,l}^- \right)^{-\beta_{ik}}\quad &\text{if $-\beta_{ki}>0$}\\
	q_{i,j}\cdot \left( \prod_{l=1}^{r_k}q_{k,l}^+ \right)^{-\beta_{ik}} \quad &\text{if $-\beta_{ki}\leq 0$},
	\end{dcases}
\end{equation}
or equivalently
\begin{equation*}
	\text{for $i\neq k$,}\quad q'_{i,j} = 
	q_{i,j} \prod_{l=1}^{r_k} \left( 1 \oplus q_{k,l}^{\mathrm{sgn}(-\beta_{ki})} \right) ^{\beta_{ik}};	
	\end{equation*}

	\begin{equation}
		y_i' = \begin{dcases}
			y_i \prod_{l=1}^{r_k} \left( q_{k,l}^{\mathrm{sgn}(\beta_{ik})}  y_k^{\mathrm{sgn}(\beta_{ik})}  + q_{k,l}^{\mathrm{sgn}(-\beta_{ik})} \right)^{\beta_{ik}} \quad &\text{if $i\neq k$}\\
			y_k^{-1} \quad &\text{if $i = k$}.
		\end{dcases}
	\end{equation}

\end{definition}

As in \Cref{lemma: mutation is involutive}, it is straightforward to check that the mutation $\mu_k$ on a generalized $Y$-seed is involutive in the same direction.

\begin{definition}
	A \emph{generalized $Y$-pattern} is an association $$t\mapsto \Delta_t = (\mathbf y_t, {\mathbf q}_t, B^t)$$ to every vertex $t\in \mathbb T_n$ a generalized labeled $Y$-seed $\Delta_t$ such that if $t$ and $t'$ are connected by an edge labeled by $k\in I$, then we have
	\[
	\Delta_{t'} = \mu_k(\Delta_t).
	\]
\end{definition}

\begin{definition}
	We say that a generalized $Y$-pattern has principal coefficients at a vertex $t_0\in \mathbb T_n$ if $\mathbb P$ is the tropical semifield
\[
\mathrm{Trop}\left ( q_{i,j;t_0} \mid i\in I, j\in [1, r_i]\right).
\]
\end{definition}

Given a $Y$-pattern, the elements $y_{i;t}$ for $t\in \mathbb T_n$ are called \emph{$Y$-variables}.

\begin{remark}
	In the case that for any $i\in I$,
	\[
	q_{i,1} = q_{i,2} = \cdots = q_{i,r_i},
	\]
	a generalized $Y$-seed with coefficients as in \Cref{def: y-seed} becomes a labeled $Y$-seed with coefficients in \cite{bossinger2020toric}. In this case, the mutation formula of $Y$-variables is independent of the choice $r_i$. So we get back to the non-generalized version by letting the coefficients $q_{i,j}$, $j\in [1, r_i]$, equal. While in the cluster case, one recovers the non-generalized seed mutation by choosing $r_i = 1$. This asymmetry suggests that our generalization is a natural one.
\end{remark}

To the best knowledge of the author, the generalized version of $Y$-patterns has not been considered in the literature. It is interesting to see if these generalized patterns appear naturally anywhere.

\section{Generalized cluster varieties}\label{section: cluster varieties}
Cluster varieties were introduced by Fock and Goncharov \cite{fock2009cluster}, giving a geometric view to cluster algebras (of geometric types). We follow \cite{gross2013birational} to define relevant notions such as fixed data and seeds. However, in order to deal with generalized coefficients, some new gadgets are needed.

\begin{definition}
	We recall the \emph{fixed data} $\Gamma$ from \cite{gross2013birational}. The fixed data $\Gamma$ consists of 
	\begin{itemize}
		\item a lattice $N$ of finite rank with a skew-symmetric bilinear form $\omega \colon N\times N \rightarrow \mathbb Q$;
		\item an \emph{unfrozen sublattice} $N_\mathrm{uf}\subset N$, a saturated sublattice of $N$;
		\item an index set $I = \{1, \dots, \operatorname{rank} N\}$ and a subset $I_\mathrm{uf} = \{1, \dots, \operatorname{rank} N_\mathrm{uf}\}$;
		\item positive integers $d_i$ for $i\in I$ with greatest common divisor $1$;
		\item a sublattice $N^\circ \subset N$ of finite index such that $\omega( N_\mathrm{uf}, N^\circ )\subset \mathbb Z$, $\omega( N, N_\mathrm{uf}\cap N^\circ )\subset \mathbb Z$;
		\item $M = \Hom(N, \mathbb Z)$, $M^\circ = \Hom(N^\circ, \mathbb Z)$;
	\end{itemize}
\end{definition}

\subsection{Generalized $\mathcal A$-cluster variety}

\begin{definition}
	Given fixed data $\Gamma$, an \emph{$\mathcal A$-seed with (generalized) coefficients} is a pair $\mathbf s = ( \mathbf e, \mathbf p)$ consisting of a \emph{seed} $\mathbf e = (e_i)_{i\in I}$ which is a labeled collection of elements in $N$ and a labeled collection of tuples of coefficients $\mathbf p = (\mathbf p_i)_{i\in I_\mathrm{uf}}$ where $\mathbf p_i = (p_{i,j})_{j \in [1, r_j]}$ and $p_{i,j}$ belongs to some tropical semifield $\mathbb P$ such that
	\begin{enumerate}
		\item $\{ e_i \mid i\in I\}$ is a basis for $N$;
		\item $\{ e_i \mid i\in I_\mathrm{uf}\}$ is a basis for $N_\mathrm{uf}$;
		\item $\{d_ie_i\mid i\in I\}$ is a basis for $N^\circ$;
		\item for $i\in I_\mathrm{uf}$, the elements $w_i \coloneqq \omega(-,d_i e_i)/r_i$ belong to $M$. 
	\end{enumerate}
\end{definition}

For such a seed $\mathbf s$, we define two matrices $B = B(\mathbf s) = (b_{ij})$ and $\tilde B = \tilde B(\mathbf s) = (\beta_{ij})$ by setting
\[
b_{ij} \coloneqq \omega( e_i, d_j e_j)\quad \text{and} \quad \beta_{ij} \coloneqq \langle e_i, w_j \rangle = b_{ij}/r_j.
\]

\begin{definition}\label{def: mutation of geometric seed}
	Given $\mathbf s$ an $\mathcal A$-seed with coefficients, for $k\in I_\mathrm{uf}$, we define the \emph{mutation in direction $k$}, $\mu_k(\mathbf s) = (\mathbf e',  {\mathbf p}')$ by
\[
e_i' = \begin{dcases}
	-e_k\quad &\text{if $i = k$}\\
	e_i + [\langle e_i, - r_kw_k \rangle]_+ e_k &\text{if $i\neq k$}; 
\end{dcases}
\]
and
$$p_{k,j}' = p_{k,j}^{-1}\quad  \text{for $j \in [1, r_k]$};$$
\begin{equation*}
	\text{for $i\neq k, j \in [1, r_i]$,}\quad p'_{i,j} = \begin{dcases}
	p_{i,j}\cdot \left( p_{k;-}\right)^{\beta_{ki}}\quad &\text{if $\beta_{ik}>0$}\\
	p_{i,j}\cdot \left( p_{k;+}\right)^{\beta_{ki}} \quad &\text{if $\beta_{ik}\leq 0$},
	\end{dcases}
\end{equation*}
\end{definition}

\begin{remark}
If we write $w_i' = \omega (-, \dfrac{d_i}{r_i}e_i')$ as the mutations of $w_i$, then they are given by 
\[
w_i'= \begin{dcases}
	-w_k,\quad &\text{if $i = k$};\\
	w_i + [\langle r_ke_k, w_i\rangle]_+ w_k, \quad &\text{if $i\neq k$}.
\end{dcases}
\]
Denote the dual basis of $(e_i)$ by $(e_i^*)$ and the dual of $(e_i') = \mu_k(\mathbf e)$ by $(e_i^{\prime,*})$. We have
\[
e_i^{\prime,*} = \begin{cases}
	-e_k^* + \sum_{j} [\langle e_j, -r_kw_k\rangle]_+e_j^* \quad &\text{if $i = k$};\\
	e_i^* \quad &\text{if $i\neq k$}.
\end{cases}
\] 
	
\end{remark}

If there is no confusion, we will call an $\mathcal A$-seed with coefficients simply a seed.

Let $R = \Bbbk \mathbb P$, the group algebra of $\mathbb P$ over the ground field $\Bbbk$. To any $\mathcal A$-seed $\mathbf s$, we associate a copy of the $R$-torus $T_{N, \mathbf s}(R)\coloneqq \Spec(\Bbbk[M] \otimes_\Bbbk R)$.

\begin{definition}
To the mutation $\mu_k$ from $\mathbf s$ to $\mu_k(\mathbf s)$, there is an associated birational morphism (over $R$)
\[
	\mu_k \colon T_{N, \mathbf s}(R) \dasharrow T_{N, \mu_k(\mathbf s)}(R), \quad \mu_k^*(z^m) = z^m f_k^{-\langle e_k, m \rangle},
\]
where 
\[
	f_k \coloneqq \prod_{j = 1}^{r_k} \left( p_{k,j}^- + p_{k,j}^+ z^{w_k}\right)\in R[M].
\]
We call this birational transformation the \emph{$\mathcal A$-cluster mutation} associated to the mutation $\mu_k$ of seeds.
\end{definition}

\begin{definition}\label{def: rooted tree GHK}
	We define the \emph{oriented rooted tree} $\mathfrak T_n$ (where $n = |I_\mathrm{uf}|$) as in \cite{gross2013birational}. It is the infinite tree generated from a root $v_0$ such that
\begin{enumerate}
	\item $v_0$ has outgoing edges labeled by $I_\mathrm{uf} = \{1,\dots, n\}$;
	\item any other vertex has one unique incoming edge, and outgoing edges labeled by $I_\mathrm{uf}$.
\end{enumerate}

\end{definition}

Let $v_0\in \mathfrak T_n$ be the root. Then for any other vertex $v\in \mathfrak T_n$, there is a unique oriented path from $v_0$ to $v$. We associate a seed $\mathbf s$ to the root $v_0$, the unique path from $v$ to $v_0$ determines a seed $\mathbf s_{v}$ by applying the mutations in directions of the labelings in the path to the initial seed $\mathbf s$. Therefore we have an association $v\mapsto \mathbf s_v$ for $v\in \mathfrak T_n\setminus \{v_0\}$ and $v_0\mapsto \mathbf s$ such that for an edge $v \xrightarrow{k} v'$ in $\mathfrak T_n$, then
\[
\mathbf s_{v'} = \mu_k(\mathbf s_v).
\]

Suppose the unique path from $v_0$ to $v$ walks through edges labeled by $k_1,k_2,\dots,k_l$.   There is then the birational map
\[
\mu_{v_0,v}\coloneqq \mu_{k_l}\circ\cdots \circ \mu_{k_2}\circ\mu_{k_1}\colon T_{N,\mathbf s}(R) \dasharrow T_{N,\mathbf s_{v}}(R).
\]
For arbitrary two vertices $v$ and $v'$ in $\mathfrak T_n$, there is the birational map
\[
\mu_{v,v'}\coloneqq \mu_{v_0,v'}\circ \mu_{v_0,v}^{-1}\colon T_{N,\mathbf s_{v}}(R) \dasharrow T_{N,\mathbf s_{v'}}(R).
\]
These birational maps surely satisfy the cocycle condition. We use the following lemma to glue $T_{N,\mathbf s_v}$ together.

\begin{lemma}[{\cite[Lemma 3.10]{bossinger2020toric}}, {\cite[Proposition 2.4]{gross2013birational}}]\label{lemma: glue schemes along birational maps}
	Let $\mathcal I$ be a set and $\{S_i\mid i\in I\}$ be a collection of integral separated schemes of finite type over a locally Noetherian ring $R$, with birational maps (of $R$-schemes) $f_{ij}\colon S_i \dasharrow S_j$ for all $i,j$, verifying the cocycle condition $f_{jk}\circ f_{ij} = f_{ik}$ as rational maps and such that $f_{ii}$ is the identity map. Let $U_{ij}\subset S_i$ be the largest open subscheme such that $f_{ij}\colon U_{ij}\rightarrow f_{ij} (U_{ij})$ is an isomorphism. Then there is an $R$-scheme
	\[
		S= \bigcup_{i\in \mathcal I}S_i
	\]
obtained by gluing the $S_i$ along the open sets $U_{ij}$ via the maps $f_{ij}$.
\end{lemma}

\begin{definition}
	Let $\Gamma$ be fixed data and $\mathbf s$ be an $\mathcal A$-seed with coefficients. We apply \Cref{lemma: glue schemes along birational maps} to glue together the collection of tori indexed by $\mathfrak T_n$ to get the \emph{generalized $\mathcal A$-cluster variety} associated to $\mathbf s$ (as an $R$-scheme)
	\[
	\mathcal A_\mathbf s = \mathcal A_{\Gamma, \mathbf s} \coloneqq \bigcup_{v\in \mathfrak T} T_{N, \mathbf s_v}(R).
	\]
\end{definition}

We now explain how to obtain a generalized cluster pattern from $\mathcal A_\mathbf s$, justifying the name generalized $\mathcal A$-cluster variety. We assume $N_\mathrm{uf} = N$, thus $I_\mathrm{uf} = I$.\footnote{This is because we do not define cluster patterns with frozen directions. This can be done by making mutations only available at a subset of a given cluster, leaving the rest variables frozen. However, one can always treat the frozen variables as making up coefficients in a cluster pattern.}

Recall we have the association $v \mapsto \mathbf s_v = \mu_{v_0, v}(\mathbf s)$ for $v\in \mathfrak T_n$. We write $\mathbf s_{v} = (\mathbf e_v,  \mathbf p_{v})$ where $\mathbf e_v = (e_{i;v}\mid i\in I)$, $\mathbf p_{v} = (\mathbf p_{i;v}\mid i\in I)$ and $\mathbf p_{i;v} = (p_{i,j;v} \mid j\in [1, r_i])$.

Sending $v_0$ to any vertex $t_0$ in the $n$-regular tree $\mathbb T_n$ gives a unique surjective map
\[
\pi\colon \mathfrak T_n \rightarrow \mathbb T_n, \quad v_0\mapsto t_0
\]
such that the labeling on edges is preserved.

For any seed $v\in \mathfrak T_n$, there is the corresponding labeled seed with coefficients (in the sense of \Cref{def: seed})
\[
	\Sigma_v = \Sigma(\mathbf s_v)\coloneqq  (\mathbf x_v, \mathbf p_v, B^v)
\]
where
\[
x_{i,v} \coloneqq \mu_{v_0,v}^*\left( z^{e_{i;v}^*} \right)\in \mathbb {QP}(x_1, \dots, x_n), \quad b^v_{ij} \coloneqq \omega(e_{i;v}, d_je_{j;v}),
\]
where $x_i = x_{i,v_0}$.

\begin{lemma}
	If two vertices $v$ and $v'$ vertices of $\mathfrak T_n$ descend to the same vertex in $\mathbb T_n$, i.e. $\pi(v) = \pi(v')$, then their corresponding labeled seeds with coefficients are identical, i.e. $\Sigma_v = \Sigma_{v'}$.
\end{lemma}

\begin{proof}
	Suppose the unique path in $\mathfrak T_n$ from $v_0$ to $v$ goes through edges labeled by $k_1,\dots, k_l$ in order. We show in the following by induction that
	\[
	\mu_{k_l}\circ \cdots \circ \mu_{k_1} (\Sigma_{v_0}) = \Sigma_v.
	\]
	where the operation $\mu_k$ is the mutation in direction $k$ of labeled seeds with coefficients in the sense of \Cref{def: mutation of seed}.
	
	Let $v_1 \overset{k}{\longrightarrow} v_2$ be in $\mathfrak T_n$. Then one checks $B^{v_2} = \mu_k(B^{v_1})$ using the fact that $\mathbf e_{v_2} = \mu_k(\mathbf e_{v_1})$, which is standard from \cite{gross2013birational}.
	The coefficients parts $\mathbf p_{v_1}$ and $\mathbf p_{v_2}$ are related by the mutation $\mu_k$ by definition. So we only need to check that $\mathbf x_{v_1}$ and $\mathbf x_{v_2}$ are also related by $\mu_k$.
	
	Note that $\mu_{v_0,v_2}^* =  \mu_{v_0,v_1}^* \circ \mu_k^*$. So we have for $i\neq k$ $$x_{i; v_2} = \mu_{v_0, v_1}^*\left(\mu_k^*\left(z^{e_{i;v_2}^*}\right)\right) = \mu_{v_0, v_1}^* \left(z^{e_{i;v_1}^*}\right) = x_{i;v_1}$$
	and \begin{align*}
		x_{k; v_2} &= \mu_{v_0,v_1}^*\left( \mu_k^*\left( z^{e_{k;v_2}^*}\right) \right)\\
		&= \mu_{v_0,v_1}^*\left( z^{-e_{k;v_1}^* + \sum_{}[-b_{ik}^{v_1}]_+e_{i;v_1}^*} \prod_{j = 1}^{r_k} \left( p_{k,j;v_1}^- + p_{k,j;v_1}^+ z^{w_{k;v_1}}\right) \right)\\
	&= \mu_{v_0,v_1}^*\left( z^{-e_{k;v_1}^*} \prod_{j = 1}^{r_k} \left( p_{k,j;v_1}^-z^{w_{k;v_1}^-} + p_{k,j;v_1}^+z^{w_{k;v_1}^+} \right) \right)\\
	&= \mu_{v_0,v_1}^*\left(z^{-e_{k;v_1}^*}\right) \prod_{j = 1}^{r_k} \left( p_{k,j;v_1}^-\mu_{v_0,v_1}^*\left(z^{w_{k;v_1}^-}\right) + p_{k,j;v_1}^+\mu_{v_0,v_1}^*\left(z^{w_{k;v_1}^+}\right) \right)\\
	&= x_{k;v_1}^{-1}\prod_{j = 1}^{r_k}\left(p_{k,j;v_1}^- \prod_{i\in I} x_{i;v_1}^{[-\beta_{ik}]_+} + p_{k,j;v_1}^+\prod_{i\in I} x_{i;v_1}^{[\beta_{ik}]_+} \right).
	\end{align*} 
	The only unexplained notation in the above equations is that for any $w = \sum_{i\in I} a_i e_i^*\in M$, we write
	\[
	w^- \coloneqq \sum_{i\in I} [-a_i]_+e_i^*\quad \text{and} \quad w^+ \coloneqq \sum_{i\in I} [a_i]_+e_i^*.	
	\]
	Now we have checked that $\mu_k\left(\Sigma_{v_1}\right) = \Sigma_{v_2}$. By induction on the distance from $v$ to the root $v_0$, we conclude that $\mu_{k_l}\circ \cdots \circ \mu_{k_1}\left( \Sigma_{v_0} \right) = \Sigma_v$ for any $v\in \mathfrak T_n$. Since $\mu_k$ is involutive, we can reduce the sequence $(k_1,\cdots, k_l)$ by deleting pairs of consecutive identical indices until there is none. So $\Sigma_v$ only depends on the reduced sequence of edge labels from $v_0$ to $v$. Now notice that two vertices $v$ and $v'$ in $\mathfrak T_n$ have the same projection $t$ in $\mathbb T_n$ if and only if they have the same reduced sequence of edge labels from $v_0$, meaning the same labeled seed with coefficients $\Sigma_t \coloneqq \Sigma_v = \Sigma_{v'}$.
\end{proof}

\begin{proposition}\label{prop: cluster variety to cluster pattern}
	According to the above lemma, we have that the labeled seeds $\Sigma_v$ and $\Sigma_{v'}$ are equal if $\pi(v) = \pi(v') = t\in \mathbb T_n$. So we can denote them all by $\Sigma_t$. The association $t\mapsto \Sigma_t$ for every $t\in \mathbb T_n$ is a cluster pattern.
\end{proposition}

\begin{proof}
Suppose the unique path from $t_0$ to some $t\in \mathbb T_n$ walks through edges in order of $k_1,\dots, k_l$. Then already in the proof of the above lemma, we have
\[
\Sigma_t = \mu_{k_l}\circ \cdots \circ \mu_{k_1} \left( \Sigma_{t_0} \right).
\]
This association by definition gives a cluster pattern.
\end{proof}

\begin{definition}
	The (generalized) \emph{upper cluster algebra} $\overline {\mathscr A}(\mathbf s)$ (of an $\mathcal A$-seed $\mathbf s$ with coefficients) is defined to be the $R$-algebra
	\[
	H^0 \left(\mathcal A_{\mathbf s}, \mathcal O_{\mathcal A_{\mathbf s}}\right) = \bigcap_{v\in \mathfrak T_n} H^0 \left(T_{N, \mathbf s_v}(R), \mathcal O_{T_{N, \mathbf s_v}(R)} \right),
	\]
	the ring of regular functions on the (generalized) $\mathcal A$-cluster variety $\mathcal A_{\mathbf s}$.
\end{definition}

By definition the upper cluster algebra is the algebra of all Laurent polynomials that remains Laurent polynomials after an arbitrary sequence of mutations. It follows from the Laurent phenomenon that all cluster variables are elements in the upper cluster algebra, thus the inclusion
\[
\mathscr A(\mathbf s) \subset \overline{\mathscr A}(\mathbf s),
\]
where the former denotes the subalgebra generated by cluster variables, i.e. the cluster algebra (over $R$).

The notion of principal coefficients can be easily translated into the current setting.

\begin{definition}
An $\mathcal A$-seed $\mathbf s$ is said to have \emph{principal coefficients} if the associated labeled seed $\Sigma(\mathbf s)$ has principal coefficients.	
\end{definition}

The associated cluster pattern with $t_0\mapsto \Sigma(\mathbf s)$, $t\mapsto \Sigma(\mathbf s_v)$ (where $t = \pi(v)$) then has principal coefficients at $t_0$. In this case, we denote the corresponding cluster variety by $\mathcal A_\mathbf s^\mathrm{prin}$.

\subsection{Generalized $\mathcal X$-cluster variety} Given fixed data $\Gamma$ as in the last section, we define the notion of (generalized) $\mathcal X$-seeds with coefficients.

\begin{definition}
An $\mathcal X$-seed with (generalized) coefficients $\mathbf s = (\mathbf e, \mathbf q)$ is the same as an $\mathcal A$-seed. We use the symbol $\mathbf q$ instead of $\mathbf p$ to stress that it is an $\mathcal X$-seed.	
\end{definition}

What distinguish $\mathcal X$-seeds with $\mathcal A$-seeds is the mutation.

\begin{definition}\label{def: mutation for x-seed}
	Given an $\mathcal X$-seed $\mathbf s = (\mathbf e, \mathbf q)$, we define the \emph{mutation in direction $k$}, $\mu_k(\mathbf s) = (\mathbf e', \mathbf q')$ by
\[
e_i' = \begin{dcases}
	-e_k\quad &\text{if $i = k$}\\
	e_i + [\langle e_i, - r_kw_k \rangle]_+ e_k &\text{if $i\neq k$}; 
\end{dcases}
\]
and
$$q_{k,j}' = q_{k,j}^{-1}\quad  \text{for $j \in [1,  r_k]$};$$
\begin{equation*}
	\text{for $i\neq k, j\in [1, r_i]$,}\quad q'_{i,j} = \begin{dcases}
	q_{i,j}\cdot \left( q_{k;-}\right)^{-\beta_{ik}}\quad &\text{if $-\beta_{ki}>0$}\\
	q_{i,j}\cdot \left( q_{k;+}\right)^{-\beta_{ik}} \quad &\text{if $-\beta_{ki}\leq 0$},
	\end{dcases}
\end{equation*}
\end{definition}

So the pure seed part $\mathbf e$ behaves in the same way under mutation as in an $\mathcal A$-seed while the coefficients part $\mathbf q$ mutates differently, but same as the coefficients in a labeled $Y$-seed. Roughly, if in $\mathcal A$-seeds, the matrix $B$ governs the mutation of coefficients, then in $\mathcal X$-seeds, $-B^T$ does the job.

\begin{definition}
	Let $\mathbf s = (\mathbf e, \mathbf q)$ be an $\mathcal X$-seed with coefficients. Then there is the associated \emph{$\mathcal X$-cluster mutation}
\[
\mu_k\colon T_M(R) \dasharrow T_M(R),\quad \mu_k^*(z^n) = z^n \cdot \left( \prod_{l=1}^{r_k} \left( q_{k,l}^- + q_{k,l}^+ z^{e_k} \right) \right)^{-\langle n , -w_k \rangle},
\]
where $T_M(R)$ is the $R$-torus $\Spec(\Bbbk[N]\otimes R)$.
\end{definition}

\begin{definition}
	Let $\mathbf s$ be an $\mathcal X$-seed for $\Gamma$. Then there is a unique association $v\mapsto \mathbf s_v$ for every $v\in \mathfrak T_n$ such that $v_0\mapsto \mathbf s$ and adjacent associated seeds are related by mutations of $\mathcal X$-seeds in corresponding directions. Define the \emph{generalized $\mathcal X$-cluster variety} associated to $\mathbf s$ to be the $R$-scheme
	\[
	\mathcal X_\mathbf s = \mathcal X_{\Gamma, \mathbf s} \coloneqq \bigcup_{v\in \mathfrak T_n} T_{M, \mathbf s_v}(R)
	\]
	obtained by glueing $T_{M, \mathbf s_v}(R)$ via $\mathcal X$-cluster mutations using \Cref{lemma: glue schemes along birational maps}.
\end{definition}

Write $\mathbf s_v = ((e_{i;v}), (\mathbf q_{i;v}))$. Let us keep track of the monomials $z^{e_{i;v}}\in \Bbbk[N]$ (instead of $z^{e_{i;v}^*}$ in the $\mathcal A$-case). We define
\[
y_{i;v}\coloneqq \mu_{v,v_0}^*(z^{e_{i;v}})\in \mathrm{Frac}(\Bbbk[N]\otimes R).
\]
It turns out that these $y_{i;v}$ are the $Y$-variables of the $Y$-pattern induced by the $\mathcal X$-seed $\mathbf s$ described as follows. We take $\mathbf s$ as the initial seed. Analogous to the $\mathcal A$-situation, any vertex $v\in \mathfrak T_n$ descends to a vertex $t = \pi(v)\in \mathbb T_n$.

\begin{proposition}\label{prop: cluster variety to y pattern}
	For $v\in \mathfrak T_n$, define the generalized labeled $Y$-seed $\Delta_v = ((y_{i;v}), (\mathbf q_{i;v}), B^v)$. Then we have $\Delta_v = \Delta_{v'}$ if $\pi(v) = \pi(v') = t\in \mathbb T_n$. Then the association $t\mapsto \Delta_t$ for $t\in \mathbb T_n$ is a generalized $Y$-pattern with coefficients where $\Delta_t \coloneqq \Delta_v$ for any $v$ such that $t = \pi(v)$.
\end{proposition}

\begin{proof}
	We first note that the $Y$-variables $y_{i;v}$ live in the universal semifield $\mathbb{QP}_\mathrm{sf}(y_1, \dots, y_n)$ where $y_i = z^{e_i}$ are the initial $Y$-variables. The proof is completely analogous to \cref{prop: cluster variety to cluster pattern}. We leave the details to the reader.
\end{proof}

\subsection{Special coefficients}

By construction, given an $\mathcal A$-seed (resp. $\mathcal X$-seed) $\mathbf s$, there is the flat family
\[
\pi_{\mathcal A} \colon \mathcal A_{\mathbf s} \rightarrow \Spec R\quad (\mathrm{resp.}\ \pi_{\mathcal X}\colon \mathcal X_{\mathbf s} \rightarrow \Spec R). 
\]
Let $\lambda$ be a $\Bbbk$-point of $\Spec R$. Then the special fiber $\pi^{-1}(\lambda)$ is a $\Bbbk$-scheme and can be viewed as a \emph{generalized cluster variety with special coefficients}, denoted by $\mathcal A_{\mathbf s, \lambda}$ (resp. $\mathcal X_{\mathbf s, \lambda}$). They are also glued together by tori via birational morphisms (namely the $\mathcal A$- or $\mathcal X$-mutations specialized at $\lambda$)
\[
	\mathcal A_{\mathbf s, \lambda} = \bigcup_{v\in \mathfrak T_n} T_{N,v},\quad \mathcal X_{\mathbf s, \lambda} = \bigcup_{v\in \mathfrak T_n} T_{M,v}.
\]
The $\mathcal A$-type varieties (resp. $\mathcal X$-type varieties) lead to cluster patterns (resp. $Y$-patterns) with specialized coefficients. We have as before in the $\mathcal A$-case the inclusion of algebras
\[
\mathscr A(\mathbf s, \lambda) \subset \overline {\mathscr A}(\mathbf s, \lambda) \coloneqq H^0(\mathcal A_{\mathbf s,\lambda}, \mathscr O_{\mathcal A_{\mathbf s,\lambda}}).
\]

\subsection{Cluster duality}
The cluster duality of Fock and Goncharov predicts, in the ordinary case, that the varieties $\mathcal A_\mathbf s$ and $\mathcal X_\mathbf s$ are dual in the sense that the upper cluster algebra $\overline{\mathscr A}(\mathbf s)$ has a basis parametrized by the tropical set $\mathcal X^\mathrm{trop}(\mathbb Z)$ (and vice versa). Note here $\mathbf s$ is viewed as a seed without coefficients so we do not need to distinguish between $\mathcal A$- and $\mathcal X$-seeds. Strictly speaking, this statement is not true as in some cases $\mathcal X_\mathbf s$ may have too few regular functions \cite{gross2013birational}. This duality (or named the Fock--Goncharov full conjecture) is the main subject of study (on a precise modified formulation and when it is true) in the paper \cite{gross2018canonical}.

Our point of view is that it is more natural to include generalized cluster varieties in cluster dualities, which we will demonstrate in the principal coefficients case. We denote the $\mathcal X$-cluster variety with principal coefficients by $\mathcal X^\mathrm{prin}_\mathbf s$, where the coefficient group is the tropical semifield $$\mathbb P = \mathrm{Trop}(q_{ij}\mid i\in I, j \in [1, r_i]).$$ The scheme $\mathcal X^\mathrm{prin}_\mathbf s$ is over $\mathrm{Spec}(R)$ where $R = \Bbbk \mathbb P$. There are evaluations $\lambda$ sending $q_{ij}$ to $\lambda_{ij}\in \Bbbk^*$. Each $\lambda$ specifies an $\mathcal X$-cluster variety with special coefficients as in the following diagram
\[\begin{tikzcd}
\mathcal X^{\mathrm{prin}}_{\mathbf s,\lambda} \ar[r, hook]\ar[d] &	 \mathcal X^{\mathrm{prin}}_{\mathbf s} \ar[d, "\pi_\mathcal X"]\\
\mathrm{Spec}(\Bbbk) \ar[r, hook, "\lambda"] & \mathrm{Spec}(R).
\end{tikzcd}\]
With a general choice coefficients, $\mathcal X^{\mathrm{prin}}_{\mathbf s,\lambda}$ should be considered mirror dual to the family
\[
	\pi_\mathcal A \colon \mathcal A^{\mathrm{prin}}_\mathbf s \rightarrow \mathrm{Spec}(R),
\]
where $\mathbf s$ is viewed as an $\mathcal A$-seed with coefficients. We shall not fully justify this statement in this paper, but instead will show that the family $\pi_\mathcal A\colon \mathcal A^{\mathrm{prin}}_\mathbf s \rightarrow \mathrm{Spec}(R)$ (as well as the generalized cluster algebra $\mathscr A^\mathrm{prin}(\mathbf s)$) can be reconstructed from $\mathcal X_{\mathbf s, \lambda}^\mathrm{prin}$, through a consistent wall-crossing structure (or scattering diagram) $\mathfrak D_\mathbf s$ associated to $\mathcal X_{\mathbf s, \lambda}^\mathrm{prin}$; see \Cref{section: reconstruct cluster algebra}.

\section{Toric models and mutations}

This section is a generalization of \cite[Section 3]{gross2013birational} aiming for generalized cluster varieties. A \emph{log Calabi--Yau pair} $(X,D)$ is a smooth projective variety $X$ (over an algebraically closed field $\Bbbk$) with a reduced simple normal crossing divisor $D$ such that $K_X + D = 0$ where $K_X$ is the canonical divisor of $X$. A \emph{log Calabi--Yau variety} $U$ is the interior of a log Calabi--Yau pair $(X,D)$, i.e. $U = X\setminus D$. Described in \cite{gross2013birational}, particularly relevant in cluster theory are log Calabi--Yau pairs $(X,D)$ obtained from a blow-up $\pi\colon X \rightarrow X_\Sigma$
where $X_\Sigma$ is the toric variety associated to a fan $\Sigma$ in $\mathbb R^n$. The blow-up is along a hypersurface in the toric boundary of $X_\Sigma$, and $D$ is given by the strict transform of the toric boundary. We will see that both generalized $\mathcal X$- and $\mathcal A$-varieties can be realized as log Calabi--Yau varieties obtained this way (up to codimension two subsets).

\subsection{Toric models} Fix a lattice $N\cong \mathbb Z^n$ and let $M$ be its dual. Suppose for $i\in I = [1, l]$ we have pairs of vectors $(e_i, w_i)\in N\times M$ such that $\langle e_i, w_i\rangle =0$. We assume that all non-zero $e_i$ are primitive, but some of them may equal. For each $i$, we fix a positive integer $r_i$. We also take functions (elements in $\Bbbk[M]$)
\[
f_i = a_{i,0}+a_{i,1}z^{w_i}+\cdots+a_{i,r_i}z^{r_iw_i}
\]
with non-zero $a_{i,0}$ and $a_{i,r_i}$.

We construct in below a log Calabi--Yau variety $U_\Lambda$ using the data $$\Lambda \coloneqq ((e_i)_{i\in I}, (w_i)_{i\in I}, (f_i)_{i\in I}).$$ The following construction is what we mean by a \emph{toric model} for $U_\Lambda$ and we call such $\Lambda$ a \emph{toric model data}.

\begin{construction}[cf. {\cite[Construction 3.4]{gross2013birational}}]\label{construction: log calabi yau of a seed}
Given the data $\Lambda$, consider the fan 
\[
	\Sigma = \Sigma_\Lambda \coloneqq \{\mathbb R_{\geq 0}e_i\mid i\in I\} \cup \{0\}
\]
in $N_\mathbb R$. Let $X_\Sigma$ be the toric variety defined by $\Sigma$, and $D_i$ be the irreducible toric boundary divisor corresponding to $\mathbb R_{\geq 0}e_i$. Note that since $\langle e_i,w_i\rangle = 0$, $z^{w_i}$ does not vanish on $D_i$. Let $Z_i$ be the zero locus of $f_i$ on $D_i$, i.e. the closed subscheme $\bar V(f_i) \cap D_i$, which is a hypersurface. Blow up $X_\Sigma$ along $\bigcup_{i=1}^l Z_i$ to obtain 
\[	\pi \colon \tilde X_\Sigma \rightarrow X_\Sigma. 	\]
Let $\tilde D_i$ be the strict transform of $D_i$ in $\tilde X_\Sigma$. Then the open subscheme $U_\Lambda \coloneqq \tilde X_\Sigma \setminus \bigcup_i \tilde D_i$ is a log Calabi--Yau variety.
\end{construction}

\begin{definition}\label{def: k-mutable}
For $k\in I$, we say a toric model data $\Lambda$ \emph{$k$-mutable} if the pairs $(e_i, w_i)$ satisfy the condition $$\langle e_i, w_k \rangle = 0 \implies \langle e_k, w_i \rangle = 0 $$ for any $i\in I$. 	
\end{definition}

We define mutations of a $k$-mutable toric model data.

\begin{definition}\label{def: mutation equivalent data}
	Let $\Lambda$ be a $k$-mutable toric model data and $\Lambda' = ((e_i'), (w_i'), (f_i'))$ be another set of data. Write $\beta_{ij} = \langle e_i, w_j\rangle$. We write $\Lambda' = \mu_k(\Lambda)$ (or say they are \emph{$\mu_k$-equivalent}) if they satisfy the following conditions:
	\begin{itemize}
	\item $e_k' = -e_k$ and $w_k' = - w_k'$;
	\item if $i\neq k$ and $\beta_{ik} \geq  0$, $e_i' = e_i$ and $w'_i = w_i$;
	\item if $i\neq k$ and $\beta_{ik} \leq  0$, $e_i' = e_i-\langle e_i, r_k w_k\rangle e_k$ and $w'_i = w_i + \langle e_k, w_i \rangle r_kw_k$;
	\end{itemize}
	and if writing $f_i' = a_{i,0}'+ a_{i,1}'z^{w'_i}+ \cdots + a_{i,r_i}'z^{r_iw'_i}$,
	\begin{itemize}
	\item $a_{k,j}' = a_{k,r_k-j}$ for $j\in  [1, r_k]$;
	\item for $i\neq k, j\in [1, r_i]$,
	\begin{equation}
	\quad a'_{i,j}/a'_{i,0} = \begin{dcases}
	(a_{k,0})^{j\beta_{ki}}\cdot a_{i,j}/a_{i,0}\quad &\text{if $\beta_{ik}>0$}\\
	(a_{k,r_k})^{j\beta_{ki}} \cdot a_{i,j}/a_{i,0} \quad &\text{if $\beta_{ik}\leq 0$}.
\end{dcases}
	\end{equation}
	\end{itemize}
\end{definition}

We note that the mutation $\mu_k$ is not deterministic for the $(f_i)$ part, and is not involutive for the $((e_i),(w_i))$ part.

Applying \Cref{construction: log calabi yau of a seed} to $\Lambda' = \mu_k(\Lambda)$, we obtain another log Calabi--Yau variety $U_{\Lambda'}$. Note that both $U_\Lambda$ and $U_{\Lambda'}$ contain the torus $T_N$. Consider the birational morphism
\[
\mu_k\colon T_N\dasharrow T_N,\quad \mu_k^*(z^m) = z^m \cdot f_k^{-\langle m, e_k\rangle}.
\]

The following theorem is a generalization of the results in \cite[Section 3]{gross2013birational}.

\begin{theorem}\label{thm: mutation of toric model}
	The birational morphism $\mu_k$ extends to an isomorphism $\mu_k \colon U_\Lambda \rightarrow U_{\Lambda'}$ outside codimension two subsets if $\dim \bar V(f_k) \cap Z_i < \dim Z_i$ whenever $\langle e_i, w_k \rangle = 0$ for $i\in I$.
\end{theorem}

\begin{proof}
	We first make up some auxiliary varieties. Let $\Sigma^+ = \Sigma \cup \{\mathbb R_{\geq 0}e_k'\}$ and $\Sigma^- = \Sigma' \cup \{\mathbb R_{\geq 0}e_k\}$. We can blow up $X_{\Sigma^+}$ (resp. $X_{\Sigma^-}$) in the same way as we do so for $X_\sigma$ (resp. $X_{\Sigma'}$) to obtain $\tilde X_+$ (resp. $\tilde X_-$). Removing the strict transforms of the toric boundaries, we can still get $U_\Lambda$ and $U_{\Lambda'}$. Following Lemma 3.6 in \cite{gross2013birational}, we show that $\mu_k$ extends to an isomorphism (outside codimension two subsets) between $\tilde X_+$ and $\tilde X_-$, mapping the toric boundary of one to that of the other.
	
	Suppose we only blow up $X_{\Sigma^+}$ along $Z_k$ and $X_{\Sigma^-}$ along $Z_k'$. Then the blow-up $\tilde X_+$ has a covering of open subsets
	\begin{equation}\label{eq: a covering}
	\tilde X_+ = \tilde {\mathbb P}_+ \cup \left(\bigcup_{i\neq k} U_{i}\right)	
	\end{equation}
	where $\tilde {\mathbb P}_+$ is the blow-up along $Z_k$ of the toric variety of the fan $\{\mathbb R_{\geq 0} e_k', \mathbb R_{\geq 0} e_k\}$ and $U_i$ is the standard open toric chart corresponding to the ray $\mathbb R_{\geq 0} e_i$. Replacing $U_i$ with $U_i \setminus \bar V(f_k)$ for $i\neq k$, (\ref{eq: a covering}) is still a covering but up to codimension two (with $\bar V(f_k)\cap D_i$ missing). More precisely, $f_k$ is a regular function on $U_i$ if $\langle w_k, e_i \rangle \geq 0$. In this case, $\bar V(f_k) \cap D_i$ is just the zero locus of the restriction of $f_k$ on $D_i$, i.e. $V(f_k)\cap D_i$. As $z^{w_k}$ vanishes on $Z_i$ when $\langle w_k, e_i \rangle > 0$, $\bar V(f_k) \cap D_i = \emptyset$ since $f_k$ has non-zero constant term. When $\langle w_k, e_i \rangle < 0$, then $\bar V(f_k)\cap D_i = V(z^{-r_kw_k}f_k)\cap D_i$ where $z^{-r_kw_k}f_k = f'_k$ is a regular function on $U_i$. So $\bar V(f_k)\cap D_i$ is still empty since $f_k'$ has non-zero constant. Therefore we only fail to cover $\bar V(f_k)\cap D_i$ when $\langle w_k, e_i\rangle = 0$, which is a codimension two subset.
	
	By Lemma 3.2 of \cite{gross2013birational}, $\mu_k$ extends to a regular isomorphism from $\tilde{\mathbb P}_+$ to $\tilde {\mathbb P}_-$. Here $\tilde {\mathbb P}_-$ is the blow-up along $Z_k'$ of the toric variety defined by the fan $\{\mathbb R_{\geq 0} e_k', \mathbb R_{\geq 0} e_k\}$. We check that $\mu_k$ also extends to a regular isomorphism from $U_i\setminus \bar V(f_k)$ to $U_i'\setminus \bar V(f'_k)$. Note that these are affine schemes so we check that $\mu_k^*$ extends to an isomorphism between their rings of regular functions. There are two cases.
	\begin{enumerate}
		\item If $\langle e_i, w_k \rangle \geq 0$, then $e_i' = e_i$. Note that $f_k$ is a regular function on $U_i$ as well as on $U_i'$. Thus we have
			\[
				U_i\setminus \bar V(f_k) = U_i\setminus V(f_k)\quad \text{and}\quad U_i'\setminus \bar V(f_k') = U'_i \setminus V(f_k).
			\]
	For $\langle m, e_i\rangle \geq 0$, $z^m$ defines a regular function on $U_i'$ and
	\[
	\mu_k^*(z^m) = z^m f_k^{-\langle m, e_k\rangle}
	\]
	is a regular function on $U_i\setminus V(f_k)$.

		\item If $\langle e_i, w_k \rangle < 0$, then $e_i' = e_i - \langle e_i, r_kw_k \rangle e_k$. Instead of $f_k$, the function $f_k' = z^{-r_kw_k}f_k$ is a regular function on $U_i$ and $\bar V(f_k) = V(f'_k)$. For $\langle m, e_i'\rangle \geq 0$ and $z^m$ a regular function on $U_i'$, we have
		\[
		\mu_k^*(z^m) = z^m f_k^{-\langle m, e_k\rangle} = z^{m-r_kw_k\langle m, e_k\rangle} (f_k')^{-\langle m, e_k \rangle}.
		\]
	We check that $\langle  m-r_kw_k\langle m, e_k\rangle, e_i \rangle = \langle m-r_kw_k\langle m, e_k\rangle, e_i' + \langle e_i, r_kw_k\rangle e_k \rangle = \langle m, e_i'\rangle>0$. Thus $\mu_k^*(z^m)$ is a regular function on $U_i\setminus \bar V(f_k) = U_i\setminus V(f_k')$. 
	\end{enumerate}
Therefore $\mu_k^*$ is a morphism between regular functions. In all the cases above, one checks that  sending $z^m$ to $z^mf_k^{\langle m,e_k \rangle}$ is the inverse of $\mu_k^*$. Summarizing, we have so far proven that there is an isomorphism 
\[
\mu_k \colon U_+ \coloneqq \tilde{\mathbb P}_+ \cup \left(\bigcup_{i\neq k} U_i\setminus \bar V(f_k)\right) \longrightarrow U_-\coloneqq \tilde {\mathbb P}_- \cup \left(\bigcup_{i\neq k} U'_i\setminus \bar V(f'_k)\right) 
\]
extending the birational morphism $\mu_k$ between tori.

Now we analyze the impact of blowing up the hypersurfaces $Z_i$ (and $Z_i'$) for $i\neq k$. When $\langle w_k, e_i\rangle \neq 0$, as discussed $D_i \cap \bar V(f_k) = \emptyset$, so $Z_i\subset D_i$ is contained in $U_+$. Since $\langle w_k', e_i' \rangle = -\langle w_k, e_i \rangle \neq 0$, the same is true for $Z_i'$, i.e. $Z_i'\subset U_-$. We would like to show that $\mu_k(Z_i) = Z_i'$ when $\langle w_k, e_i\rangle \neq 0$. There are two cases.

\begin{enumerate}
	\item Suppose $\langle w_k, e_i\rangle >0$. In this case, $e_i' = e_i$ and $w_i' = w_i$. By definition $Z'_i = D'_i \cap V(f'_i) = V(z^{m_0}) \cap V(f'_i)\subset U'_i$ for some $m_0$ such that $\langle m_0, e_i'\rangle = 1$. Now we have $\mu_k^*(z^{m_0}) = z^{m_0}f_k^{-\langle m_0, e_k \rangle}$ and 
	\begin{align*}
	\mu_k^*(f_i') &= a_{i,0}' + a_{i,1}'z^{w_i}f_k^{-\langle w_i, e_k\rangle} + \cdots + a_{i,r_i}'z^{r_iw_i}f_k^{-\langle r_iw_i, e_k\rangle}.
	\end{align*}
Note that $f_k$ is invertible on $U_i\setminus\bar V(f_k)$ and restricts to constant $p_{k0}$ on $D_i$. So $V(\mu_k^*(z^{m_0}))$ is just the divisor $D_i$ and 
\[
\mu_k^*(f'_i)\vert_{D_i} = a_{i,0}' + a_{i,1}'a_{k,0}^{-\beta_{ki}}z^{w_i}+\cdots a_{i,r_i}'a_{k,0}^{-r_i\beta_{ki}}z^{r_iw_i} = \lambda\cdot f_i\vert_{D_i}. 
\]
for some non-zero $\lambda\in \Bbbk$ by the $\mu_k$-equivalence assumption on $\Lambda$ and $\Lambda'$. Therefore $\mu_k(Z_i) = Z_i'$.

	\item Suppse $\langle w_k, e_i \rangle <0$. In this case we have $e_i' = v_ i - \langle r_kw_k, e_i \rangle e_k$ and $w_i' = w_i + \langle w_i, e_k \rangle r_kw_k$. Still $Z_i' = V(z^{m_0}) \cap V(f_i')$. Now instead of $f_k$, the function $f_k' = z^{-r_kw_k}f_k$ is a regular function on $U_i$ and restricts to constant $a_{k,r_k}$ on $D_i$. First, $\mu_k^*(z^{m_0}) = z^{m_0 - \langle m_0, e_k\rangle r_kw_k} (f_k')^{\langle m_0, e_k\rangle}$. Since $f_k'$ is invertible on $U_i\setminus V(f_k)$, $V(\mu_k^*(z^{m_0})) = D_i$ as $\langle m_0 + \langle m_0, e_k\rangle r_kw_k, e_i\rangle = 1$. Secondly we have
	\begin{align*}
	\mu_k^*(f_i') &= a_{i,0}' + a_{i,1}'z^{w'_i}f_k^{-\langle w_i, e_k\rangle} + \cdots + a_{i,r_i}'z^{r_iw_i'}f_k^{-\langle r_iw_i, e_k\rangle}\\
	& = a_{i,0}' + a_{i,1}'z^{w'_i-\langle w_i, e_k\rangle r_kw_k}(f'_k)^{-\langle w_i, e_k\rangle} + \cdots + a_{i,r_i}'z^{r_iw'_i-\langle r_i w_i , e_k\rangle r_k w_k}(f'_k)^{-\langle r_i w_i, e_k\rangle}
	\end{align*}
	
	Hence
	\begin{align*}
		\mu_k^*(f_i')\mid_{D_i} &= a_{i,0}' + a_{i,1}' a_{k,r_k}^{-\beta_{ki}} z^{w_i}  + \cdots + a_{i,r_i}' a_{k,r_k}^{-r_i\beta_{ki}} z^{r_iw_i} = \lambda \cdot f_i\mid_{D_i}
	\end{align*}
	for some non-zero $\lambda\in \Bbbk$ again by the $\mu_k$-equivalence assumption. Therefore in this case we also have $\mu_k(Z_i) = Z_i'$.
\end{enumerate}

Finally we consider the case $\langle w_k, e_i \rangle =0$. The argument we need is exactly the same as in the last paragraph of the proof in \cite{gross2013birational}. By the assumption $\langle w_k , e_i \rangle = 0 \implies \langle w_i, e_k \rangle = 0$, so we have 
\[
	\mu_k^*(f_i') = f_i,
\]
and thus $\mu_k(Z_i) = Z_i'$. The problem is that $Z_i$ may not be fully contained in $D_i\setminus V(f_k)$, with $V(f_k)\cap Z_i$ missing. If $V(f_k)\cap Z_i$ contains a irreducible component of $Z_i$, then $U_\Lambda$ would contain the corresponding exceptional divisor while blowing up in $U_+$ does not. However the isomorphism $\mu_k\colon U_+\rightarrow U_-$ need not extend as isomorphism across this exceptional divisor. Now we need the further hypothesis $\dim V(f_k) \cap Z_i < \dim Z_i$ so that the missing part in the blow-up center is of at least codimension three in $U_i$. After blowing up the corresponding locus in $U_+$ and $U_+$, we have the following diagram
\[
\begin{tikzcd}
	\tilde U_+ \ar[r, "\mu_k"]\ar[d, "\pi"] & \tilde U_-\ar[d, "\pi"] \\
	U_+ \ar[r, "\mu_k"] & U_-
\end{tikzcd}
\]
where vertical arrows are blow-ups and horizontal arrows are genuine isomorphisms. Removing the strict transform of the toric boundary, we have immersions 
\[
\tilde U_+\setminus \tilde D \subset U_\Lambda \quad \text{and} \quad \tilde U_- \setminus \tilde D \subset U_{\Lambda'}
\]
missing codimension two loci. Summarizing, the birational map $\mu_k$ can be extended to an isomorphism $\mu_k\colon U_\Lambda \dashrightarrow U_{\Lambda'}$ outside sets of codimension two.
\end{proof}

We state a sufficient condition for the assumption in \Cref{thm: mutation of toric model},
\[
\forall \langle e_i, w_k\rangle = 0,\quad \dim \bar V(f_k) \cap Z_i < \dim Z_i,
\]
to hold.

\begin{definition}[cf. {\cite[Definition 1.4]{berenstein2005cluster}}]
	A toric model data $\Lambda = ((e_i),(w_i),(f_i))$ is said to be \emph{coprime} if the functions $f_i$ are pairwise coprime as elements in the ring $\Bbbk[M]$.
\end{definition}

\begin{corollary}
	The result in \Cref{thm: mutation of toric model} holds if $\Lambda$ is coprime.
\end{corollary}

\begin{proof}
	Note that $Z_i = \bar V(f_i) \cap D_i$. If needed, multiply some monomial $z^m$ to $f_i$ so that $\tilde f_i = z^mf_i$ is a regular function on $D_i$. Do the same to $f_k$ to get $\tilde f_k$. By the coprime condition on $\Lambda$, $\tilde f_i$ and $\tilde f_k$ are still coprime, so we have
	\[
	\dim V(\tilde f_k)\cap V(\tilde f_i) < \dim V(\tilde f_i)
	\]
	where the above subschemes are taken inside $D_i$.
\end{proof}

The following is an easy-to-check condition on $\Lambda$ for the coprimeness to hold.

\begin{lemma}\label{lemma: full rank implies coprime}
	If the vectors $w_i$ are linear independent, then $\Lambda$ is coprime.
\end{lemma}

\subsection{The upper bound}\label{subsection: upper bound}

Suppose we are given the data $\Lambda = ((e_i), (w_i), (f_i))$. Assume that $\Lambda$ is $i$-mutable for any $i\in I$. For $i\in I$, let $T_N^{(i)}$ be a copy of the torus $T_N$. Then we have birational maps for each $i\in I$,
\[
\mu_i\colon T_N \dasharrow T_N^{(i)},\quad \mu_i^*(z^m) = z^m f_i^{-\langle e_i, m\rangle}.
\]
We glue the $|I| + 1$ tori along the maps $\mu_i$ to obtain a scheme $X_\Lambda$.

In previous section, we know that not only the torus $T_N$, $U_\Lambda$ also contains the torus $T_N^{(i)}$, that is, we have the following diagram for every $i\in I$ : \[\begin{tikzcd}
	T_N \ar[r, "\mu_i", dashed]\ar[d, hook] &T_N^{(i)}\ar[ld, hook]\\
	U_\Sigma &
\end{tikzcd}\]
These diagrams determine a unique morphism $\psi \colon X_\Lambda \rightarrow U_\Lambda$.

\begin{lemma}[{\cite[Lemma 3.5]{gross2013birational}}]\label{lemma: GHK lemma 3.5}
	The morphism $\psi \colon X_\Lambda \rightarrow U_\Lambda$ satisfies the following properties
	\begin{enumerate}
		\item If $\dim Z_i \cap Z_j < \dim Z_i$ for all $i\neq j$, then $\psi$ is an isomorphism outside a set of codimension at least two.
		\item If $Z_i \cap Z_j = \emptyset$ for all $i\neq j$, then $\psi$ is an open immersion. In particular, in this case, $X_\Lambda$ is separated.
	\end{enumerate}
\end{lemma}

In the $\mathcal A$-cluster case to be explained later, the variety $X_\Lambda$ may be named \emph{the upper bound} according to \cite{fomin2007cluster}.

\subsection{Toric models for cluster varieties}\label{subsection: toric model for cluster}
In this section, we realize generalize cluster varieties as log Calabi--Yau varieties utilizing \Cref{construction: log calabi yau of a seed}.

\subsubsection{\textbf{$\mathcal A$-cluster cases}}

Suppose we have fixed data $\Gamma$ and an $\mathcal A$-seed with coefficients $\mathbf s = (\mathbf e, \mathbf p)$. We further choose an evaluation $\lambda\colon \mathbb P \rightarrow \Bbbk^*$. This amounts to pick a $\Bbbk$-point of $\Spec(\Bbbk \mathbb P)$. These lead to the generalized $\mathcal A$-cluster variety $\mathcal A_{\mathbf s, \lambda}$ with special coefficients.

Meanwhile consider the toric model data 
\[
	\Lambda(\mathbf s,\lambda) \coloneqq ((e_i)_{i\in I_\mathrm{uf}}, (w_i)_{i\in I_\mathrm{uf}}, (f_i)_{i\in I_\mathrm{uf}})
\]
defined as follows. The vectors $(e_i)_{i\in I_\mathrm {uf}}$ are taken from the seed $\mathbf s$. Recall that we have the exchange matrix $B = (b_{ij})$ where $b_{ij} \coloneqq \omega(e_i, d_je_j)$. Write $\beta_{ij} = b_{ij}/r_j$. Note that $\{e_i\mid i\in I\}$ form a basis of the lattice $N$ and we denote by $e_i^*$ the dual basis of $M$. Then define 
\[
w_i  \coloneqq \omega (- , d_ie_i/r_i)  = \sum_{j\in I} \beta_{ij}e_i^*\in M,\quad
f_i \coloneqq \lambda\left( \theta[\mathbf p_i](z^{w_i},1) \right) \in \Bbbk[M].
\]
Then \Cref{construction: log calabi yau of a seed} applies to the toric model data $\Lambda(\mathbf s, \lambda)$, and thus there is the associated log Calabi--Yau variety $U_{\Lambda(\mathbf s, \lambda)}$. Recall that we also have the scheme $X_{\Lambda(\mathbf s,\lambda)}$ obtained by glueing $n+1$ copies of the torus $T_N$ as in \Cref{subsection: upper bound}. We call $X_{\Lambda(\mathbf s,\lambda)}$ \emph{the upper bound} for $(\mathbf s, \lambda)$, which by definition is an open subscheme of $\mathcal A_{\mathbf s,\lambda}$.

The following lemma is easy to verify by direct computations.
\begin{lemma}\label{lemma: mutation of A data}
	We have $\mu_k(\Lambda(\mathbf s,\lambda)) = \Lambda(\mu_k(\mathbf s),\lambda)$ in the sense of \Cref{def: mutation equivalent data}. The later $\mu_k$ is the mutation of an $\mathcal A$-seed with coefficients.
\end{lemma}

\begin{proposition}\label{prop: mutation of A cluster variety}
	 We have
	\begin{enumerate}
		\item the morphism $\psi\colon X_{\Lambda(\mathbf s,\lambda)} \rightarrow U_{\Lambda(\mathbf s,\lambda)}$ is an open immersion with image an open subset whose complement has codimension at least two;
		\item the birational map $\mu_k\colon U_{\Lambda(\mathbf s,\lambda)} \dasharrow U_{\Lambda(\mu_k(\mathbf s),\lambda)}$ is an isomorphism outside codimension two in each of the listed situations \begin{itemize}
			\item [A.] the functions $f_i$ have general coefficients;
			
			\item [B.] the seed $\mathbf s$ is mutation equivalent to one with principal coefficients, and $\lambda\in (\Bbbk^*)^{|I'|}$ is general enough.
			
		\end{itemize} 
	\end{enumerate}
\end{proposition}

\begin{proof}
	(1) follows from \Cref{lemma: GHK lemma 3.5}, part (2) - as we only need to check the hypothesis $Z_i \cap Z_j = \emptyset$ for all $i\neq j$. In fact, in $\mathcal A$-cluster case, since $e_i \neq e_j$, we have $T_{N/\langle e_i \rangle} \cap T_{N/\langle e_j\rangle} = \emptyset$ for all $i\neq j$ where $T_{N/\langle e_i \rangle}$ is viewed as the dense torus contained in the divisor $D_i$. As $Z_i$ is a closed subset of $T_{N/\langle e_i \rangle}$, the hypothesis holds.
	
	(2) follows from \Cref{thm: mutation of toric model}. We need to check that whenever $\langle e_i, u_k\rangle = 0$, \[\dim \bar V(f_k) \cap \bar V(f_i) \cap D_i < \dim \bar V(f_i) \cap D_i. \]
	A sufficient condition is the functions $f_i$ being coprime. Note that for $i\in I$,
	\[
	f_i = \prod_{j = 1}^{r_i}\left( \lambda(p_{i,j}^+) z^{w_i} + \lambda(p_{i,j}^-)\right).
	\]
	When these $f_i$ have general coefficients (case A), they are coprime. In case $B$, one may replace $f_i$ by
	\[
	\tilde f_i = \prod_{j = 1}^{r_i}\left( \lambda(p_{i,j}) z^{w_i} + 1\right).
	\]
	Since the elements $p_{i,j}$ for $i\in I$ and $j\in [1, r_i]$ form a $\mathbb Z$ basis in $\mathbb P$ (by \Cref{lemma: generalized c vectors form basis}) when $\mathbf s$ is mutation equivalent to one with principal coefficients, these $\tilde f_i$ are coprime as long as $\lambda$ is general.
	\end{proof}

\begin{remark}\label{rmk: Laurent phenomenon from geometry}
Suppose we are in the situation of case B of \Cref{prop: mutation of A cluster variety}, (2). Then we have isomorphisms of the rings of regular functions
\[
\Bbbk[X_{\Lambda(\mathbf s, \lambda)}] \cong \Bbbk[U_{\Lambda(\mathbf s, \lambda)}] \cong \Bbbk[U_{\Lambda(\mu_k(\mathbf s), \lambda)}].
\]
The equality then extends to any seed $\mathbf s_v$ that is mutation equivalent to $\mathbf s$. It then follows that they are all isomorphic to the upper cluster algebra
\[
\mathscr A(\mathbf s, \lambda) = \Bbbk [\mathcal A_{\mathbf s, \lambda}].
\]
The cluster variables in seed $\mathbf s$ are $x_{i,\mathbf s} \coloneqq z^{e_i^*}$. Each $x_{i,\mathbf s}$ extends to a regular function on the toric variety $X_\Sigma$ corresponding to the toric model data $\Lambda(\mathbf s, \lambda)$. Then $x_{i,\mathbf s}$ pulls back to the blow-up $\tilde X_\Sigma$ and restricts to a regular function on the open subvariety $U_{\Lambda(\mathbf s,\lambda)}$. It follows from (2) of \Cref{prop: mutation of A cluster variety} that $x_{i,\mathbf s}$ is also a regular function on $X_{\Lambda(\mathbf s_v, \lambda)}$ and in particular is a Laurent polynomial if restricted to $T_{N,\mathbf s_v}$. This explains the generalized Laurent phenomenon \Cref{thm: generalized laurent}, which was observed in \cite{gross2013birational} for the ordinary case.	
\end{remark}

\subsubsection{\textbf{$\mathcal X$-cluster cases}}\label{section: x toric model}
Suppose we have fixed data $\Gamma$ and an $\mathcal X$-seed with coefficients $\mathbf s = (\mathbf e,\mathbf q)$. Let us make the assumption that for any $j\in I_\mathrm{uf}$, $$r_j = \mathrm{gcd}(b_{ij}, i\in I).$$ This is equivalent to say that each $w_j$ for $j\in I_\mathrm{uf}$ is primitive as an element of $M = \Hom(N, \mathbb Z)$. Switching the roles of $(e_i)$ and $(w_i)$, we obtain the toric model data 
\[
	\Omega (\mathbf s,\lambda) = ((-w_i), (e_i), (g_i))
\]
for $M$ instead of $N$, where 
\[
	g_i \coloneqq \lambda(\theta[\mathbf q_i](z^{e_i},1))\in \Bbbk[N]
\]
with some chosen evaluation $\lambda$. Since the matrix $B$ is skew-symmetrizable, $\Omega(\mathbf s, \lambda)$ is $k$-mutable for any $k\in I_\mathrm{uf}$.

\begin{lemma}
	The assumption that $r_j = \operatorname{gcd}(b_{ij},i\in I)$ is invariant under mutations.
\end{lemma}
\begin{proof}
	This is because if the $j$-th column of $B$ is divisible by $r_j$ then the same is true for the matrix $\mu_k(B) = (b'_{ij})$. Thus we have
	\[
	\mathrm{gcd}(b_{ij}, i\in I) = \mathrm{gcd}(b'_{ij}, i\in I)
	\]
	as $\mu_k$ is involutive on $B$.
\end{proof}

The above lemma shows that we have well-defined data $\Omega(\mu_k(\mathbf s), \lambda)$. 
\begin{lemma}
	We have $\mu_k(\Omega(\mathbf s,\lambda)) = \Omega(\mu_k(\mathbf s), \lambda)$, where the later $\mu_k$ is the mutation for an $\mathcal X$-seed with coefficients.
\end{lemma}

\begin{proof}
	This lemma is analogous to \Cref{lemma: mutation of A data} and is also easy to check. However, to show that the carefully chosen signs and conventions are the correct ones, we record some details here.
	
	In the notations of \Cref{def: mutation equivalent data}, for the data $\Omega(\mathbf s, \lambda)$, we take $e_i = -w_i$ and $w_i = e_i$. So after the mutation $\mu_k$ in sense of \Cref{def: mutation equivalent data}, for $i\neq k$
	\[
	(-w_i)' = \begin{cases}
		-w_i - \langle (-w_i), r_ke_k \rangle (-w_k) \quad &\text{if $\langle -w_i, e_k \rangle \leq 0$}\\
		-w_i &\text{if $\langle -w_i, e_k\rangle >0$}.
	\end{cases}
	\]
	Note that the two conditions are equivalent to $\beta_{ik}\leq 0$ and $\beta_{ik}>0$ respectively. And in these two cases, we have
	\[
	(-w_i)' = -w_i - \langle e_k, w_i\rangle r_kw_k\quad \text{and} \quad -w_i
	\]
	respectively. This is exactly $-w_i'$ for $w_i' = \omega (-, d_ie_i'/r_i)$ from the seed $\mu_k(\mathbf s)$. Similarly, one checks that the $\mathbf e$ part is also compatible with mutations.
	
	As for coefficients, for the data $\Omega(\mu_k(\mathbf s), \lambda)$, we have
	\[
	g_i'(u,v) = \lambda(\theta[\mathbf q_i'](u,v)).
	\]
	Here $q_{i,j}'$ is obtained from $\mathcal X$-type mutations for coefficients (see \Cref{def: mutation for x-seed}) which coincides with \Cref{def: mutation equivalent data}.
\end{proof}

Recall that $X_{\Omega(\mathbf s,\lambda)}$ is the upper bound for $\Omega(\mathbf s, \lambda)$ as defined in \Cref{subsection: upper bound}.

\begin{proposition}
	 We have (for the $\mathcal X$-type constructions)
	\begin{enumerate}
		\item the morphism $\psi\colon X_{\Omega(\mathbf s,\lambda)} \rightarrow U_{\Omega(\mathbf s,\lambda)}$ is an open immersion with image being an open subset whose complement has codimension at least two;
		\item the birational map $\mu_k\colon U_{\Omega(\mathbf s,\lambda)} \dasharrow U_{\Omega(\mu_k(\mathbf s),\lambda)}$ is an isomorphism outside codimension two subsets.
	\end{enumerate}
\end{proposition}

\begin{proof}
	The proof of $(1)$ is completely analogous to the one of $(1)$ of \Cref{prop: mutation of A cluster variety}. For $(2)$, it follows from that for any $\mathcal X$-seed $\mathbf s$, the data $\Omega(\mathbf s,\lambda)$ is always coprime by \Cref{lemma: full rank implies coprime} as the vectors $e_i$ form a basis of $N$.
\end{proof}

\section{Scattering diagrams}

This section deals with scattering diagrams. Our main objects of study \emph{generalized cluster scattering diagrams} will be defined in \Cref{subsection: generalized cluster sd}.

\subsection{The tropical vertex}\label{section: the tropical vertex}

We start with a more general setup of scattering diagrams as in \cite[Section 5.1.1]{arguz2020higher}. Let $N$ be a lattice of finite rank, $M = \Hom_\mathbb Z( N, \mathbb Z)$ and $M_\mathbb R = M\otimes _\mathbb Z \mathbb R$. Let $P$ be a monoid with a monoid map $r\colon P\rightarrow M$. Denote by $P^\times$ the groups of units of $P$ and let $\mathfrak m_P = P\setminus P^\times$. An ideal of the monoid $P$ induces a monomial ideal of the ring $\Bbbk[P]$ where $\Bbbk$ is a ground field. So we use the same letter to denote both. For any monomial ideal $I\subset \Bbbk[P]$, define the ring
\[
R_I \coloneqq \Bbbk[P]/I.
\]
Denote by $\widehat{\Bbbk[P]}$ the completion of $\Bbbk[P]/\mathfrak m_P^n$ for $n\in \mathbb N$.

For $I$ such that its radical $\sqrt I$ is equal to $\mathfrak m_P$ (e.g. $I = \mathfrak m_P^n$ for some $n\in \mathbb N$), define the \emph{module of log derivations} $\Theta(R_I) \coloneqq R_I \otimes_\mathbb Z N$ as follows.

If we write the element $z^p\otimes n$ as $z^p\partial_n$ for $p\in P$ and $n\in N$, then it acts on $R_I$ by \[ z^p\partial_n(z^{p'}) = \langle n , r(p') \rangle z^{p+p'},\quad p'\in P.\] Then the submodule $\mathfrak m_P\Theta(R_I)$ is a Lie algebra with the commutator bracket \[ [z^{p_1}\partial_{n_1}, z^{p_2}\partial_{n_2}] = z^{p_1+p_2}\partial_{ \langle r(p_2), n_1 \rangle n_2 - \langle r(p_1), n_2 \rangle n_1}. \] Taking exponential of elements in this Lie algebra, we get group elements in $\mathrm{Aut}(R_I)$. There is a nilpotent Lie subalgebra of $\mathfrak m_P\Theta(R_I)$ defined by
\[
\mathfrak v_I \coloneqq \bigoplus_{m\in P\setminus I, r(m)\neq 0} z^m \left( \Bbbk \otimes r(m)^\perp \right).
\]
Since it is nilpotent, this Lie subalgebra (as a set) is in bijection with the corresponding algebraic group $\mathbb V_I \coloneqq \exp (\mathfrak v_I)\subset \mathrm{Aut}(R_I)$. Taking the projective limit with respect to the ideals $\mathfrak m_P^n$ for $n\in \mathbb N$, we get a pro-unipotent group $\widehat {\mathbb V}$, which is in bijection with the pro-nilpotent Lie algebra $\hat {\mathfrak v} \coloneqq \underset{\longleftarrow }{\lim} \mathfrak v_{\mathfrak m_P^n}$. The group $\widehat {\mathbb V}$ is called the \emph{higher dimensional tropical vertex group}, acts by automorphisms on $\widehat{\Bbbk[P]}$. We also denote (without completion)
\[
\mathfrak v \coloneqq \bigoplus_{m\in P, r(m)\neq 0} z^m(\Bbbk\otimes r(m)^\perp).
\]

\begin{definition}
	A \emph{scattering diagram} in $M_\mathbb R$ over $R_I$ is a finite set $\mathfrak D$ of \emph{walls} where each \emph{wall} $(\mathfrak d, f_\mathfrak d)$ is a rational polyhedral cone $\mathfrak d\subset M_\mathbb R$ of codimension one along with an attached element called \emph{wall-crossing function}
	\[ f_\mathfrak d = \sum_{\substack{m\in P\setminus I \\ r(m)\in \Lambda_\mathfrak d}} c_m z^m \in R_I, \]
	where $\Lambda_{\mathfrak d}\subset M$ is the integral tangent space of any point in $\mathfrak d$, i.e. $\Lambda_\mathfrak d = M \cap \mathbb R\langle \mathfrak d\rangle$. We require that $f_\mathfrak d \equiv 1 \mod \mathfrak m_P$.
\end{definition}

\begin{remark}\label{rmk: how to define wall-crossing automorphism}
	Upon choosing a generator $n_0$ of $\Lambda_{\mathfrak d}^\perp \cap N$, the wall-crossing function $f_\mathfrak d$ induces an element in $\mathbb V_I\subset \operatorname{Aut}(R_I)$ by the action $$z^p\mapsto z^p f_{\mathfrak d}^{\langle r(p), n_0\rangle}.$$
	
	So this wall-crossing automorphism depends on how one crosses the wall.
	One may view that this wall-crossing automorphism depends on the direction in which one transversally crosses the wall. With $n_0$ chosen, such an automorphism can be equivalently represented by the corresponding Lie algebra element $\log (f_{\mathfrak d})\partial_{n_0}\in \mathfrak v_I$.
\end{remark}

Let $\mathrm{Supp}(\mathfrak D)$ be the union of all walls in $\mathfrak D$. Let $\mathrm{Sing}(\mathfrak D)$ be the union of at least codimension two intersections of every pair of walls and the boundary of every wall. Let $\gamma\colon [0,1]\rightarrow M_\mathbb R$ be a piecewise smooth proper map such that the end points $\gamma(0)$ and $\gamma(1)$ avoid $\mathrm{Supp}(\mathfrak D)$ and whose image is disjoint from $\mathrm{Sing}(\mathfrak D)$. We also assume that $\gamma$ meets walls transversally. 

Suppose that $\gamma$ crosses walls $\mathfrak d_1,\dots, \mathfrak d_s$ in $\mathfrak D$ at times
\[ 0<t_1\leq t_2\leq \cdots \leq t_s < 1. \]
These numbers $t_i$ are obtained by considering the finite set $\gamma^{-1}(\mathrm{Supp}(\mathfrak D))\subset [0,1]$ as $\gamma$ is proper. It is possible that $t_i = t_j$ as walls may overlap. Suppose $\gamma$ crosses a wall $(\mathfrak d, f_\mathfrak d)$ at time $t$. Denote by $\xi_{\gamma, \mathfrak d}$ the element in $\mathbb V_I$ with the action
\[ z^p\mapsto z^p f_\mathfrak d^{\langle r(p), n_0\rangle}, \quad p\in P\setminus I \]
where $n_0$ is chosen such that $\langle n_0, \gamma'(t_i) \rangle >0.$

\begin{definition}\label{def: path-ordered product}
	We define the \emph{path-ordered product} of $\gamma$ in $\mathfrak D$ to be the element
	\[ \mathfrak p_{\gamma, \mathfrak D} \coloneqq \xi_{\gamma, \mathfrak d_s} \xi_{\gamma, \mathfrak d_{s-1}} \cdots \xi_{\gamma, \mathfrak d_1} \in \mathbb V_I.\]
\end{definition}

\begin{definition}
A scattering diagram $\mathfrak D$ over $R_I$ is \emph{consistent} if the path-ordered product $\mathfrak p_{\gamma, \mathfrak D}$ only depends on the endpoints $\gamma(0)$ and $\gamma(1)$ for any path $\gamma\colon [0,1]\rightarrow M_\mathbb R$ for which $\mathfrak p_{\gamma, \mathfrak D}$ is well-defined.
\end{definition}

Recall that we have the completed algebra $\widehat{\Bbbk[P]} \coloneqq \operatorname{\underset{\longleftarrow}{\lim}} R_{\mathfrak m_P^k}$. For an element $f\in \widehat{\Bbbk [P]}$, denote by $f^{<k}$ its projection in $R_{\mathfrak m_P^k}$.

\begin{definition}\label{def: tropical vertex sd}
	A \emph{scattering diagram} in $M_\mathbb R$ over $\widehat{\Bbbk[P]}$ is a (possibly infinite) set $\mathfrak D$ of walls $(\mathfrak d, f_\mathfrak d)$ with $\mathfrak d$ a rational polyhedral cone of codimension one and the wall-crossing function \[ f_\mathfrak d = \sum_{\substack{m\in P \\ r(m)\in \Lambda_\mathfrak d}} c_mz^m\in \widehat {\Bbbk[P]},\] such that modulo the ideal $\mathfrak m_p^n$, the collection $\mathfrak D^{<n} \coloneqq \{(\mathfrak d, f_\mathfrak d^{<n}) \}$ is a scattering diagram over $R_{\mathfrak m_P^n}$. A scattering diagram $\mathfrak D$ is consistent if $\mathfrak D^{<n}$ is consistent for any $n\in \mathbb N$.
\end{definition}

The path-ordered product for $\mathfrak D$ over $\widehat{\Bbbk[P]}$ is defined through the projective limit of path-ordered products for $\mathfrak D^{<n}$:
\[
\mathfrak p_{\gamma, \mathfrak D} \coloneqq \lim_{\longleftarrow} \mathfrak p_{\gamma, \mathfrak D^{<n}}\in  \widehat{\mathbb V} \subset \mathrm{Aut}(\widehat{\Bbbk[P]}).
\]

\begin{definition}
	We say two scattering diagrams $\mathfrak D$ and $\mathfrak D'$ (over the same algebra) are \emph{equivalent} if for any $\gamma$, we have $\mathfrak p_{\gamma, \mathfrak D} = \mathfrak p_{\gamma, \mathfrak D'}$ whenever both path-ordered products are well-defined. 
\end{definition}

\begin{definition}\label{def: direction of a wall}
We say a wall $\mathfrak d$ has \emph{direction} $m_0$ for some $m_0\in M$ if the attached wall-crossing function $f_\mathfrak d$ only contains monomials $z^p$ such that $r(p) = -km_0$ for some $k\in \mathbb N$. A wall $(\mathfrak d, f_\mathfrak d)$ with direction $m_0$ is called \emph{incoming} if $\mathfrak d = \mathfrak d - \mathbb R_{\geq 0} m_0$.
\end{definition}

Next we explain how to assign a scattering diagram to an $\mathcal X$-type toric model. We are actually in a particular situation within the more general framework of \cite{arguz2020higher}, which works for any log Calabi--Yau variety obtained from blowing-up a toric variety along hypersurfaces in the toric boundary. 

Let $\mathbf s = (\mathbf e, \mathbf q)$ be an $\mathcal X$-seed with principal coefficients for some fixed data $\Gamma$. We assume that $N_\mathrm{uf} = N$ to avoid frozen directions. As usual, write $\mathbf e = (e_i)$. We assume the condition that $r_j = \operatorname{gcd}(b_{ij}\mid i\in I)$ for any $j\in I$. This assumption implies any $w_i \coloneqq \frac{d_i}{r_i}\omega (-, e_i)\in M$ is primitive. Recall that we have used the fan
\[
\Sigma_0 \coloneqq \{0\}\cup  \{-\mathbb R_{\geq 0} w_i\}
\]
to describe the toric model of $U = U_{\Omega(\mathbf s, \lambda)}$. The functions (in the data $\Omega(\mathbf s, \lambda)$ to define $U$) are then 
\[ 
	g_i = \prod_{j = 1}^{r_i}\left( 1 + \lambda_{ij}z^{e_i}\right) \in \Bbbk[N].
\]
We pick a complete fan $\Sigma$ in $M_\mathbb R$ containing $\Sigma_0$. For example, we may take a refinement of (the cone complex induced by) the hyperplane arrangement $\{e_i^\perp \mid i\in I\}$. Let $X_\Sigma$ be the corresponding (complete) toric variety, with $D_i$ being the boundary toric divisor corresponding to the ray $-\mathbb R_{\geq 0}w_i$. Let $H = \cup_{i} H_i$ where $$H_i = \bigcup_{j\in [1, r_i]} H_{ij} \coloneqq \bigcup_{j \in [1, r_i]} \overline V(1+\lambda_{ij}z^{e_i}) \cap D_i$$ which is a union of disjoint hypersurfaces in $D_i$ (as the coefficients $\lambda_{ij}\in \Bbbk^*$ are general). These hypersurfaces are exactly where we blow up $X_\Sigma$ to obtain the log Calabi--Yau variety $U_{\Omega(\mathbf s, \lambda)}$. 

Take the monoid
\[
P \coloneqq M \oplus \prod_{i\in I} \mathbb N^{r_i},
\]
with the natural projection $r\colon P\rightarrow M$.
We write multiplicatively $t_{i,1}, t_{i, 2},\dots t_{i,r_i}$ for the generators of $\mathbb N^{r_i}$. For each ray $\rho_i \coloneqq -\mathbb R_{\geq 0} w_i$ and $H_{ij}$, there is a finite scattering diagram $\mathfrak D_{ij}$ called a \emph{widget} from a certain \emph{tropical hypersurface} \cite[Definition 5.3 and Section 5.1.3]{arguz2020higher}. In our case, they are given by

\begin{lemma}\label{lemma: widgets}
	The widget $\mathfrak D_{ij}$ consists of all codimension one cones of the fan $\Sigma$ contained in the hyperplane $e_i^\perp$ containing $\rho_i$, with the same wall-crossing function $(1 + t_{i,j} z^{w_i})$. In other words, we have
\[
\mathfrak D_{ij} = \{(\sigma, 1 + t_{i,j}z^{w_i}) \mid \sigma\in \Sigma,\ \dim \sigma = n-1,\ \sigma\subset e_i^\perp,\ \rho_i\subset \sigma\}.
\]
\end{lemma}

\begin{proof}
	By definition \cite[Definition 5.3 and Section 5.1.3]{arguz2020higher}, $\mathfrak D_{ij}$ consists of walls $(\sigma , (1+t_{i,j}z^{w_i})^{\omega_\sigma})$ where $\sigma$ runs through all codimension one cones in $\Sigma$ containing $\rho_i$ and $\omega_\sigma = H_{ij}\cdot D_{\sigma}$ is the intersection number computed in the divisor $D_i$. Here $D_\sigma$ is the one dimensional toric stratum in $D_i$ corresponding to $\sigma$. Note that if $e_i\notin \sigma^\perp$, then $z^{e_i}$ or $z^{-e_i}$ vanishes along $D_\sigma$. So $H_{ij} = \overline V(1 + \lambda_{ij}z^{e_i})$ does not intersect $D_\sigma$ and thus $\omega_\sigma = 0$. If $\sigma\subset e_i^\perp$, as $e_i$ is primitive, the intersection is at the point $z^{e_i} = -1/\lambda_{ij}$ where $z^{e_i}$ can be viewed as the coordinate on $D_\sigma$. Thus the multiplicity $\omega_\sigma$ is $1$.
\end{proof}

Note that by \Cref{def: direction of a wall} every wall in $\mathfrak D_{ij}$ is incoming since $-w_i$ is contained in every $\sigma$.

\begin{theorem}[{\cite[Theorem 5.6 and Section 5.1.3]{arguz2020higher}}]\label{thm: tropical vertex sd}
	Consider the scattering diagram (with only incoming walls)
	\[ \mathfrak D_{(X_\Sigma, H), \mathrm{in}} \coloneqq \bigcup_{i\in I}\bigcup_{j\in [1, r_i]} \mathfrak D_{ij}. \] There exists a unique (up to equivalence) consistent scattering diagram $\mathfrak D_{(X_\Sigma, H)}$ over $\widehat{\Bbbk[P]}$ containing $\mathfrak D_{(X_\Sigma, H),\mathrm{in}}$ such that $\mathfrak D_{(X_\Sigma, H)} \setminus \mathfrak D_{(X_\Sigma, H), \mathrm{in}} $ consists only of non-incoming walls.
\end{theorem}

\subsection{Generalized cluster scattering diagrams}\label{subsection: generalized cluster sd}

Instead of applying \Cref{thm: tropical vertex sd} to $(X_\Sigma, H)$, there is another way to obtain the same scattering diagram by generalizing the construction of cluster scattering diagrams in \cite{gross2018canonical}.

Given fixed data $\Gamma$ and an $\mathcal A$-seed $\mathbf s = (\mathbf e, \mathbf p)$ with principal coefficients, we are going to define the \emph{generalized cluster scattering diagram} $\mathfrak D_\mathbf s$. 

Recall that we have the semifield $\mathbb P = \mathrm{Trop}(\mathbf p)$, isomorphic to $\prod_{i\in I}\mathbb Z^{r_i}$ as an abelian group. Let $P = P_\mathbf s$ as before be $M\oplus \prod_{i \in I} \mathbb N^{r_i}$, but regarded as a submonoid of $M\oplus \mathbb P$ generated by $M$ and $\mathbf p$. There is a submonoid $P^\oplus = P^\oplus_\mathbf s \subset P$ generated by elements
\[
	\{(w_i, p_{i,j})\mid i\in I, j\in [1, r_i]\}.
\]
One could take the completion of $P^\oplus$ with respect to the ideal $P^+ \coloneqq P^\oplus \setminus\{0\}$, resulting $\widehat{\Bbbk[P^\oplus]}\subset \widehat{\Bbbk[P]}$. In $N$, there is a submonoid $N^\oplus_\mathbf s = N^\oplus$ generated by $\{e_i\mid i\in I\}$. Denote $N^+ = N^\oplus \setminus \{0\}$. We also consider the monoid map $$\pi\colon P^\oplus \rightarrow N^\oplus,\quad (w_i, p_{i,j})\mapsto e_i. $$

Let $n = \sum_{i\in I} \alpha_i e_i\in N$. Define 
\[
\bar n  \coloneqq \sum_{i\in I} \alpha_i \cdot \frac{d_i}{r_i}e_i\in N_\mathbb R.
\]
These $\bar n$ form a sublattice $\overline N$ of $N_\mathbb R$ isomorphic to $N$. We have the similar notion $\overline N^+$, the monoid generated by $\bar e_i$.

There is a subspace $\mathfrak g$ of the tropical vertex lie algebra $\mathfrak v$ defined as
\[ \mathfrak g = \mathfrak g_\mathbf s \coloneqq \bigoplus_{n\in N^+} \mathfrak g_n, \quad  \mathfrak g_n  \coloneqq  \bigoplus_{\substack{\pi(p) = n\\p\in P^+}} z^p \cdot \left( \Bbbk\otimes \bar n \right). \]
\begin{lemma}
	The subspace $\mathfrak g$ is an $N^+$-graded Lie subalgebra of $\mathfrak v$.
\end{lemma}
\begin{proof}
	For any $n = \sum_{i\in I} \alpha_ie_i\in N^+$, consider the elements
\[
\prod_{i,j} p_{i,j}^{c_{i,j}} \cdot z^{p^*(n)}
\]	
such that $\sum_{j\in [1, r_i]} c_{i,j} = \alpha_i$ and	
\[
p^*(n) \coloneqq \omega (- , \bar n)  = \sum_{i\in I} \alpha_i\omega(-, d_i e_i/r_i) = \sum_{i\in I} \alpha_i w_i.
\]
Those elements form a basis of the vector space $\mathfrak g_n$. We check that for two such elements
\begin{align*}
	\left[p_1 z^{p^*(n_1)} \partial_{\bar n_1}, p_2 z^{p^*(n_2)} \partial_{\bar n_2}\right] &= p_1p_2 \cdot z^{p^*(n_1+n_2)}\partial_{\omega(\bar n_1, \bar n_2)\bar n_2 - \omega(\bar n_2, \bar n_1)\bar n_1}\\
	& = \omega(\bar n_1, \bar n_2) p_1p_2 \cdot z^{p^*(n_1+n_2)}\partial_{\bar n_1 + \bar n_2}\in \mathfrak g_{n_1+n_2}.
\end{align*}
\end{proof}

\begin{remark}
	One may also view the above Lie algebra $\mathfrak g$ as being $\overline N^+$-graded where both $\overline N$ and $N$ are sublattices of $N_\mathbb R$. When later considering a scattering diagram $\mathfrak D$ over an $\overline N^+$-graded Lie algebra $\mathfrak g$ (instead of $N^+$-graded), the walls live in $M_\mathbb R$ with integral normal vectors in $\overline N^+$.
\end{remark}

Consider the ideals $(N^+)^k\subset N^+$ for $k\geq 1$. These correspond to the monomial ideals $(P^+)^k$. Then we have quotient Lie algebras (and their corresponding groups $G^{<k}$)
\[
\mathfrak g^{<k} \coloneqq \mathfrak g/\mathfrak g_{(N^+)^k} = \bigoplus_{n\in N^+\setminus (N^+)^k} \mathfrak g_n,
\]
and their projective limits
\[
\hat {\mathfrak g} = \prod_{n\in N^+} \mathfrak g_n\quad \text{and}\quad G \coloneqq \exp(\hat {\mathfrak g}). 
\]
The group $G^{<k}$ acts on $\Bbbk[P^\oplus]/(P^+)^k$ by automorphisms as in \Cref{rmk: how to define wall-crossing automorphism}.

For $n_0\in N^+$ primitive, we define as in \cite{gross2018canonical} a Lie algebra (and its corresponding pro-unipotent group)
\[ \mathfrak g_{n_0}^\parallel \coloneqq \bigoplus_{k>0} \mathfrak g_{k\cdot n_0} \subset \mathfrak g \quad \text{and} \quad G_{n_0}^\parallel \coloneqq \exp (\hat{\mathfrak g}_{n_0}^\parallel)\subset G. \]
There is a general framework for scattering diagrams over an $N^+$-graded Lie algebra (as opposed to the tropical vertex case); see \cite[Section 2.1]{kontsevich2014wall} and \cite[Section 1.1]{gross2018canonical}. In this case, one could make use of an existence-and-uniqueness theorem of \cite{kontsevich2014wall} (see also \cite[Theorem 1.21]{gross2018canonical}) to obtain a consistent scattering diagram with certain prescribed \emph{incoming data}. The cluster scattering diagram of \cite{gross2018canonical} can be defined this way, which we will extend to the generalized case in \Cref{def: generalized cluster scattering}.

\begin{definition}\label{def: wall}
	A \emph{wall} in $M_\mathbb R$ (for $N^+$ and an $N^+$-graded Lie algebra $\mathfrak g$) is a pair $(\mathfrak d, g_\mathfrak d)$ such that
	\begin{enumerate}
		\item $g_\mathfrak d$ belongs to $G_{n_0}^\parallel$ for some primitive $n_0\in N^+$;
		\item $\mathfrak d\subset n_0^\perp \subset M_\mathbb R$ is a codimension one convex rational polyhedral cone.
	\end{enumerate}
\end{definition}

\begin{remark}
	The above definition works for general $N^+$-graded Lie algebras. In the case that $\mathfrak g$ is a Lie subalgebra of the tropical vertex Lie algebra $\mathfrak v$, the group $G_{n_0}^\parallel$ is embedded in $\Aut(\widehat{\Bbbk[P^\oplus]})$. Then the wall-crossing element $g_\mathfrak d$ can be equivalently represented by a function $f_\mathfrak d\in \widehat{\Bbbk [P^\oplus]}$.
\end{remark}

Now every wall has a direction $-p^*(n_0)\in M$ in the sense of \Cref{def: direction of a wall}. We call a wall $(\mathfrak d, g_\mathfrak d)$ with direction $m_0$ \emph{incoming} if $\mathfrak d = \mathfrak d - \mathbb R_{\geq 0}m_0$ and \emph{non-incoming} (or \emph{outgoing}) otherwise.

\begin{definition}
	A \emph{scattering diagram} over an $N^+$-graded algebra $\mathfrak g$ in $M_\mathbb R$ is a collection of walls such that for every degree $k>0$, there are only a finite number of $(\mathfrak d, g_\mathfrak d)\in \mathfrak D$ with the image of $g_\mathfrak d$ in $G^{<k}$ not being identity.
\end{definition}

The path-ordered product of a path $\gamma \colon [0,1]\rightarrow M_\mathbb R$ for a scattering diagram $\mathfrak D$ over $\mathfrak g$ can be defined similarly as in \Cref{def: path-ordered product}. We note that when $\gamma$ crosses a wall $(\mathfrak d, g_\mathfrak d)$ at time $t$, then the element $\xi_{\gamma, \mathfrak d}$ also depends on $\gamma'(t)$:
\[
\xi_{\gamma, \mathfrak d} = \begin{cases}
	g_\mathfrak d \quad & \text{if $\langle n_0, \gamma'(t)\rangle>0$}\\
	g_\mathfrak d^{-1} \quad & \text{if $\langle n_0, \gamma'(t) \rangle <0$}.
\end{cases}
\]
The consistency for these scattering diagrams is defined using path-ordered products in the same way as \Cref{def: path-ordered product}.

\begin{theorem}[{\cite[Proposition 2.1.12]{kontsevich2014wall}}, {\cite[Theorem 1.21]{gross2018canonical}}]\label{thm: KS reconstruction theorem}
	Let $\mathfrak D_\mathrm{in}$ be a scattering diagram over $\mathfrak g$ consisting only of incoming walls. Then there exists a unique (up to equivalence) consistent scattering diagram $\mathfrak D$ containing $\mathfrak D_\mathrm{in}$ such that $\mathfrak D\setminus \mathfrak D_\mathrm{in}$ consists only of outgoing walls.
\end{theorem}

Now we get back to the cluster situation. Suppose given fixed data $\Gamma$ and $\mathbf s$ an $\mathcal A$-seed with principal coefficients. Unlike the previous section, here we do not assume the maximality of the positive integers $r_i$, i.e. $r_i$ needs not to be $\mathrm{gcd}(b_{ki}\mid k\in I)$.

We calculate in the following how the group $G_{n_0}^\parallel$ is embedded in $\Aut (\widehat{ \Bbbk[P^\oplus]})$. Suppose $n_0 = \sum_{i\in I} \alpha_i e_i$, a primitive element in $N^+$. Consider any element
\[ x = \sum_{k>0} \sum_{\substack {p\in \mathbb P^\oplus \\ \pi(p) = k n_0 }} c_{p} \cdot p\cdot z^{kp^*(n_0)}\partial_{k\bar n_0} \in \hat {\mathfrak g}_{n_0}^\parallel,\quad c_{p}\in \Bbbk.\]
For non-zero $n\in N_\mathbb Q$, denote by $\operatorname{ind}(n)$ the largest number in $\mathbb Q_{\geq 0}$ such that $n/\operatorname{ind}(n)\in N$. Thus $n/\operatorname{ind}(n)$ is primitive in $N$.

\begin{lemma}\label{lemma: wall-crossing by multiply function}
	The group element $\exp(x) \in G_{n_0}^\parallel$ acts on $\widehat {\Bbbk[P^\oplus]}$ as an automorphism by
\[ 
	z^m \mapsto z^m \exp \left( \sum_{k>0} \sum_{\substack {p\in \mathbb P^\oplus \\ \pi(p) = k n_0 }} \mathrm{ind}(\bar n_0) kc_{p} \cdot p\cdot z^{kp^*(n_0)}\right)^{\langle r(m), \bar n_0/\mathrm{ind}(\bar n_0) \rangle},\quad m\in P^\oplus.
\]
\end{lemma}

\begin{proof}
	This follows by rewriting $x$ as
	\[
		 x = \left( \sum_{k>0} \sum_{\substack {p\in \mathbb P^\oplus \\ \pi(p) = k n_0 }} \mathrm{ind}(\bar n_0) kc_p \cdot p \cdot z^{kp^*(n_0)} \right) \partial_{\bar n_0/\mathrm{ind}(\bar n_0)}.
	\]
\end{proof}

Due to \Cref{lemma: wall-crossing by multiply function}, any $\exp(x)\in G_{n_0}^\parallel$ can be represented by a function $f$ as in \Cref{lemma: wall-crossing by multiply function} such that the action of $\exp(x)$ sends $z^m$ to $z^m f^{\langle r(m), \bar n_0/\mathrm{ind}(\bar n_0) \rangle}$.

Given $\mathbf s = (\mathbf e, \mathbf p)$, for each $i\in I$, consider the hyperplane $e_i^\perp$ with the attached wall-crossing function
\[
f_i = \prod_{j = 1}^{r_i} \left(1 + p_{i,j} z^{w_i} \right) \in \Bbbk[P^\oplus].
\]
As discussed, the function $f_i$ represents an element in $G_{e_i}^\parallel$.

\begin{definition}\label{def: generalized cluster scattering}
Let $\mathfrak D_{\mathbf s, \mathrm{in}}$ be the scattering diagram over $\mathfrak g$ in $M_\mathbb R$ consisting only of the incoming walls of the form $\mathfrak d_i \coloneqq (e_i^\perp, f_i)$, i.e.
	\[
		\mathfrak D_{\mathbf s, \mathrm{in}} \coloneqq \{(e_i^\perp, f_i)\mid i\in I\}.
	\]
	We define \emph{the generalized cluster scattering} $\mathfrak D_\mathbf s$ to be the unique (up to equivalence) consistent scattering diagram associated to $\mathfrak D_{\mathbf s,\mathrm{in}}$ guaranteed by \Cref{thm: KS reconstruction theorem}.
\end{definition}

\begin{remark}\label{rmk: change of lattice}
One may tend to think of $\mathfrak D_\mathbf s$ as a scattering diagram over $\widehat{\Bbbk[P^\oplus]}$ or over $\widehat{\Bbbk[P]}$ (as $\mathfrak g$ is a Lie subalgebra of $\mathfrak v$) in \Cref{def: tropical vertex sd}. However there is one subtle issue. Suppose that there is a wall $\left(\mathfrak d\subset n_0^\perp, f_{\mathfrak d} \right)$ in $\mathfrak D_\mathbf s$ for some $n_0\in N^+$ primitive. Then the wall-crossing action is given by
	\[
	\xi_{f_\mathfrak d}(z^p) = z^p f_{\mathfrak d}^{\langle \bar n_0/\mathrm{ind}(\bar n_0), r(p)\rangle}.
	\]
	Since in general $\bar n_0$ may not be proportional to $n_0$, the cone $\mathfrak d$ may not be contained in $\bar n_0^\perp$. In this case, the wall $(\mathfrak d, f_\mathfrak d)$ does not qualify as a wall in \Cref{def: tropical vertex sd}. This issue can be resolved in the following two ways (so that one can view $\mathfrak D_\mathbf s$ as a scattering diagram of \Cref{def: tropical vertex sd}). 
	
	\begin{enumerate}
		\item We could regard $\mathfrak g$ as graded by $\overline N^+\subset N_\mathbb Q$ (rather than $N^+$-graded) and modify \Cref{def: wall} (the definition of a wall $(\mathfrak d, g_\mathfrak d)$) so that $\mathfrak d$ is a codimension one cone in some hyperplane $n_0^\perp$ for $n_0\in \overline N^+$ and $g_\mathfrak d$ belongs to $G_{n_0}^\parallel$.
		\item Another way to resolve the issue is to consider the dual $\eta^* \colon M_\mathbb R\rightarrow M_\mathbb R$ of the linear map
		\[
		\eta \colon N_\mathbb R \rightarrow N_\mathbb R, \quad n\mapsto \bar n.
		\]
		We then apply $(\eta^*)^{-1}$ to every wall $(\mathfrak d, f_\mathfrak d)$ to get the collection
		\[ (\eta^*)^{-1}(\mathfrak D_\mathbf s) \coloneqq \{\left((\eta^*)^{-1}(\mathfrak d), f_\mathfrak d \right) \mid (\mathfrak d, f_\mathfrak d)\in \mathfrak D_\mathbf s \} \]
		Then the cone $(\eta^*)^{-1}(\mathfrak d)$ is indeed contained in $\bar n_0^\perp$. So this collection of walls is a scattering diagram in \Cref{def: tropical vertex sd}.
	\end{enumerate} 
From now on, to avoid any further confusion, the notation $\mathfrak D_\mathbf s$ is reserved for the consistent scattering diagram $(\eta^*)^{-1}(\mathfrak D_\mathbf s)$ over $\widehat{\Bbbk[P^\oplus]}$.
\end{remark}

\begin{lemma}\label{lemma: tropical vertex sd equals cluster sd}
	Let $\mathbf s$ be a seed with principal coefficients for some generalized fixed data $\Gamma$ (viewed of both $\mathcal A$- and $\mathcal X$-type) with the condition that for each $i\in I$, the element $$w_i = \omega(-, d_ie_i/r_i)$$ is primitive in $M$. In this case, we have defined both scattering diagrams $\mathfrak D_{(X_\Sigma, H)}$ (with a chosen general evaluation $\lambda$) and $\mathfrak D_\mathbf s$. Identify the parameters $t_{i,j}$ with $p_{i,j}$. Then $\mathfrak D_{(X_\Sigma, H)}$ and $\mathfrak D_\mathbf s$ are equivalent
	as scattering diagrams over $\widehat{ \Bbbk[P]}$.
\end{lemma}

\begin{proof}
We require $\omega_i$ to be primitive so that $\mathfrak D_{(X_\Sigma, H)}$ is defined. According to \Cref{rmk: change of lattice}, $\mathfrak D_\mathbf s$ is viewed as a scattering diagram over $\widehat{ \Bbbk[P]}$ in the same $M_\mathbb R$ as $\mathfrak D_{(X_\Sigma, H)}$ so it is legitimate to compare them. Let $\tilde {\mathfrak D}$ be the consistent scattering diagram over $\mathfrak g$ obtained using the initial data $\mathfrak D_{(X_\Sigma, H), \mathrm{in}}$. Notice that the walls in $\mathfrak D_{(X_\Sigma, H), \mathrm{in}}$ are parts of the hyperplanes $e_i^\perp$. We then subdivide the walls in $\mathfrak D_{\mathbf s, \mathrm{in}}$ so that $\mathfrak D_{(X_\Sigma, H), \mathrm{in}}$ becomes the subset of incoming walls. Thus $\tilde {\mathfrak D}$ is equivalent to $\mathfrak D_\mathbf s$ by \Cref{thm: KS reconstruction theorem}. 

On the other hand, $\tilde {\mathfrak D}$ is also a scattering diagram over $\widehat{\Bbbk[P]}$. By \Cref{thm: tropical vertex sd}, It is also equivalent to $\mathfrak D_{(X_\Sigma, H)}$ since they have the same incoming walls. Therefore we have $\mathfrak D_{(X_\Sigma, H)} \cong \tilde {\mathfrak D} \cong  \mathfrak D_\mathbf s$.
\end{proof}

\subsection{The cluster scattering diagrams of GHKK}\label{subsection: ordinary cluster sd}

The ordinary cluster scattering diagram $\mathfrak D_\mathbf s^\mathrm{ord}$ corresponds to the case where $r_i = 1$ for each $i\in I$. Thus there is only one parameter $p_i \coloneqq p_{i,1}$ for each $i\in I$. The lattice $\overline N$ is generated by $\bar e_i = d_i e_i$. The initial incoming walls are then
\[ \{\left( e_i^\perp, 1 + p_iz^{w_i} \right) \mid i\in I\},\]
where $w_i = \omega(-, \bar e_i)\in M$. 

This scattering diagram is closely related to the \emph{cluster scattering diagram} $\mathfrak D_\mathbf s^\mathrm{GHKK}$ of Gross, Hacking, Keel and Kontsevich \cite[Theorem 1.12]{gross2018canonical}. We explain the difference and relation here. The scattering diagram $\mathfrak D_\mathbf s^\mathrm{GHKK}$ is actually defined for $\overline N$ and $\overline M \coloneqq \Hom(\overline N, \mathbb Z)$ (in the ordinary case equal to $N^\circ$ and $M^\circ$ respectively). Under the injectivity assumption \cite[Section 1.1]{gross2018canonical}, the incoming walls are
\[
	\left \{\left( e_i^\perp, 1 + z^{\omega(e_i, -)} \right) \mid i\in I \right \}
\]
where $\omega(e_i, -)$ is in $M^\circ$. The injectivity assumption means that $\omega(e_i, -)$ generate a strict convex cone. If this is not the case, we may extend $M^\circ$ to $M^\circ \oplus \mathbb P$ (identified with $M^\circ \oplus N$ in \cite{gross2018canonical}) and let incoming walls be
\[
\left \{\left( e_i^\perp, 1 + p_iz^{\omega(e_i, -)} \right) \mid i\in I \right \}.
\]
It lives in $(M^\circ \oplus N)\otimes \mathbb R$, or in $M^\circ \otimes \mathbb R$ if regarding $p_i$ as formal parameters as we do. Then $\mathfrak D_\mathbf s^\mathrm{GHKK}$ is defined to be the unique consistent scattering diagram over $\widehat{\Bbbk[P]}$ with only these incoming walls, where $P\subset M^\circ \oplus N$ is a submonoid contained in a strictly convex cone and containing the cone generated by $(p_i, \omega(e_i, -))$. The Lie algebra $\mathfrak g$, however, is naturally graded by $N^+$ (generated by $e_i$'s), not $\overline N^+$ (generated by $\overline e_i$'s). Thus if one uses \Cref{thm: KS reconstruction theorem} to define $\mathfrak D_\mathbf s^\mathrm{GHKK}$, the same rescaling issue in \Cref{rmk: change of lattice} still exists and can be resolved in a similar way. In \cite{gross2018canonical}, $\mathfrak D_\mathbf s^\mathrm{GHKK}$ is regarded as living in $M^\circ_\mathbb R$ with the integral normal vectors of walls being in $N^\circ$.

The structures of $\mathfrak D_\mathbf s^\mathrm{GHKK}$ and $\mathfrak D_\mathbf s^\mathrm{ord}$ are very much alike. For example, they both admit cluster complex structures; see \cite[Theorem 2.13]{gross2018canonical} and \Cref{thm: wall crossing on cluster chamber}. It turns out that in the convention of \cite{fomin2007cluster} (e.g. the definition of $g$-vectors), $\mathfrak D_\mathbf s^\mathrm{GHKK}$ corresponds to the cluster algebra associated to $-B^T$ while $\mathfrak D_\mathbf s^\mathrm{ord}$ corresponds to the one associated to $B$, where $B = (b_{ij})$ with $b_{ij} = \omega(e_i, \bar e_j)$.

\subsection{Scattering diagrams with special coefficients}\label{subsection: sd with special coefficients}

Just as specializing a cluster algebra $\mathscr A$ at some evaluation $\lambda\colon \mathbb P\rightarrow \Bbbk^*$, one can do the same to $\mathfrak D_\mathbf s$, obtaining a consistent scattering diagram with special coefficients.

We consider another monoid $Q = M\oplus \prod_{i\in I} \mathbb N$ (with $t_i$ being the standard generators of $\prod_{i\in I} \mathbb N$). Let $\lambda \colon \mathbb P \rightarrow \Bbbk^*,\ p_{i,j}\mapsto \lambda_{i,j}$ be an evaluation. Define the map (abusing the same notation $\lambda$)
\[
	\lambda \colon \Bbbk[P] \rightarrow \Bbbk[Q], \quad z^m\mapsto z^m \text{ for $m\in M$},\quad p_{i,j}\mapsto \lambda_{i,j}t_i.
\]

\begin{lemma}
The collection
	\[ \lambda(\mathfrak D_\mathbf s)\coloneqq \{(\mathfrak d, \lambda(f_\mathfrak d)) \mid (\mathfrak d, f_\mathfrak d)\in \mathfrak D_\mathbf s\} \]
		obtained by applying the algebra homomorphism $\lambda$ to every wall-crossing function $ f_\mathfrak d $ is a consistent scattering diagram over $\widehat {\Bbbk[Q]}$.
\end{lemma}

\begin{proof}
	The algebra homomorphism $\lambda$ respects the completions of $\Bbbk[P]$ and $\Bbbk[Q]$. So $\lambda(f_\mathfrak d)$ belongs to $\widehat{\Bbbk[Q]}$. Recall we have the monoid map $r\colon P \rightarrow M$ which forgets the components in $\bigoplus_{i\in I}\mathbb N^{r_i}$. We use the same notation $r\colon Q\rightarrow M$ for the analogous map on $Q$. Then $(\mathfrak d, \lambda(f_\mathfrak d))$ becomes a wall over $\widehat{\Bbbk[Q]}$, and $\lambda(\mathfrak D_\mathbf s)$ is a scattering diagram over $\widehat{\Bbbk[Q]}$.
	
	The consistency of $\lambda(\mathfrak D_\mathbf s)$ follows from the consistency of $\mathfrak D_\mathbf s$ as $\lambda$ is an algebra homomorphism.
\end{proof}

We call $\lambda(\mathfrak D_\mathbf s)$ the (generalized) cluster scattering diagram of $\mathbf s$ with special coefficients $\lambda$. In fact, the ordinary cluster scattering diagram $\mathfrak D_\mathbf s$ when $r_i = 1$ can be obtained this way. We denote the ordinary one by $\mathfrak D_\mathbf s^\mathrm{ord}$. Its incoming walls are
\[
(e_i^\perp, 1 + p_i z^{\omega(-, d_ie_i)}).
\]
If there exist coefficients $\lambda_{ij}\in \Bbbk^*$  such that
\[
	\prod_{j = 1}^{r_i} \left( 1 + \lambda_{ij} t_i z^{w_i} \right) = 1 + t_i^{r_i} z^{r_iw_i} = 1 + t_i^{r_i} z^{\omega(-, d_ie_i)},
\]
then we can apply the corresponding morphism $\lambda \colon \Bbbk [P] \rightarrow \Bbbk[Q]$ to $\mathfrak D_\mathbf s$ so that
\[
\lambda(\mathfrak D_\mathbf s)  \cong \mathfrak D^\mathrm{ord}_\mathbf s
\]
as they have the exact same set of incoming walls. Here $t_i^{r_i}$ is identified with $p_i$. The existence of such an evaluation $\lambda$ amounts to find the $r_i$ roots of the polynomial $1 + x^{r_i}$ in $\Bbbk$, which is always possible if $\Bbbk$ is algebraically closed.

\subsection{Examples}
We illustrate some examples of generalized cluster scattering diagrams in this section.
\begin{example}\label{ex: example for b2}
Consider the fixed data $\Gamma$ consisting of
\begin{itemize}
		\item the lattice $N = \mathbb Z^2$ with with the standard basis $e_1 = (1,0)$ and $e_2 = (0,1)$, and the skew-symmetric form $\omega$ be determined by $\omega(e_1, e_2) = -1$;
		\item $N_\mathrm{uf} =  N$;
		\item the rank $r = 2$ and $I = I_\mathrm{uf} = \{1 ,2 \}$;
		\item positive integers $d_1 = 1$ and $d_2 = 2$;
		\item the sublattice $N^\circ$ generated by $e_1$ and $2e_2$;
		\item $M = \Hom(N, \mathbb Z)$, $M^\circ = \Hom(N^\circ, \mathbb Z)$;
	\end{itemize}

Let $\mathbf s$ be a seed consisting of $\mathbf e = (e_1, e_2)$ and $\mathbf p_1 = (t_{11})$, $\mathbf p_2 = (t_{21}, t_{22})$. We have matrices
\[
B = \begin{pmatrix}
	0 & -2\\
	1 & 0
\end{pmatrix}\quad \text{and} \quad \beta = \begin{pmatrix}
	0 & -1\\
	1 & 0
\end{pmatrix}.
\]
In this case we have $\bar e_i = d_ie_i/r_i = e_i$. So $\overline N = N$ and we shall not worry about the rescaling issue. Then $w_1 = e_2^*$ and $w_2 = -e_1^*$. We write $A_i = z^{e_i^*}$ for $i = 1, 2$. The coefficients group is $\mathbb P = \mathbb Z^3$ with generators $\{t_{11}, t_{21}, t_{22}\}$. The initial incoming scattering diagram is
\[ \mathfrak D_{\mathbf s, \mathrm{in}} = \{ (e_1^\perp, 1 + t_{11}A_2), (e_2^\perp, (1 + t_{21}A_1^{-1})(1 + t_{22} A_1^{-1})) \}.\]
The resulting generalized cluster scattering diagram is
\[
    \mathfrak D_\mathbf s = \mathfrak D_{\mathbf s, \mathrm{in}} \cup \left \{ \left(\mathbb R_{>0}(1,-1), f_{(1,-1)} \right), \left(\mathbb R_{>0}(2,-1), f_{(2,-1)} \right) \right \}
\]
where
\[
    f_{(1,-1)} = \left( 1 + t_{11}t_{21} A_1^{-1}A_2 \right) \left( 1 + t_{11}t_{22} A_1^{-1}A_2 \right) \quad \text{and} \quad f_{(2,-1)} = 1 + t_{11}t_{21}t_{22}A_1^{-2} A_2.
\]

The scattering diagram $\mathfrak D_\mathbf s$ is depicted in Figure \ref{figure: example for b2}.
\end{example}

\begin{figure}[h!]
\centering
\begin{tikzpicture}[scale = 0.9]

\node (A) at (0,3) {};
\node (B) at (3, 0) {};
\node (C) at (6, -3) {};
\node (D) at (4, -4) {};
\node (E) at (0, -3) {};
\node (F) at (-3, 0) {};

\draw[thick] (A) -- (0,0);
\draw[thick] (B) -- (0,0);
\draw[thick] (C) -- (0,0);
\draw[thick] (D) -- (0,0);
\draw[thick] (E) -- (0,0);
\draw[thick] (F) -- (0,0);
 
\draw (0,3) node[anchor = south]
{
	$1 + t_{11}A_2$
};
\draw (0,-3)  node[anchor = north]
{
	$1+t_{11}A_2$
};
\draw (3,0) node[anchor = west]
{
	$\left( 1 + t_{21} A_1^{-1} \right) \left( 1 + t_{22} A_1^{-1} \right)$
};
\draw (-3,0) node[anchor = east]
{
	$\left( 1 + t_{21} A_1^{-1} \right) \left( 1 + t_{22} A_1^{-1} \right)$
};
\draw (4,-4)  node[anchor = north]
{
	$\left( 1 + t_{11} t_{21} A_1^{-1}A_2 \right) \left( 1 + t_{11}t_{22} A_1^{-1}A_2 \right)$
};
\draw (6,-3)  node[anchor = north]
{
	$1 + t_{11}t_{21}t_{22}A_1^{-2} A_2$
};
\end{tikzpicture}
\caption{}
\label{figure: example for b2}	
\end{figure}

\begin{example}\label{ex: example for k2}
	Consider the fixed data $\Gamma$ consisting of
\begin{itemize}
		\item the lattice $N = \mathbb Z^2$ with with the standard basis $e_1 = (1,0)$ and $e_2 = (0,1)$, and the skew-symmetric form $\omega$ be determined by $\omega(e_1, e_2) = -1$;
		\item $N_\mathrm{uf} =  N$;
		\item the rank $r = 2$ and $I = I_\mathrm{uf} = \{1 ,2 \}$;
		\item positive integers $\lambda_1 = 1$ and $\lambda_2 = 1$;
		\item the sublattice $N^\circ$ generated by $e_1$ and $e_2$;
		\item $M = \Hom(N, \mathbb Z)$, $M^\circ = \Hom(N^\circ, \mathbb Z)$;
	\end{itemize}	
	
	The seed is given by $\mathbf e = (e_1, e_2)$ and $\mathbf p_1 = (s_1, s_2)$, $\mathbf p_2 = (t_1, t_2)$. The corresponding $\mathfrak D_\mathbf s$ is depicted in \Cref{figure: example for kronecker}. We write $X = z^{e_2^*}$ and $Y = z^{-e_1^*}$. The five rays depicted in the fourth quadrant are in the directions $(2,-1)$, $(3,-2)$, $(1,-1)$, $(2,-3)$ and $(1,-2)$ in clockwise order. In fact, in the fourth quadrant there are additional non-trivial walls whose underlying cones are $\mathbb R_{\geq 0} (n, -(n+1))$ and $\mathbb R_{\geq 0} (n+1, -n)$ for each positive integer $n\geq 3$ (which we omit in the figure below). The wall-crossing function, for example for $\mathbb R_{\geq 0} (2k, -(2k+1))$ for $k\in \mathbb Z_{>0}$, is
	\[
	f_{\mathbb R_{\geq 0} (2k, -(2k+1))} = \left(1 + s_1^{k+1}s_2^k t_1^kt_2^kX^{2k+1}Y^{2k}\right) \left(1 + s_1^{k}s_2^{k+1} t_1^kt_2^kX^{2k+1}Y^{2k} \right),
	\]
	which can be obtained using \Cref{thm: wall crossing on cluster chamber}.
	
	The wall-crossing function attached to the ray $\mathbb R_{\geq 0}(1,-1)$
	\[
	f_{\mathbb R_{\geq 0}(1,-1)} = \frac{(1 + s_1t_1XY)(1 + s_1t_2XY)(1 + s_2t_1XY)(1 + s_2t_2XY)}{(1-s_1s_2t_1t_2X^2Y^2)^4}
	\]
	is much more difficult to calculate. This was explicitly obtained by Reineke and Weist \cite{reineke2013refined} by relating the wall-crossing functions to quiver representations.
	
\begin{figure}[h!]
\centering
	\begin{tikzpicture}[scale = 1.7]
	
		\draw[thick] (0,0) -- (0,1);
		\draw[thick] (0,0) -- (0,-1);
		\draw[thick] (0,0) -- (1,0);
		\draw[thick] (0,0) -- (-1,-0);
		\draw[thick] (0,0) -- (2,-1);
		\draw[thick] (0,0) -- (2.4,-1.6);
		\draw[thick] (0,0) -- (1,-2);
		\draw[thick] (0,0) -- (1.6,-2.4);
		\draw[thick] (0,0) -- (2,-2);
		
		\draw (0,1) node[anchor = south] {$(1+s_1X)(1+s_2X)$};
		\draw (-1,0) node[anchor = east] {$(1 + t_1 Y) (1 + t_2 Y)$};
		\draw (1,-2) node[anchor = north east] {$(1+s_1s_2t_1 X^2Y)(1+s_1s_2t_2 X^2Y)$};
		\draw (1.6,-2.4) node[anchor = north] {$(1+s_1^2s_2t_1t_2 X^3Y^2)(1+s_1s_2^2t_1t_2 X^3Y^2)$};
		\draw (2.4,-1.6) node[anchor = north west] {$(1 + s_1s_2t_1^2t_2 X^2Y^3)( 1 + s_1s_2t_1t_2^2 X^2Y^3)$};
		\draw (2,-1) node[anchor = west] {$\left( 1 + s_1t_1t_2 XY^2 \right)\left( 1 + s_2 t_1t_2 XY^2\right)$};
		\draw (1,0) node[anchor = west] {$(1 + t_1 Y) (1 + t_2 Y)$};
		\draw (0,-1) node[anchor = north east] {$(1+s_1X)(1+s_2X)$};
		\draw (2,-2) node[anchor = north west] {$f_{\mathbb R_{\geq 0}(1,-1)}$};
		\draw (2.2, -1.8) node[anchor = center] {$\cdots$};
		\draw (1.8, -2.2) node[anchor = center] {$\cdots$};

	\end{tikzpicture}
	\caption{}
	\label{figure: example for kronecker}
\end{figure}
\end{example}

\subsection{Mutation invariance of $\mathfrak D_\mathbf s$}\label{subsection: mutation invariance}

A first step to investigate the structure of $\mathfrak D_\mathbf s$ is through a comparison with $\mathfrak D_{\mu_k(\mathbf s)}$. For the ordinary case, this is called the \emph{mutation invariance} in \cite{gross2018canonical}. In the generalized situation, we show an analogous mutation invariance still holds. One just needs to take care of the generalized coefficients $p_{i,j}$.

Notice that the definition of $\mathfrak D_\mathbf s$ does not involve the semifield structure of $\mathbb P$. So one can view that the coefficients part $\mathbf p$ actually provides a $\mathbb Z$-basis of the mutiplicative abelian group $\mathbb P$ (grouped and labeled in a certain way). Thus even though $\mu_k(\mathbf s)$ no longer has principal coefficients in $\mathbb P$, $\mathfrak D_{\mu_k(\mathbf s)}$ is still defined. To stress that the coefficients are no longer semifield elements, we use $t_{i,j}$ instead of $p_{i,j}$.

Now $\mathbf s = (\mathbf e, \mathbf t)$ consists of $\mathbf e$ a labeled basis of $N$ and tuples of coefficients $\mathbf t = (\mathbf t_i)$. 
\begin{definition}\label{def: a new definition of mutation}
Define the mutation $\mu_k^+(\mathbf s) = (\mathbf e', \mathbf t')$ such that $\mathbf e' = \mu_k(\mathbf e)$ as before and for the coefficients,
\[
	t_{i,j}' = \begin{cases}
			t_{k,j}^{-1} \quad &\text{if $i = k$}\\
 			t_{i,j} \cdot \prod\limits_{l=1}^{r_k} t_{k,l}^{[\beta_{ki}]_+} \quad &\text{if $i\neq k$}.
 			\end{cases}
\]	
\end{definition}

\begin{remark}
Note that this mutation does not depend on any semifield structure on $\mathbb P$. So it is different from the $\mu_k$ from \Cref{def: mutation of geometric seed} for mutations of many steps. For this reason, we call $\mathbf s = (\mathbf e, \mathbf p)$ a seed with coefficients (avoiding the type $\mathcal A$- or $\mathcal X$-) and use the new symbol $\mu_k^+$ for mutations in this context (as we will see in \Cref{subsection: cluster complex} the meaning of the sign $+$).	
\end{remark}

\begin{definition} We set
\[
\mathcal H_{k,+} \coloneqq \{ m\in M_\mathbb R \mid \langle e_k, m \rangle \geq 0 \},\quad \mathcal H_{k,-} \coloneqq \{ m\in M_\mathbb R \mid \langle e_k, m \rangle \leq 0 \}. 
\]
For $k\in I$, define the piecewise linear transformation $T_k \colon M_\mathbb R \rightarrow M_\mathbb R$ by
\[
T_k(m) \coloneqq  \begin{cases}
	m +  \langle e_k , m\rangle r_kw_k, \quad & m\in \mathcal H_{k,+}\\
	m,\quad & m\in H_{k,-}. 
\end{cases}
\]
\end{definition}

One sees that in the two half spaces, the map $T_k$ is actually the restriction of two linear maps $T_{k,+}$ and $T_{k,-}$ respectively. The map $T_k$ is with respect to the seed $\mathbf s$ and thus sometimes will be denoted as $T_k^\mathbf s$. The vector $r_kw_k$ can also be expressed as $r_kw_k = \omega (-, d_ke_k) = \sum_{i=1}^n b_{ik}e_i^*$. One checks that
\[
T_{k,+}(w_i) = w_i + \beta_{ki}r_kw_k.
\]

Recall we have the projection $r\colon M\oplus \mathbb P \rightarrow M$. The transformation $T_k$ can be lifted to $M\oplus \mathbb P$ by
\[
\tilde T_k(m,p) \coloneqq  \begin{cases}
	\left(m +  \langle e_k , m\rangle r_kw_k,\ p\cdot t_k^{\langle e_k, m\rangle}\right), \quad & m\in \mathcal H_{k,+}\\
	(m, p),\quad & m\in \mathcal H_{k,-},
\end{cases}
\]
where $t_k = \prod\limits_{l=1}^{r_k} t_{k,l}$. Note that $\tilde T_k$ on its domain of linearity is the restriction of two linear transformations $\tilde T_{k,\varepsilon}$ respectively.

\begin{construction}\label{construction: T_k action on sd}
We define the scattering diagram $T_k(\mathfrak D_\mathbf s)$ as in \cite[Definition 1.22]{gross2018canonical} (but taking care of the parameters $t_{i,j}$ here) in the following steps.

\begin{enumerate}
	\item Replace each wall in $\mathfrak D_\mathbf s$ not fully contained in $e_k^\perp$ if necessary by splitting it into two new walls
\[
	\left (\mathfrak d \cap \mathcal H_{k,+}, f_\mathfrak d\right) \quad \text{and} \quad \left (\mathfrak d \cap \mathcal H_{k,-}, f_\mathfrak d\right).
\]
Regard this new collection of walls as the current representative of $\mathfrak D_\mathbf s$.
	\item For a wall $(\mathfrak d, f_\mathfrak d)$ contained in $\mathcal H_{k,\varepsilon}$, define the wall $T_{k,\varepsilon}(\mathfrak d, f_\mathfrak d) = (T_{k,\varepsilon}(\mathfrak d), \tilde T_{k,\varepsilon}(f_\mathfrak d))$ where the new wall-crossing function $\tilde T_{k,\varepsilon}(f_\mathfrak d)$ is the one obtained from $f_\mathfrak d$ by replacing each monomial of the form
\[ 
pz^m \quad \text{by} \quad \tilde T_{k,\varepsilon}(pz^m),
\]
where the later is the monomial corresponding to $\tilde T_{k,\varepsilon}(m, p)\in M\oplus \mathbb P$. For example, we have
\[
\tilde T_{k,+}(t_{i,j}z^{w_i}) = t_{i,j}t_k^{\beta_{ki}} z^{w_i + \beta_{ki}r_kw_k},\quad \text{while} \quad \tilde T_{k,-}(t_{i,j}z^{w_i}) = t_{i,j}z^{w_i}.
\]
We call these walls uniformly by $T_k(\mathfrak d, f_\mathfrak d)$ no matter which half they belong to. We stress that the sign $\varepsilon$ in $T_{k,\varepsilon}$ is determined by which half space the wall $\mathfrak d$ lies in.
	\item Consider the collection of walls
\[
	T_k(\mathfrak D_\mathbf s) \coloneqq \left\{ T_k(\mathfrak d, f_\mathfrak d)\ \middle |\ (\mathfrak d, f_\mathfrak d)\in \mathfrak D(\mathbf s) \setminus \left(e_k^\perp, \prod_{j=1}^{r_k}(1+t_{k,j}z^{w_k}) \right) \right\} \cup  \left \{ \left( e_k^\perp, \prod_{j = 1}^{ r_k }  \left( 1 + t_{k,j}^{-1} z^{-w_k} \right) \right) \right\}.
\]
\end{enumerate} 
\end{construction}

Denote the monoid $(P')^\oplus \coloneqq P^\oplus_{\mu_k^+(\mathbf s)} \subset M \oplus \mathbb P$. While $\mathfrak D_\mathbf s$ is over $\widehat{\Bbbk[P^\oplus_\mathbf s]}$, $\mathfrak D_{\mu_k^+(\mathbf s)}$ is over $\widehat{\Bbbk[(P')^\oplus]}$.

\begin{theorem}[cf. {\cite[Theorem 1.24]{gross2018canonical}}]\label{thm: mutation invariance}
		The set of walls $T_k (\mathfrak D_\mathbf s)$ is indeed a consistent scattering diagram over $\widehat{\Bbbk[(P')^\oplus]}$, and furthermore is equivalent to $\mathfrak D_{\mu_k^+(\mathbf s)}$.
\end{theorem}
	
	We find it most natural to understand the mutation invariance by making connection to the \emph{canonical wall structure} (or canonical scattering diagram) \cite{gross2021canonical} via \cite[Theorem 6.1]{arguz2020higher}, where $\mathfrak D_\mathbf s$ can be viewed as associated to the toric model $U_{\Omega(\mathbf s, \lambda)}$ for general $\lambda$. However, as in \Cref{section: x toric model}, this would require the condition
	\[
	r_i = \mathrm{gcd}(b_{ij},i\in I).
	\]
	Fortunately, we can prove the mutation invariance following the same strategy in \cite{gross2018canonical} without this condition. The proof occupies the rest of the section.

First define a monoid $\bar P$ containing both $P^\oplus$ and $(P')^\oplus$. Let $\sigma$ be the cone in $(M\oplus \mathbb P)_\mathbb R$ generated by $\{(w_i, t_{i,j}) \mid i\in I, j \in [1, r_i]\} \cup \{(-w_k, -t_{k,j})\mid 1\leq j\leq r_k\}$. Take $\bar P = \sigma\cap (M\oplus \mathbb P)$ and we tend to talk about scattering diagrams over $\widehat {\Bbbk[\bar P]}$. However the ideal $\mathfrak m_{\bar P}$ misses the elements $(w_k, t_{k,j})$. This means  a wall such as
\[
(e_k^\perp, (1 + t_{k,j}z^{w_k}))
\]
in $\mathfrak D_\mathbf s$ does not qualify as a wall over $\widehat {\Bbbk[\bar P]}$. For this reason, we extend the definition of scattering diagram as in \cite[Definition 1.27]{gross2018canonical} (slightly generalizing the \emph{slab} for our needs).

Define 
\[
	\overline N_\mathbf s^{+,k} \coloneqq \left\{ \sum_{i\in I} a_i\bar e_i \ \middle | \ a_i\in \mathbb Z_{\geq 0} \text{ for $i\neq k$, $a_k\in \mathbb Z$, and $\sum_{i\in I\setminus\{k\}}a_i>0$}  \right\} \subset \overline N.
\]
Since $\overline N_\mathbf s^{+,k} = \overline N_{\mu_k^+(\mathbf s)}^{+,k}$, we denote them by $\overline N^{+,k}$.

\begin{definition}[cf. {\cite[Definition 1.27]{gross2018canonical}}]\label{def: sd with a slab}
	A \emph{wall} for $\bar P$ is a pair $(\mathfrak d, f_\mathfrak d)$ with $\mathfrak d$ as before but with primitive normal vector $n_0$ in $\overline N^{+,k}$ and 
	\[
		f_\mathfrak d = 1 + \sum_{k\geq 1, \pi(t) = kn_0} c_{k,t} \cdot tz^{k\omega(-, n_0)} \equiv 1 \mod  \mathfrak m_{\bar P}.
	\]
	The \emph{slab} for $\mathbf s$ and $k\in I$ means the pair 
	\[
	\mathfrak d_k \coloneqq \left( e_k^\perp, \prod_{j=1}^{r_k}(1 + t_{k,j}z^{w_k}) \right).
	\]
	A scattering diagram $\mathfrak D$ for $\bar P$ is a collection of walls and possibly this single slab, with the condition that for each $k>0$, $f_\mathfrak d \equiv 1 \mod \mathfrak m_{\bar P}^k$ for all but finitely many walls in $\mathfrak D$.
\end{definition}

We quote the following very hard theorem from \cite{gross2018canonical}. The objects here are understood in our definitions so there are minor differences. However, one can still prove the theorem in the exact same way. So we omit its proof here.

\begin{theorem}[{\cite[Theorem 1.28]{gross2018canonical}}]\label{thm: sd with a slab}
	There exists a unique (up to equivalence) consistent scattering diagram $\overline {\mathfrak D}_\mathbf s$ in the sense of \Cref{def: sd with a slab} such that
	\begin{enumerate}
		\item $\overline {\mathfrak D}_\mathbf s \supseteq \mathfrak D_{\mathbf s, \mathrm{in}}$,
		\item $\overline {\mathfrak D}_\mathbf s \setminus \mathfrak D_{\mathbf s, \mathrm{in}}$ consists only of outgoing walls.
	\end{enumerate}
Furthermore, $\overline {\mathfrak D}_\mathbf s$ is also a scattering diagram for the $\overline N_\mathbf s^+$-graded Lie algebra $\mathfrak g_\mathbf s$. As such, it is equivalent to $\mathfrak D_\mathbf s$.
\end{theorem}

\begin{proof}[Proof of {\Cref{thm: mutation invariance}}]
	First we choose a representative for $\mathfrak D_\mathbf s$ given by \Cref{thm: sd with a slab}. Now $T_k(\mathfrak D_\mathbf s)$ becomes a scattering diagram in the sense of \Cref{def: sd with a slab} for the seed $\mathbf s' = \mu_k^+(\mathbf s)$. This is because that
	\begin{enumerate}
		\item the operation $T_k$ removes the old slab $\mathfrak d_k$ and adds the new slab 
			\[
				\mathfrak d_k' \coloneqq \left(e_k^\perp, \prod_{j=1}^{r_k}(1 + t_{k,j}^{-1}z^{-w_k})\right); 
			\]
		\item for a wall (contained in either $\mathcal H_{k,+}$ or $\mathcal H_{k,-}$), $\tilde T_k$ sends a monomial of the form $\prod_{i,j}(t_{i,j}z^{w_i})^{a_{ij}}$ in its wall-crossing function to
		\[
		\prod_{i,j} \left(t_{i,j}t_k^{\beta_{ki}} z^{w_i + \beta_{ki}r_kw_k} \right)^{a_{ij}}\quad \text{or}\quad \prod_{i,j}(t_{i,j}z^{w_i})^{a_{ij}}.
		\]So if $tz^m \in \mathfrak m_{\bar P}^i$ for some $i$, so is $\tilde T_k(tz^m)$.
	\end{enumerate}
	We next show that (1) $T_k(\mathfrak D_\mathbf s)$ and $\mathfrak D_{\mathbf s'}$ have the same set of slabs and incoming walls; (2) $T_k(\mathfrak D_\mathbf s)$ is consistent as a scattering diagram with a slab. Then by the uniqueness statement of \Cref{thm: sd with a slab}, $T_k(\mathfrak D_\mathbf s)$ and $\mathfrak D_{\mathbf s'}$ are equivalent.
	
	The statement (1) follows from the same argument in \emph{Step I} of \cite[\emph{Proof of Theorem} 1.24]{gross2018canonical}.
	
	For (2), we check the consistency of $T_{k}(\mathfrak D_\mathbf s)$, that is, for any loop $\gamma$, $\mathfrak p_{\gamma, T_k(\mathfrak D_\mathbf s)} = \mathrm{id}$ whenever it is defined.
	
	If $\gamma$ is confined in one of the half spaces, the path-ordered product is identity because of the consistency of $\mathfrak D_\mathbf s$. So we assume that $\gamma$ crosses the slab $\mathfrak d_k'$. Split $\gamma$ into four subpaths $\gamma_1$, $\gamma_2$, $\gamma_3$ and $\gamma_4$ such that
	\begin{enumerate}
		\item $\gamma_1$ starts at a point in $\mathcal H_{k,-}$ and only crosses the slab $\mathfrak d_k'$;
		\item $\gamma_2$ is contained entirely in $\mathcal H_{k,+}$;
		\item $\gamma_3$ only crosses $\mathfrak d_k'$ back to $\mathcal H_{k,-}$;
		\item $\gamma_4$ is contained entirely in $\mathcal H_{k,-}$.
	\end{enumerate}
	
Let ${\tilde T_{k,+}}\colon \Bbbk[M\oplus \mathbb P] \rightarrow \Bbbk[M\oplus \mathbb P]$ be the algebra automorphism induced by $\tilde T_{k,+}$ (see (2) in the \Cref{construction: T_k action on sd} the action of ${\tilde T_{k,+}}$ on monomials). Denote by $\mathfrak p_{\mathfrak d_k'}$ the wall-crossing automorphism
\[
z^m\mapsto z^m\prod_{j=1}^{r_k}(1+t_{k,j}^{-1}z^{-w_k})^{-\langle e_k, m\rangle}.
\]
So we have
\begin{align}
	\mathfrak p_{\gamma_1, T_k(\mathfrak D_\mathbf s)} &= \mathfrak p_{\mathfrak d_k'}\\
	\label{1} \mathfrak p_{\gamma_2, T_k(\mathfrak D_\mathbf s)} &= {\tilde T_{k,+}} \circ \mathfrak p_{\gamma_2, \mathfrak D_\mathbf s} \circ {\tilde T_{k,+}}^{-1}\\ 
	\mathfrak p_{\gamma_3, T_k(\mathfrak D_\mathbf s)} &= \mathfrak p_{\mathfrak d_k'}^{-1}\\
	\mathfrak p_{\gamma_4, T_k(\mathfrak D_\mathbf s)} &= \mathfrak p_{\gamma_4, \mathfrak D_\mathbf s}.
\end{align}
All the above equalities except (\ref{1}) are by definitions. To show (\ref{1}), we see that it suffices to show the case where $\gamma_2$ only crosses one wall $\mathfrak d$ contained in $n_0^\perp$ with the wall-crossing function $f(m_0)$. We write $\tilde T = \tilde T_{k,+}$ and $T = T_{k,+}$. Then we compute the action of the right-hand side of (\ref{1}) on $z^m$:
\[
z^m \mapsto \tilde T^{-1} (z^m) \mapsto \tilde T^{-1} (z^m) f(z^{m_0})^{\langle T^{-1}(m), n_0 \rangle} \mapsto z^m f(\tilde T(z^{m_0}))^{\langle m, (T^{-1})^*(n_0)\rangle}.
\]
Note that the wall $\mathfrak d$ gets transformed under $T_k$ to be contained in $(T^{-1})^*(n_0)$ with $f(\tilde T(z^{m_0}))$. So the above action is the same as $\mathfrak p_{\gamma_2, T_k(\mathfrak D_\mathbf s)}(z^m)$.

To show $\mathfrak p_{\gamma, T_k(\mathfrak D_\mathbf s)} = \mathrm{id}$, it suffices to show that
\begin{equation}\label{eq: a mutation invariance equation}
{\tilde T_{k,+}}^{-1}\circ \mathfrak p_{\mathfrak d_k'} = \mathfrak p_{\mathfrak d_k},	
\end{equation}
so that $\mathfrak p_{\gamma, T_k(\mathfrak D_\mathbf s)} = \mathfrak p_{\gamma, T_k(\mathfrak D_\mathbf s)} = \mathrm{id}$.

Let the left-hand side act on some monomial, we have
\begin{align}\label{eq: an identity for mutation invariance}
	\begin{split}
	{\tilde T_{k,+}}^{-1}\circ \mathfrak p_{\mathfrak d_k'} (tz^m) & = {\tilde T_{k,+}}^{-1}\left( tz^m \prod_{j=1}^{r_k} (1 + t_{k,j}^{-1}z^{-w_k})^{-\langle e_k, m\rangle} \right )\\
	& = t \cdot t_k^{-\langle e_k, m\rangle} \cdot z^{m - \langle e_k, m \rangle r_kw_k} \prod_{j=1}^{r_k} (1 + t_{k,j}^{-1}z^{-w_k})^{-\langle e_k, m\rangle}\\
	& = tz^m \prod_{j=1}^{r_k}(1 + t_{k,j}^{-1}z^{w_k})^{-\langle e_k, m\rangle}\\
	& = \mathfrak p_{\mathfrak d_k}(tz^m).
\end{split}
\end{align}
This finishes the proof.
\end{proof}

\begin{example}
In this example we compute $T_2(\mathfrak D_\mathbf s)$ for the scattering diagram $\mathfrak D_\mathbf s$ in \Cref{ex: example for b2}. Recall that the exchange matrix for $\mathbf s$ is $B = \begin{psmallmatrix}
	0 & -2\\
	1 & 0
\end{psmallmatrix}.$ So we have $T_{2,+}(e_2^*) = e_2^* - 2 e_1^*$, which determines the ray $\mathbb R_{\geq 0}(e_2^*-2e_1^*)$ of the diagram below.

	\begin{figure}[h!]\label{figure: b2 example sd}
\centering
\begin{tikzpicture}[scale = 0.9]

\node (A) at (-3,1.5) {$1 + t_{11}t_{21}t_{22}A_1^{-2}A_2$};
\node (B) at (3, 0) {};
\node (C) at (6, -3) {};
\node (D) at (4, -4) {};
\node (E) at (0, -3) {};
\node (F) at (-3, 0) {};

\draw[thick] (A) -- (0,0);
\draw[thick] (B) -- (0,0);
\draw[thick] (C) -- (0,0);
\draw[thick] (D) -- (0,0);
\draw[thick] (E) -- (0,0);
\draw[thick] (F) -- (0,0);

\draw (0,-3)  node[anchor = north]
{
	$1+t_{11}A_2$
};
\draw (3,0) node[anchor = west]
{
	$\left( 1 + t_{21}^{-1} A_1 \right) \left( 1 + t_{22}^{-1} A_1 \right)$
};
\draw (-3,0) node[anchor = east]
{
	$\left( 1 + t_{21}^{-1} A_1 \right) \left( 1 + t_{22}^{-1} A_1 \right)$
};
\draw (4,-4)  node[anchor = north]
{
	$\left( 1 + t_{11} t_{21} A_1^{-1}A_2 \right) \left( 1 + t_{11}t_{22} A_1^{-1}A_2 \right)$
};
\draw (6,-3)  node[anchor = north]
{
	$1 + t_{11}t_{21}t_{22}A_1^{-2} A_2$
};
\end{tikzpicture}
\caption{}
\label{figure: scattering diagram for B2}	
\end{figure}

\end{example}

\subsection{Positivity} 
The scattering diagram $\mathfrak D_\mathbf s$ has the following positivity.

\begin{theorem}[cf. {\cite[Theorem 1.28]{gross2018canonical}}]\label{thm: generalized positivity}
	The scattering diagram $\mathfrak D_{\mathbf s}$ is equivalent to a scattering diagram all of whose walls $(\mathfrak d, f_\mathfrak d)$ satisfy $f_\mathfrak d = (1 + tz^m)^c$ for some $m = \omega(-, \bar n)$, $n\in N^+$, some $t\in \mathbb P$ such that $\pi(t) = n$, and $c$ being a positive integer. In other words, if we write $n = \sum_{i\in I} \alpha_i e_i$, then
	\begin{enumerate}
		\item $\mathfrak d$ is contained in $\bar n^\perp \subset M_\mathbb R$ where $\bar n = \sum_{i\in I}\alpha_i \frac{d_i}{r_i} e_i$;
		\item $ m = \sum_{i\in I} \alpha_i w_i = \omega(-, \bar n)$;
		\item if writing $t = \prod\limits_{i,j} t_{i,j}^{\alpha_{i,j}}$, then $\sum\limits_{j = 1}^{r_i} \alpha_{ij} = \alpha_i$.
	\end{enumerate}
\end{theorem}

\begin{proof}
	This theorem essentially follows from \cite[Appendix C.3]{gross2018canonical}, the proof of the positivity of $\mathfrak D^\mathrm{GHKK}_\mathbf s$. We use a representative of $\mathfrak D_\mathbf s$ constructed in the same algorithm used to produce $\mathfrak D^\mathrm{GHKK}_\mathbf s$ in the proof of \cite[Theorem 1.28]{gross2018canonical}. We will construct order by order a sequence of finite scattering diagrams $\mathfrak D_1\subset \mathfrak D_2 \subset \cdots$ (over $\widehat{k[P^\oplus_\mathbf s]}$ or the graded Lie algebra $\mathfrak g_\mathbf s$) such that their union
	\[
	\mathfrak D = \bigcup_{k=1}^\infty \mathfrak D_k
	\]
	is equivalent to $\mathfrak D_\mathbf s$. We then prove inductively that every wall in $\mathfrak D_k$ has the positivity property.
	
	Let $\mathfrak D_1 = \mathfrak D_{\mathbf s, \mathrm{in}}$. Note that $\mathfrak D_1$ is equivalent to $\mathfrak D_\mathbf s$ modulo $(P^+)^2$. Suppose that we have defined up to $\mathfrak D_k$ which is equivalent to $\mathfrak D$ modulo $(P^+)^{k+1}$, and assume that every wall in $\mathfrak D_k$ has wall-crossing function of the form $(1 + t z^m)^c$ for some positive integer $c$. We construct $\mathfrak D_{k+1}$ as follows, and show that it is equivalent to $\mathfrak D$ modulo $(P^+)^{k+2}$ and furthermore that it still has the same positivity property for its wall-crossing functions.
	
	There is a finite rational polyhedral cone complex that underlies the support of $\mathfrak D_k$ (which is true for any scattering diagram with finitely many walls). We call the codimension two cells \emph{joints}. Let $\mathfrak j$ be a joint of $\mathfrak D_k$. Then by \cite[Definition-Lemma C.2]{gross2018canonical}, it falls into two classes
	\begin{enumerate}
		\item \emph{parallel}, if every wall with the normal vector $n$ containing $\mathfrak j$ has $\omega(-,n)$ tangent to $\mathfrak j$;
		\item \emph{perpendicular}, if every wall with the normal vector $n$ containing $\mathfrak j$ has $\omega(-,n)$ not tangent to $\mathfrak j$. 
	\end{enumerate}
	Let $\gamma_\mathfrak j$ be a simple loop around $\mathfrak j$ small enough so that it only intersects walls containing $\mathfrak j$. By our assumption, the path-ordered product $\mathfrak p_{\gamma_\mathfrak j, \mathfrak D_k}$ is identity modulo $(P^+)^{k+1}$, but modulo $(P^+)^{k+2}$, it can be written as
	\[
	\mathfrak p_{\gamma_\mathfrak j, \mathfrak D_k} = \exp\left(\sum_{d(t,m) = k+1} c_{t,m}tz^m \partial_{n(t,m)} \right),
	\]
	where $c_{t,m}\in \Bbbk$. Here we define the degree $d(t,m) \coloneqq k+1$ if $(t,m)\in (P^+)^{k+1} \setminus (P^+)^{k+2}$, and $n(t,m)$ is primitive in $N^+$ uniquely determined by $(t,m)$.

	If $\mathfrak j$ is perpendicular, we define a set of walls
	\[
	\mathfrak D[\mathfrak j]\coloneqq \{ (\mathfrak j - \mathbb R_{\geq 0}m, (1 + tz^m)^{\pm c_{t,m}}) \mid d(t,m) = k+1 \},
	\]
	where $\mathfrak j-\mathbb R_{\geq 0}m$ is of codimension one since $m$ is not tangent to $\mathfrak j$.
	Here the function $(1 + tz^m)^{\pm c_{t,m}}$ makes sense as a power series. The sign $\pm$ in the power is chosen so that when $\gamma_\mathfrak j$ crosses $\mathfrak j-\mathbb R_{\geq 0}m$, the wall-crossing automorphism is $$\exp(-c_{t,m}tz^m \partial_{n(t,m)}).$$
	In this way, if we add the walls in $\mathfrak D[\mathfrak j]$ to $\mathfrak D_k$, we have the path-ordered product $\mathfrak p_{\gamma_\mathfrak j, \mathfrak D_k \cup \mathfrak D[\mathfrak j]} = \mathrm{id}$ modulo $(P^+)^{k+2}$. We then define
	\[
	\mathfrak D_{k+1} = \mathfrak D_k \cup \bigcup_{\mathfrak j} \mathfrak D[\mathfrak j],
	\]
	where the union is over all perpendicular joints of $\mathfrak D_k$.
	
	There are two things we need to show in the induction:
	\begin{enumerate}
		\item $\mathfrak D_{k+1}$ is equivalent to $\mathfrak D_\mathbf s$ modulo $(P^+)^{k+2}$.
		\item All the walls in $\mathfrak D_{k+1}$ have wall-crossing functions of the form
		$(1 + t z^m)^c$ for some positive integer $c$.
	\end{enumerate}
	
	Part (1) follows from the argument in \cite[Lemma C.6 and Lemma C.7]{gross2018canonical}. This part guarantees that the constructed union $\mathfrak D$ is equivalent to $\mathfrak D_\mathbf s$. 
	
	Part (2) is about the positivity of wall-crossings. By the construction of $\mathfrak D_{k+1}$, we only need to examine the new walls emerging from perpendicular joints of $\mathfrak D_k$. Let $\mathfrak j$ be a perpendicular joint of $\mathfrak D_k$.
	The integral normal space $\mathfrak j^\perp \cap N$ is a rank two saturated sublattice $O$ of $N$. Locally at $\mathfrak j$, $\mathfrak D_k\cup \mathfrak D[\mathfrak j]$ induces a scattering diagram living in $O_\mathbb R^\vee = M_\mathbb R/(\Lambda_j\otimes \mathbb R)$. Precisely, consider the set of walls
	\[
	\mathfrak D' = \{ ((\mathfrak d + \Lambda_\mathfrak j \otimes \mathbb R)/(\Lambda_\mathfrak j\otimes \mathbb R), f_\mathfrak d) \mid \mathfrak j\subset \mathfrak d,\ (\mathfrak d, f_\mathfrak d)\in \mathfrak D_k\cup \mathfrak D[\mathfrak j]\}.
	\]
	The wall-crossing functions $f_\mathfrak d$ are all of the form
	\[
	(1 + t z^m)^c,
	\]
	$c\in \Bbbk$ ($f_\mathfrak d$ makes sense as a power series). The wall $\mathfrak d$ has some primitive normal vector $o\in O\cap N^+$, and $m$ is proportional to $\omega(-,o)$. We also know since $\mathfrak j$ is perpendicular, $\bar m\neq 0$ (the image of $m$ under the quotient $M\rightarrow O^\vee$) in $O_\mathbb R^\vee$. And the one dimensional wall $\bar {\mathfrak d} = (\mathfrak d + \Lambda_\mathfrak j \otimes \mathbb R)/(\Lambda_\mathfrak j\otimes \mathbb R)$ is contained in $\mathbb R(\bar m)$, orthogonal to the normal vector $o$. Then $\mathfrak D'$ is a rank two scattering diagram in $O_\mathbb R^\vee$ over $\widehat{\Bbbk[P^+]}$, with the monoid map from $P^+$ to $O^\vee$ being $r\colon P\rightarrow M$ post-composed by the quotient from $M$ to $O^\vee$. It is consistent up to modulo $(P^+)^{k+2}$. Then by \cite[Proposition C.13]{gross2018canonical}, the wall-crossing functions admit the positivity property, i.e. the power $c$ is always a positive integer. This shows the positivity for $\mathfrak D_{k+1}$ assuming that of $\mathfrak D_k$. Therefore, the union $\mathfrak D$ is also positive by induction, hence so is $\mathfrak D_\mathbf s$.
\end{proof}

\section{The cluster complex structure}\label{section: cluster complex structure}

In this section, we study the \emph{cluster complex structure} of the scattering diagram $\mathfrak D_\mathbf s$, which is a description of parts of the walls of $\mathfrak D_\mathbf s$. The construction of such a structure of $\mathfrak D_\mathbf s$ is analogous to \cite[Construction 1.30]{gross2018canonical}.

\subsection{The cluster complex}\label{subsection: cluster complex}
Take a representative for the scattering diagram $\mathfrak D_\mathbf s$ with minimal support (which always exists). By \Cref{thm: sd with a slab}, one can choose such a representative $\mathfrak D_\mathbf s$ so that there are no other walls contained in the initial incoming ones $\mathfrak d_i$.

Define
\begin{align*}
	\mathcal C^+ = \mathcal C_\mathbf s^+ \coloneqq \{ m\in M_\mathbb R \mid \langle e_i, m \rangle \geq 0\quad \forall i \in I\},\\
	\mathcal C^- = \mathcal C_\mathbf s^- \coloneqq \{ m\in M_\mathbb R \mid \langle e_i, m \rangle \leq 0\quad \forall i \in I\}.
\end{align*}
The closed cones $\mathcal C_\mathbf s^\pm$ are closures of connected components of $M_\mathbb R \setminus \mathrm{Supp}(\mathfrak D_\mathbf s)$. They are thus called \emph{chambers}. By the mutation invariance \Cref{thm: mutation invariance}, we have that the cones 
\[
T_k^{-1}\left( \mathcal C_{\mu_k^+(\mathbf s)}^\pm \right) \subset M_\mathbb R \setminus \mathrm{Supp}(\mathfrak D_\mathbf s)
\]
are also closures of connected components. Applying mutations on seeds provides an iterative way to construct chambers of  $M_\mathbb R \setminus \mathrm{Supp}(\mathfrak D_\mathbf s)$ as follows.

Note again that the coefficients part of $\mathbf s = (\mathbf e, \mathbf t)$ does not mutate as in \Cref{def: mutation of geometric seed}, which requires setting the tropical semifield $\mathbb P$ from the initial seed and once for all. Instead, we regard the coefficients part $\mathbf t$ as in the multiplicative group $\mathbb P$ and mutates in the way specified by \Cref{subsection: mutation invariance}. In this way, we can apply mutations iteratively on $\mathbf s$. 

Let us consider the rooted tree $\mathfrak T_n$ from \Cref{def: rooted tree GHK}. There is an association $v\mapsto \mathbf s_v$ such that $v_0 \mapsto \mathbf s$ and adjacent seeds with coefficients are related by the corresponding mutation (in the sense of \Cref{subsection: mutation invariance}) of the labeled edges. Once this association is done, we denote the rooted tree by $\mathfrak T_\mathbf s$.

Suppose the unique path from $v_0$ to a vertex $v$ goes through the arrows labeled by $\{k_1, k_2, \dots, k_l\}$. Define the piecewise linear map
\[
	T_{v_0, v} = T_{k_l}\circ \cdots \circ T_{k_2}\circ T_{k_1} \colon M_\mathbb R \rightarrow M_\mathbb R.
\]
Since $\mathcal C_\mathbf s^\pm$ are chambers of the scattering diagram $\mathfrak D_\mathbf s$, then again due to the mutation invariance, we have that 
\[
	\mathcal C_v^\pm \coloneqq T_{v_0, v}^{-1}\left(\mathcal C_{\mathbf s_v}^\pm \right)
\]
are chambers of $\mathfrak D_\mathbf s$.

Each $\mathcal C_v^\pm$ is a simplicial (rational polyhedral) cone of maximal dimension, as each $T_k$ is a linear isomorphism on its domains of linearity. The intersection $\mathcal C_\mathbf s^+ \cap \mathcal C_{\mu_k^+(\mathbf s)}^+$ is their common facet generated by $\{e_i^* \mid i\neq k\}$. Each facet of $\mathcal C_v$ is canonically labeled by an index $i\in I$. Inductively, for any two vertices $v$ and $v'$ connected by an arrow labeled by $k\in I$, then $\mathcal C_v^+$ and $\mathcal C_{v'}^+$ share a common facet labeled by $k$.

We borrow the following notation from \cite{gross2018canonical}: we use the short-hand subscription notation $v\in \mathbf s$ for an object parametrized by a vertex $v\in \mathfrak T_\mathbf s$ with the root $v_0$ labeled by $\mathbf s$. This is done to emphasize the dependence on the initial seed $\mathbf s$.

\begin{definition}
	We denote by $\mathcal C_{v\in \mathbf s}^\pm$ the chambers $\mathcal C_v^\pm$ of $\subset M_\mathbb R \setminus \mathrm{Supp}(\mathfrak D_\mathbf s)$. We write $\Delta_\mathbf s^\pm$ for the set of chambers $\mathcal C_{v\in \mathbf s}^\pm$ for $v$ running over all vertices of $\mathfrak T_\mathbf s$. We call elements in $\Delta_\mathbf s^+$ \emph{cluster chambers}.
 \end{definition}

\begin{remark}
	As we have pointed out, $\mathcal C_v^+\cap \mathcal C^+_{v'}$ is a common facet if $v$ and $v'$ are adjacent in $\mathfrak T_\mathbf s$. More generally, by adding all the faces of every $\mathcal C^+_v$ to the set $\Delta_\mathbf s^+$, we obtain a collection of cones which form a cone complex, still denoted by $\Delta_\mathbf s^+$. For this reason, we call $\Delta_\mathbf s^+$ \emph{the cluster (cone) complex} and $\Delta_\mathbf s^-$ \emph{the negative cluster (cone) complex}.
\end{remark}

The simplicial cone $\mathcal C_{v\in \mathbf s}^\pm$ is determined by (the generators of) its one-dimensional faces. The cone $\mathcal C_{\mathbf s_v}^+$ is generated by the dual vectors $\{e_{i;v}^*\mid i\in I\}$. These are pulled back by $T_{v_0,v}^{-1}$ to be the generators of $\mathcal C_{v\in \mathbf s}^+$.

\begin{definition}
	We define the \emph{$g$-vectors} for $v\in \mathfrak T_\mathbf s$ as a tuple
	\[
	\mathbf g_v = (g_{i;v}\mid i\in I),\quad \text{where}\quad g_{i;v} \coloneqq T_{v_0,v}^{-1}\left(e_{i;v}^*\right)\in M.
	\]
	We will use the notation $\mathbf g_{v\in \mathbf s}$ to emphasize the initial seed $\mathbf s$.
\end{definition}

\begin{remark}\label{rmk: sign coherence of normal vector}
Denote the dual vectors (in $N$) of $\mathbf g_v$ by $\mathbf g^*_v = (g_{i;v}^*\mid i\in I)$. They are normal vectors of the facets of $\mathcal C_{v}^+$. Since the walls of $\mathfrak D_\mathbf s$ only have normal vectors in $N_\mathbf s^+$ or $-N_\mathbf s^+$, the vector $g_{i;v}^*$ has a well-defined sign
\[
	\varepsilon_{i;v} = \mathrm{sgn}(g_{i;v}^*) = \begin{cases}
		+ \quad &\text{if $g_{i;v}^*\in N^+_\mathbf s$}\\
		- \quad &\text{if $g_{i;v}^*\in N^-_\mathbf s$}.
	\end{cases} 
\]
\end{remark}

We will show later the vectors $\mathbf g_v$ can be calculated iteratively by a variant of mutations as defined below.	

\begin{definition}
	Let $\mathbf e = (e_i\mid i\in I)$ be a seed (without coefficients) for $\Gamma$. Define the \emph{signed mutation} $\mu_k^\varepsilon(\mathbf e) = (e_i'\mid i \in I)$ for $\varepsilon\in \pm$ as follows.
	\[
	e_i' = \begin{cases}
		- e_k,\quad &\text{if $i = k$}\\
		e_i + [-\varepsilon \omega (e_i, d_k e_k)]_+e_k, \quad &\text{if $i\neq k$}.
	\end{cases}
	\]
	So the signed mutation $\mu_k^+$ coincides with our previous \Cref{def: a new definition of mutation} (ignoring the coefficients part). 
	
	On the mutation of the dual of $\mathbf e$, we use the same notation $\mu_k^\varepsilon(\mathbf e^*) = (f'_i \mid i\in I)$ where $\mathbf e^* = (f_i \mid i\in I)$. Then
	\[
	f_i' = \begin{cases}
		f_i,\quad &\text{if $i \neq k$}\\
		-f_k + \sum_{i\in I}[-\varepsilon \omega (e_i, d_k e_k)]_+f_k, \quad &\text{if $i =  k$}.
	\end{cases}
	\]
\end{definition}

There is another tuple of vectors in $M$ that changes under signed mutations. For a seed $\mathbf s$, let $\mathbf w = (w_{i}\mid i\in I)$ where
\[
w_i \coloneqq \omega\left(-, \frac{d_k}{r_k}e_k \right) = \sum_{j\in I} b_{ji} f_i \in M.
\]
Let $\mathbf w' = (w_i')$ associated to $\mu_k^\varepsilon(\mathbf e)$. Then we have
\[
w_i' = \begin{cases}
		- w_k,\quad &\text{if $i = k$}\\
		w_i + [\varepsilon \omega (e_k, d_k e_i)]_+w_k, \quad &\text{if $i\neq k$}.
	\end{cases}
\]
We will later denote $\mu_k^\varepsilon (\mathbf w) = \mathbf w'$.

There are also signed mutations for coefficients. Recall we have fixed a multiplicative abelian group $\mathbb P = \prod_{i\in I} \mathbb Z^{r_i}$. The coefficients $\mathbf t = (t_{i,j}\mid i\in I, j\in [1,r_i])$ are a basis of $\mathbb P$.
\begin{definition}
	For $\mathbf s = (\mathbf e, \mathbf t)$, a seed $\mathbf e$ together with coefficients $\mathbf t = (t_{i,j})$ in $\mathbb P$, we define its \emph{signed mutation in direction $k$}, $\mu_k^\varepsilon(\mathbf e, (t_{i,j})) = (\mathbf e', (t_{i,j}'))$ for $\varepsilon \in \pm$ by setting $\mathbf s' = \mu_k^\varepsilon(\mathbf s)$ and 
			\[
		t_{i,j}' = \begin{cases}
			t_{k,j}^{-1} \quad &\text{if $i = k$}\\
 			t_{i,j} \cdot \prod\limits_{l=1}^{r_k} t_{k,l}^{[\varepsilon \omega(e_k , e_i)]_+} \quad &\text{if $i\neq k$}.
 			\end{cases}
 			\]
\end{definition}

\begin{proposition}[cf. {\cite[Proposition 4.4.9]{mou2020wall}}]\label{prop: mutation of g vectors}
For every $v\in \mathfrak T_\mathbf s$, the dual of $g$-vectors $\mathbf g_v^*$ is a seed of $N$. These seeds and their duals, i.e. the $g$-vectors, can obtained iteratively as follows.
\begin{enumerate}
		\item $\mathbf g_{v_0} = \mathbf e^*$ and $\mathbf g_{v_0}^* = \mathbf e$;
		\item for any $v\overset{k}{\longrightarrow}v'$ in $\mathfrak T_\mathbf s$, we have
		\[
		\mathbf g_{v'}^* = \mu_{k}^{\varepsilon_{k;v}}(\mathbf g_{v}^*),\quad \mathbf g_{v'} = \mu_{k}^{\varepsilon_{k;v}}(\mathbf g_{v}).
		\]
\end{enumerate}
\end{proposition}

\begin{proof}
	We prove this proposition by induction on the distance from $v$ to $v_0$. The base case is when $v = v_0$, in which we have
	\[
	\mathbf g_{v'}^* = \mu_k^+(\mathbf e) = \mu_k^+(\mathbf g^*_{v_0}),\quad \mathbf g_{v'} = \mu_k^+(\mathbf e^*) = \mu_k^+(\mathbf g_{v_0}).
	\]
	
	Now assuming that $v\neq v_0$ and suppose that the unique path from $v_0$ to $v$ starts with $v_0\overset{i}{\longrightarrow}v_1$ for some $i\in I$. Write $\mathbf s_1 = \mathbf s_{v_1} = \mu_i^+(\mathbf s)$. By induction, we assume that the proposition holds for $g$-vectors with respect to the seed $\mathbf s_1$:
	\[
	\mathbf g_{v'\in \mathbf s_1} = \mu_k^\varepsilon (\mathbf g_{v\in \mathbf s_1})
	\]
	where $\varepsilon = \varepsilon_{k;v\in \mathbf s_1} = \mathrm{sgn}(g_{k; v\in \mathbf s_1}^*)$ with respect to $\mathbf s_1$.
	Note that by definition
	\[
	\mathbf g_{v'\in \mathbf s} = (T_i^{\mathbf s})^{-1}(\mathbf g_{v'\in \mathbf s_1}),\quad \mathbf g_{v\in \mathbf s} = (T_i^{\mathbf s})^{-1}(\mathbf g_{v\in \mathbf s_1}),
	\]
	and we want to prove
	\[
	\mathbf g_{v'\in \mathbf s} = \mu_k^\delta (\mathbf g_{v\in \mathbf s})
	\]
	where $\delta = \varepsilon_{k;v\in \mathbf s}$ with respect to $\mathbf s$.
	
	Then it amounts to show that
	\begin{equation}\label{eq: induction g vector}
	(T_i^{\mathbf s})^{-1}\circ \mu_k^{\varepsilon}(\mathbf g_{v\in \mathbf s_1}) = \mu_k^{\delta}\circ (T_i^{\mathbf s})^{-1}(\mathbf g_{v\in \mathbf s_1}).	
	\end{equation}
	We split the discussion in the following two cases. The codimension one skeletons of the chambers $\mathcal C_{v\in \mathbf s_1}^+$ and $\mathcal C_{v'\in \mathbf s_1}^+$ are in the essential support of $\mathfrak D_{\mathbf s_1}$. As $v$ and $v'$ are adjacent, these two chambers share a common facet. Therefore they are either separated by the hyperplane $e_i^\perp$ or contained in the same half space (since the hyperplane is also in the essential support).  
	
	\textbf{Case 1.} The two groups of $g$-vectors $\mathbf g_{v\in \mathbf s_1}$ and $\mathbf g_{v'\in \mathbf s_1}$ are separated by $e_i^\perp$. In this case, the normal vector $g_{k;v\in \mathbf s_1}^*$ is in the direction of $e_i$. The signs $\delta$ and $\varepsilon$ on the two sides of (\ref{eq: induction g vector}) are then different. We assume that $\varepsilon = \mathrm{sgn}(g_{k;v\in\mathbf s_1}^*) = +$; the other case is analogous. By our assumption, $\mathbf g_{v\in \mathbf s_1}^*$ qualifies as a seed of fixed data $\Gamma$, thus forming a basis of $N$, which implies $g_{k;v\in \mathbf s_1}^* = e_i$. Since $\{\lambda_j g_{j;v\in \mathbf s_1}^* \mid j\in I\}$ form a basis of the sublattice $N^\circ$, we have $d_i = d_k$. We note that the map $T_i^\mathbf s$ is actually determined by the vectors $e_i$ and $d_ie_i$. On the left hand side of (\ref{eq: induction g vector}), $T_i^\mathbf s$ is the identity, while on the right hand side, it is $T_{i,+}^\mathbf s$. So we need to show the equality
	\[
	\mu_k^+(\mathbf g_{v\in \mathbf s_1}) = \mu_k^-\circ (T_{i,+}^\mathbf s)^{-1}(\mathbf g_{v\in \mathbf s_1}).
	\]
	To simplify the notation, we denote $\mathbf g = \mathbf g_{v\in \mathbf s_1}$ and $g_i = g_{i;v\in \mathbf s_1}$. On the left side of the equality, the tuple $\mu_k^+(\mathbf g) = (g_i')$ differs with $\mathbf g$ by only one vector 
	\[
		g_k' = -g_k + \sum_{i\in I} [- b_{ik}^v]_+g_i.
	\]
	On the right hand side, we first have
	\[
		(T_{i,+}^\mathbf s)^{-1}(g_k) = -g_k + \sum_{i\in I} -b_{ik}^vg_i,
	\]
	while other $g$-vectors remain unchanged under $(T_{i,+}^\mathbf s)^{-1}$. It is easy to check that the dual of $(T_{i,+}^\mathbf s)^{-1}$ is an automorphism of $(N, \omega)$, that is, it is a linear automorphism on $N$ preserving the form $\omega$. Thus we have, if writing $\mu_k^-\circ (T_{i,+}^\mathbf s)^{-1}(\mathbf g_{v\in \mathbf s_1}) = (g''_i)$,
	\[
	g''_k = -g_k + \sum_{i\in I} -b_{ik}^v g_i + \sum_{i\in I} [b_{ik}^v]_+ g_i = g_k', 
	\quad \text{and}\quad
	g''_i = g_i\quad \text{for $i\neq k$}.
	\]
	This finishes the proof of the desired equality.
	
\textbf{Case 2.} The $g$-vectors $\mathbf g_{v\in \mathbf s_1}$ and $\mathbf g_{v'\in \mathbf s_1}$ are all contained in the same half $\mathcal H_{i,+}^\mathbf s$ or $\mathcal H_{i,-}^\mathbf s$. Again we need to prove (\ref{eq: induction g vector}). We observe that the two signs $\delta$ and $\varepsilon$ are equal. In fact, the sign $\varepsilon$ of $g_{k;v\in \mathbf s_1}^*$ depends on its coordinates in $e_{j;v_1}$ for $j\neq i$ since $g_{k;v\in \mathbf s_1}^*$ is not purely proportional to $e_i$. The same is true for the sign $\delta$ which only depends on $g_{k;v\in \mathbf s}^*$'s coordinates in $e_{j}$ for $j\neq i$. Since  $g_{k;v\in \mathbf s}^*$ only differ in the direction of $e_{i}$, and also because $e_{j;v_1}$ and $e_{j}$ also differ by multiples of $e_i$, we conclude that $\varepsilon = \delta$. The equality (\ref{eq: induction g vector}) then directly follows from a fact we already mentioned in Case 1 that the dual of $(T_{i,\varepsilon}^\mathbf s)^{-1}$ acts as an automorphism on $(N, \omega)$.
\end{proof}

A direct corollary of \Cref{prop: mutation of g vectors} is another description of $c$-vectors mentioned in \Cref{subsection: principal coefficients}. Recall that we have $\pi\colon \mathbb P\rightarrow \overline{\mathbb P}$, $p_{i,j}\mapsto \bar p_i$. We write the group operation in $\mathbb P$ and $\overline{\mathbb P}$ by addition instead of multiplication.

\begin{corollary}\label{cor: from cone complex to c vectors}
	We identify the lattice $\overline N$ with $\overline {\mathbb P}$ by $\bar e_i = \frac{d_i}{r_i}e_i\mapsto \bar p_i$. Then we have for any $i\in I$ and $v\in \mathfrak T_\mathbf s$,
	\[ \frac{d_i}{r_i}g_{i;v}^* = \bar p_{i;v}, \quad d_i g_{i;v}^* = r_i \bar p_{i;v} = p_{i;v}. \]
\end{corollary}

\begin{proof}
	For the initial vertex $v_0$, this is given by the identification $\bar e_i \mapsto \bar p_i$. The iteration of $g_{i;v}^*$ is provided by signed mutations according to \Cref{prop: mutation of g vectors}. We have if $v \overset{k}{\longrightarrow} v'$ in $\mathfrak T_\mathbf s$,
	\[
	g_{i;v'}^* = \begin{cases}
		-g_{k;v}^* \quad & \text{if $i = k$}\\
		g_{i;v}^* + [ - \varepsilon b_{ik}^v ]_+ g_{i;v}^* \quad &\text{if $i\neq k$}.
	\end{cases}
	\]
	where $\varepsilon = \mathrm{sgn} (g_{k;v}^*)$. What is implicit is that we have already known that $g_{i;k}^*$ is either non-negative or non-positive. On the other hand, the mutation of $p_{i;v}$ is given by
	\[
	p_{i;v'} = \begin{cases}
		-p_{k;v} \quad & \text{if $i = k$}\\
		p_{i;v} + b_{ki}^v \cdot p_{k;v}^+ \quad &\text{if $i\neq k$ and $b_{ik}\leq 0$}\\
		p_{i;v} + b_{ki}^v \cdot p_{k;v}^- \quad &\text{if $i\neq k$ and $b_{ik}> 0$}.
	\end{cases}
	\]
	Thus assuming $d_i g_{i;v}^* =  p_{i;v}$ for all $i\in I$ would imply $d_i g_{i;v'}^* = p_{i;v'}$ for all $i\in I$ as they have the same mutation formula when $p_{k;v}$ has a well-defined sign. Therefore the result is proved by induction on the distance from $v$ to $v_0$.
\end{proof}

\begin{lemma}\label{lemma: signed mutation for coefficients}
	The generalized coefficients $p_{i,j;v}$ have the following signed mutation formula. If $v\overset{k}{\longrightarrow} v'$ in $\mathfrak T_\mathbf s$, then
	\[
	p_{i,j;v'} = \begin{cases}
		-p_{k,j;v} \quad &\text{if $i = k$}\\
		p_{i,j;v} + [\varepsilon \beta_{ki}^v]_+\cdot \sum\limits_{j=1}^{r_k}{p_{k,j;v}} \quad &\text{if $i \neq k$}\\
	\end{cases}
	\]
	where $\varepsilon = \mathrm{sgn}(g_{k;v}^*)$.
\end{lemma}

\begin{proof}
	By \Cref{cor: from cone complex to c vectors}, $p_{i;v}$ is sign coherent because $g_{i;v}^*$ is so. As we have already shown in \Cref{prop: generalized sign coherence} that the sign coherence of $p_{i;v}$ implies that of $p_{i,j;v}$, the result follows by induction.
\end{proof}

\subsection{Wall-crossings}
We next study the wall-crossing functions attached to walls of the cluster chambers. Each cluster chamber $\mathcal C_{v\in \mathbf s}^+$ has exactly $n$ facets $\mathfrak d_{i; v\in \mathbf s}$ naturally indexed by $I$ (a facet has the same index as its normal vector $g_{i,v\in \mathbf s}^*$). The wall $(\mathfrak d_{i; v\in \mathbf s}, f_{i;v\in \mathbf s})$ is pulled back by $T_{v_0,v}^{-1}$ from the scattering diagram $\mathfrak D_{\mathbf s_v}$ (with coefficients $\mathbf t_v$). The wall-crossing function $f_{i;v}$ has the following description. Here we identify the initial coefficients $t_{i,j}$ with $p_{i,j}$, and endow $\mathbb P$ the semifield structure $\mathrm{Trop}(\mathbf p)$.

\begin{theorem}\label{thm: wall crossing on cluster chamber}
	The scattering diagram $\mathfrak D_\mathbf s$ has a representative in its equivalent class such that it is the union of the scattering diagram
	\[
	\mathfrak D({\Delta_\mathbf s^+}) \coloneqq \{(\mathfrak d_{i;v}, f_{i;v}) \mid i\in I, v\in \mathfrak T_\mathbf s\}\quad where \quad
	f_{i;v} = \prod_{j=1}^{r_i} \left(1 + p_{i,j;v}^{\varepsilon_{i;v}}\cdot z^{\varepsilon_{i;v} \sum\limits_{j=1}^n \beta_{ji}^v g_{j;v}}\right)
	\]
	and another one whose support is disjoint from $\Delta_\mathbf s^+$.
\end{theorem}

\begin{proof}

We prove this theorem by induction on the distance from $v$ to $v_0$. We first note that by \Cref{lemma: signed mutation for coefficients} the coefficients $p_{i,j;v}\in \mathbb P$ can be computed iteratively by signed mutations. The vectors
\[
	w_{i;v} \coloneqq \sum_{j =1 }^n \beta_{ji}^v g_{j;v} = \omega(-, \frac{d_i}{r_i} g_{i;v}^*)\in M
\]
can also be computed iteratively by signed mutations since the $g$-vectors do by \Cref{prop: mutation of g vectors}. 

Assume that the result is true for the distance between two vertices no greater $v_0$ and $v$. Suppose we have that $v\overset{k}{\longrightarrow}v'\in \mathfrak T_\mathbf s$ and that the unique path from $v_0$ to $v_1$ starts from $v_0\overset{i_0}{\longrightarrow} v_1$.

Let's look at the chambers $\tau \coloneqq \mathcal C^+_{v\in \mathbf s_1}$ and $\tau' \coloneqq \mathcal C^+_{v'\in \mathbf s_1}$ in $\mathfrak D_{\mathbf s_1}$. They have $g$-vectors satisfying
\[
	\mathbf g_{v'\in \mathbf s_1} = \mu_k^\varepsilon (\mathbf g_{v\in \mathbf s_1})
\]
where $\varepsilon = \varepsilon_{k;v\in \mathbf s_1} \coloneqq \mathrm{sgn}(g_{k;v\in \mathbf s_1}^*)$. For the wall-crossing functions, by our assumption, for $i\in I$, we have 
\[
f_{i;v\in \mathbf s_1} = \prod_{j=1}^{r_i} \left(1 +   p_{i,j; v\in \mathbf s_1} ^{\varepsilon_{i; v\in \mathbf s_1}}  z^{\varepsilon_{i; v\in \mathbf s_1}w_{i;v\in \mathbf s_1}} \right)
\]
and
\[
f_{i;v'\in \mathbf s_1} = \prod_{j=1}^{r_i} \left(1 +   p_{i,j; v'\in \mathbf s_1} ^{\varepsilon_{i; v'\in \mathbf s_1}}  z^{\varepsilon_{i; v'\in \mathbf s_1}w_{i;v'\in \mathbf s_1}} \right).
\]
These two functions are related by the signed mutation $\mu_k^\varepsilon$. More precisely, we have
\[
\mu_k^\varepsilon(\mathbf g_{v\in \mathbf s_1}^*, \mathbf p_{v\in \mathbf s_1}) = (\mathbf g_{v'\in \mathbf s_1}^*, \mathbf p_{v'\in \mathbf s_1}),\quad \mu_k^\varepsilon(\mathbf w_{v\in \mathbf s_1} ) = \mathbf w_{v'\in \mathbf s_1}.
\]
We want to pull back the chambers $\mathcal C_{v\in \mathbf s_1}^+$ and $\mathcal C_{v'\in \mathbf s_1}^+$, as well as the wall-crossing functions $f_{i;v\in \mathbf s_1}$ and $f_{i;v'\in \mathbf s_1}$ to $\mathfrak D_\mathbf s$ via the operation $(T_{i_0}^{\mathbf s})^{-1}$ to get the chambers $\sigma\coloneqq \mathcal C_{v\in \mathbf s}^+$, $\sigma' \coloneqq \mathcal C_{v'\in \mathbf s}^+$ and the wall-crossing functions $f_i \coloneqq f_{i;v\in \mathbf s}$ and $f_i'\coloneqq f_{i;v'\in \mathbf s}$ by the mutation invariance \Cref{thm: mutation invariance}. We want to show that $f_i$ and $f_i'$ are also related by signed mutations. In the following, we calculate $f_i$ and $f_i'$ in detail by applying $\tilde T_{i_0}^{-1}$ to $f_{i;v\in \mathbf s_1}$ and $f_{i;v'\in \mathbf s_1}$. This depends on the following two cases as in the proof of \Cref{prop: mutation of g vectors}: 
\begin{enumerate}
\item The two chambers $\tau$ and $\tau'$ are separated by the hyperplane $e_{i_0}^\perp$;
\item They are contained in the same half space $\mathcal H_{i_0,+}$ or $\mathcal H_{i_0,-}$.
\end{enumerate}

\textbf{Case 1}. In this case, the normal vector $g_{k; v\in \mathbf s_1}^*$ is either $e_{i_0}$ or $-e_{i_0}$. Assume it is $e_{i_0}$; the other case is similar. Then the chamber $\tau$ is in $\mathcal H_{i_0,+}$ while $\tau'$ is in $\mathcal H_{i_0,-}$. First of all, we have $f_k = f_k'$ obtained simply by reversing the monomials in $f_{k;v\in \mathbf s_1} = f_{k;v'\in \mathbf s_1}$. Since $T_{i_0}$ (as well as $\tilde T_{i_0}$) is identity on $\mathcal H_{i_0,-}$, we have for $i\neq k$, $f_i' = f_{i; v'\in \mathbf s_1}$. Note that for the signs, for $i\in I$,
\[
\varepsilon_{i;v'\in \mathbf s_1} = \varepsilon_{i;v'\in \mathbf s}
\]
unless $g_{i;v'\in \mathbf s_1}^*$ is proportional to $e_{i_0}$, which only happens for $g_{k;v'\in \mathbf s_1}^* = -g_{k;v\in \mathbf s_1}^*$, where we have
\[
\varepsilon_{k;v'\in \mathbf s_1} = -, \quad \varepsilon_{k;v'\in \mathbf s} = +.
\]
So we conclude for any $i\in I$,
\[
f_{i}' = \prod_{j=1}^{r_i}\left( 1 + (p_{i,j;v'\in \mathbf s_1}z^{w_{i;v'\in \mathbf s_1}})^{\varepsilon_{i;v'\in \mathbf s}} \right).
\]

For $f_{i; v\in \mathbf s_1}$ and $f_i$, we first consider the signs $\varepsilon_{i;v\in \mathbf s_1}$ and $\varepsilon_{i;v\in \mathbf s}$. Since the dual of $T_{i_0}^{-1}$ on $N$ only shifts in the direction of $e_{i_0}$, we have for $i\neq k$
\[
    \varepsilon_{i;v\in \mathbf s_1} = \varepsilon_{i; v\in \mathbf s},
\]
as the vectors $g_{i;v\in \mathbf s_1}^*$ and $g_{i;v\in \mathbf s}^*$ must have the same sign in all the other directions except for $e_{i_0}$, and the only one proportional to $e_{i_0}$ is $g_{k;v\in \mathbf s_1}^*$. Thus we have for $i\neq k$,
\[
    f_i = \prod_{j=1}^{r_i}\left( 1 + \tilde T_{i_0}^{-1}(p_{i,j;v\in \mathbf s_1} z^{w_{i;v\in \mathbf s_1}})^{\varepsilon_{i;v\in \mathbf s}}\right)
\]

We want to show that $f_i$ and $f_i'$ are related by the mutation $\mu_k^\delta$. Precisely, it amounts to show that
\begin{align}\label{eq: an identity for mutation of wall-crossing}
    \mu_k^\delta \left(\tilde T_{i_0}^{-1}  \left(p_{i,j;v\in\mathbf s_1}z^{w_{i;v\in \mathbf s_1}}\mid i\in I, j\in [1,r_i]  \right) \right) = \mu_k^\varepsilon \left( p_{i,j;v\in\mathbf s_1}z^{w_{i;v\in \mathbf s_1}}\mid i\in I, j\in [1,r_i] \right)	
\end{align}
where $\delta$ is the sign $\varepsilon_{k;v\in \mathbf s}$. Here we abuse the notation $\mu_k^\pm$ which acts on a tuple of functions, but it should be clear what it means. By our assumption, $\varepsilon = +$ and $\delta = -\varepsilon = -$. Then this follows from the following general fact that for any seed $(\mathbf e, \mathbf t)$ and $k\in I$, we have
\[
    \mu_k^-(\tilde T_k^{-1}(t_{i,j}z^{w_i}\mid i\in I, j\in [1,r_i]) = \mu_k^+(t_{i,j}z^{w_i}\mid i\in I, j\in [1,r_i]).
\]

\textbf{Case 2}. Suppose $\tau$ and $\tau'$ are both contained in the same half space. According to our above discussion, as in the notation of (\ref{eq: an identity for mutation of wall-crossing}), it then amounts to check that
\[
\tilde T_{i_0}^{-1}(\mu_k^\varepsilon (p_{i,j;v\in\mathbf s_1}z^{w_{i;v\in \mathbf s_1}}\mid i\in I, j\in [1,r_i])) = \mu_k^\delta (\tilde T_{i_0}^{-1}(p_{i,j;v\in\mathbf s_1}z^{w_{i;v\in \mathbf s_1}}\mid i\in I, j\in [1,r_i])).
\]
where $\delta = \varepsilon_{k; v\in \mathbf s}$. As we have discussed in the \textbf{Case 2} of the proof of \Cref{prop: mutation of g vectors}, the signs are equal: $\delta = \varepsilon$. Then the rest follows immediately from the fact that the dual of $T_{i_0, \varepsilon}$ acts as an automorphism on the data $(N,\omega)$. 
\end{proof}

\section{\texorpdfstring{Reconstruct $\mathscr A^\mathrm{prin}$}{Reconstruct A prin}}\label{section: reconstruct cluster algebra}

In this section, we see how to reconstruct the generalized cluster algebra $\mathscr A^\mathrm{prin}(\mathbf s)$ as well as the variety $\mathcal A^\mathrm{prin}(\mathbf s)$ from $\mathcal X_{\mathbf s, \lambda}$ through $\mathfrak D_\mathbf s$.

\subsection{Reconstruct $\mathscr A^\mathrm{prin}(\mathbf s)$ from $\mathfrak D_\mathbf s$}
Given fixed data $\Gamma$ and an $\mathcal A$-seed with principal coefficients $\mathbf s = (\mathbf e, \mathbf p)$, denote by $\mathscr A^\mathrm{prin}(\mathbf s)$ the corresponding generalized cluster algebra. Recall that we denote by $x_{i;v}$ the cluster variables associated to the seed $\mathbf s_v$.

Consider the generalized cluster scattering diagram $\mathfrak D_\mathbf s$, whose wall-crossings act on $\widehat{\Bbbk[P]}$ by automorphisms. For two vertices $v,v'\in \mathfrak T_\mathbf s$, let $\gamma$ be a path from the chamber $\mathcal C_{v\in \mathbf s}^+$ to $\mathcal C_{v'\in \mathbf s}^+$ and consider the path-ordered product
\[
    \mathfrak p_{v,v'} = \mathfrak p_{v,v'}^\mathbf s \coloneqq \mathfrak p_{\gamma, \mathfrak D_\mathbf s}\colon \widehat{\Bbbk[P]}\rightarrow \widehat{\Bbbk[P]}.
\] 
Since $\mathfrak D_\mathbf s$ is consistent and one can always choose some $\gamma$ contained in the cluster complex, the path-ordered product $\mathfrak p_{v,v'}$ can also be viewed as an automorphism of $\mathrm{Frac}(M \oplus \mathbb P)$.

\begin{proposition}\label{prop: cluster variable by path ordered product}
	Let $\mathcal C_{v\in \mathbf s}^+$ be a cluster chamber and $\mathbf g_{v}$ the set of $g$-vectors. Then for any $i\in I$,
	\[
	   x_{i;v} = \mathfrak p_{v,v_0}(z^{g_{i;v}}) \in \mathrm{Frac}(M\oplus \mathbb P).
	\]
\end{proposition}

\begin{proof}
	We prove this by induction on the distance from $v$ to $v_0$ in $\mathfrak T_\mathbf s$. Suppose the statement is true for a vertex $v\in \mathfrak T_\mathbf s$ and we have $v\overset{i}{\longrightarrow} v'$ in $\mathfrak T_\mathbf s$. Then the chambers $\mathcal C_v^+$ and $\mathcal C_{v'}^+$ are separated by the wall $\mathfrak d_{i;v}$ with the wall-crossing $f_{i;v}$ given in \Cref{thm: wall crossing on cluster chamber}. Denote $\varepsilon = \mathrm{sgn}(g_{i;v}^*) \in \{+, -\}$. Then we have
	\[
	   \mathfrak p_{v',v}(z^{g_{i;v'}}) = z^{g_{i;v'}} \prod_{j=1}^{r_i} \left(1 + p_{i,j;\mathbf s_v}^{\varepsilon}\cdot z^{\sum\limits_{j=1}^n \varepsilon \beta_{ji}^v g_{j;v}}\right)^{-\langle g_{i;v'}, g_{i;v}^* \rangle}.
	\]
	By \cref{prop: mutation of g vectors}, we have
	\[
	   g_{i;v'} = -g_{i;v} + \sum_{j=1}^n \left[-\varepsilon r_i\beta_{ji}^v \right]_+ g_{j;v}.
	\]
	This leads to
	\[
	   \mathfrak p_{v',v}(z^{g_{i;v'}}) = z^{-g_{i;v}} \prod_{j=1}^{r_i} \left(z^{\sum\limits_{j\in I}[-\varepsilon \beta_{ji}^v]_+g_{j;v}} + p_{i,j;\mathbf s_v}^\varepsilon \cdot z^{ \sum\limits_{j\in I} [\varepsilon \beta_{ji}^v]_+ g_{j;v}}\right).
	\]
	Note that by sign coherence, $p_{i,j; \mathbf s_v}$ has the same sign as $\varepsilon$. So the above equation is exactly the exchange relation of cluster variables. Applying the path-ordered product $\mathfrak p_{v,v_0}$ on both sides of the above equality finishes the induction.
\end{proof}

By the generalized Laurent phenomenon \Cref{thm: generalized laurent}, we know that $x_{i;v}$ actually lives in $\Bbbk[M\oplus \mathbb P]$.

\begin{corollary}\label{cor: cluster variable parametrized by g-vector}
	The set of cluster variables of $\mathscr A^\mathrm{prin}(\mathbf s)$ is in bijection with the set of $g$-vectors.
\end{corollary}

\begin{proof}
We send a cluster variable $x_{i;v}$ to the $g$-vector $g_{i;v}$. To show that $x_{i;v}$ is uniquely determined by $g_{i;v}$, we observe that the formula $\mathfrak p_{v,v_0}(z^{g_{i;v}})$ is independent of the choice of $v$. Suppose there is another chamber $\mathcal C_{v'\in \mathbf s}^+$ such that $g_{i;v}$ is one of the generators. Choose a path $\gamma$ from $\mathcal C_{v\in \mathbf s}^+$ to $\mathcal C_{v'\in \mathbf s}^+$ close enough to the ray $\mathbb R_+g_{i;v}$ so that it only crosses walls containing $\mathbb R_+g_{i;v}$. The two path-ordered products $\mathfrak p_{v,v_0}$ and $\mathfrak p_{v',v_0}$ differ by $\mathfrak p_\gamma$, which acts on $z^{g_{i;v}}$ by identity. Thus $\mathfrak p_{v,v_0}(z^{g_{i;v}}) = \mathfrak p_{v',v_0}(z^{g_{i;v}})$.
\end{proof}

\subsection{Reconstruct $\mathcal A_{\mathbf s}^\mathrm{prin}$ from $\mathfrak D_\mathbf s$}

Recall that there is a surjective map from $\mathfrak T_\mathbf s$ to $\Delta_\mathbf s^+$ (the set of cluster chambers) sending $v$ to $\mathcal C_{v\in \mathbf s}^+$. For each vertex $v\in \mathfrak T_\mathbf s$, we associate a torus $T_{N,v}(R) = T_{N}(R)$. To a pair of vertices $v$ and $v'$, we associate the birational morphism
\[
\mathfrak q_{v, v'} = \mathfrak q_{v,v'}^\mathbf s\colon T_{N,v}(R)\dasharrow T_{N,v'}(R),\quad  \mathfrak q_{v,v'}^*\coloneqq \mathfrak p_{v',v}.
\] 
Then there is an $R$-scheme obtained by glueing $T_{N,v}(R)$, $v\in \mathfrak T_\mathbf s$ via these birational morphisms
\[
\mathcal A^\mathrm{prin}_{\mathrm{scat},\mathbf s} \coloneqq \bigcup_{v\in \mathfrak T_\mathbf s} T_{N,v}(R).
\]
One can actually relate $\mathcal A^\mathrm{prin}_{\mathrm{scat},\mathbf s}$ to the previously defined cluster variety 
$$ \mathcal A_\mathbf s^\mathrm{prin} \coloneqq \bigcup_{v\in \mathfrak T_\mathbf s} T_{N,\mathbf s_v}(R),$$ which is obtained by  glueing together the same set of tori via $\mathcal A$-cluster mutations.

Recall the piecewise linear map $T_{v_0,v} \colon M_\mathbb R \rightarrow M_\mathbb R$ that sends the cluster chamber $\mathcal C_{v\in \mathbf s}^+$ to $\mathcal C_{\mathbf s_v}^+$. When restricted to a domain of linearity, $T_{v_0,v}$ becomes a linear automorphism on $M$. Denote the restriction of $T_{v_0,v}$ on $\mathcal C^+_{v\in \mathbf s}$ by $T_{v_0,v}\vert_{\mathcal C_{v\in \mathbf s}^+}$. In particular, $T_{v_0,v_0}\vert_{\mathcal C_\mathbf s^+}$ is the identity map. These linear isomorphisms induce isomorphisms (or $R$-schemes) between tori
\[
	\psi_{v_0,v}\colon T_{N,\mathbf s_v}(R) \rightarrow T_{N,v\in \mathbf s}(R),\quad \psi_{v_0,v}^*(z^m) = z^{T_{v_0,v}\vert_{\mathcal C_{v\in \mathbf s}^+}(m)}.
\]

\begin{proposition}\label{prop: reconstruct cluster variety}
	The isomorphisms $\psi_{v_0,v}$ glue to be an isomorphism
	\[
	\psi_{v_0}\colon  \mathcal A_\mathbf s^\mathrm{prin} \rightarrow \mathcal A^\mathrm{prin}_{\mathrm{scat},\mathbf s}.
	\]
\end{proposition}

\begin{proof}
The morphisms $\mu_{v,v'}$ (resp. $\mathfrak q_{v,v'}$) are generated $\mu_{v_0,v}$ (resp. $\mathfrak q_{v_0,v}$) for all $v$ in $\mathfrak T_\mathbf s$. So the statement is equivalent to the commutativity of the following diagram (for any $v$).
	\[ \begin{tikzcd}[sep = huge]
	T_{N, \mathbf s} \ar[d, dashed, "\mu_{v_0,v}"] \ar[r, "\psi_{v_0,v_0} = \mathrm{id}" ] & T_{N, v_0\in \mathbf s} \ar[d, dashed, "\mathfrak q_{v_0,v}"]\\
	T_{N, \mathbf s_{v}} \ar[r, "\psi_{v_0, v}"]	 & T_{N, v\in \mathbf s}
	\end{tikzcd} \]

To show $\mathfrak q_{v_0,v} = \psi_{v_0,v}\circ \mu_{v_0,v}$, we pull back the functions $z^{g_{i;v}}$ (for all $i\in I$) via these birational morphisms. On the left hand side, we get the cluster variables
\[
x_{i;v} = \mathfrak q_{v_0,v}^*(z^{g_{i;v}})
\]
by \Cref{prop: cluster variable by path ordered product}. On the right hand side, these $z^{g_{i;v}}$ get pulled back to $z^{e_{i;v}^*}$ by $\psi_{v_0,v}^*$ as $T_{v_0,v}\vert_{\mathcal C^+_{v\in \mathbf s}}$ sends the chamber $\mathcal C_{v\in \mathbf s}^+$ to the chamber $\mathcal C_{\mathbf s_v}^+$. Then via $\mu_{v_0,v}^*$, we still get cluster variables
\[
x_{i;v} = \mu_{v_0,v}^*(z^{e_{i;v}^*}).
\]
As $\{g_{i;v}\mid i\in I\}$ form a basis of $M$, we conclude that $\mathfrak q_{v_0,v} = \psi_{v_0,v}\circ \mu_{v_0,v}$, which finishes the proof.

\end{proof}

We next see in a certain sense the variety $\mathcal A^\mathrm{prin}_\mathbf s$ is independent of $\mathbf s$. This is a subtle issue as for the cluster algebra $\mathscr A^\mathrm{prin}(\mathbf s)$, the initial seed $\Sigma(\mathbf s)$ is distinguished from others since it has principal coefficients.

To resolve this, we again treat $\mathbb P$ as only a multiplicative abelian group. Consider $\mathbf s' = \mu_k^+(\mathbf s)$ in the sense of \Cref{thm: mutation invariance}. The tree $\mathfrak T_{\mathbf s'}$ is naturally embedded in $\mathfrak T_{\mathbf s}$, along with the association of seeds with coefficients. First of all, it is clear that the inclusion
\[
\bigcup_{v\in \mathfrak T_{\mathbf s'}} T_{N,v\in \mathbf s} \subset \mathcal A_{\mathrm{scat},\mathbf s}^\mathrm{prin}
\]
is an equality. The glueing maps are given by path-ordered products of $\mathfrak D_\mathbf s$.

Consider for $v\in \mathfrak T_{\mathbf s'}$, the isomorphism (of $R$-schemes)
\[
\varphi_v\colon T_{N,v\in \mathbf s'} \rightarrow T_{N, v\in \mathbf s}
\]
such that $\varphi_v^* \colon \Bbbk[M\oplus \mathbb P] \rightarrow \Bbbk[M\oplus \mathbb P]$ is given by the linear transformation
\[
T_k \mid_{\mathcal C_{v\in \mathbf s}^+} \colon M\oplus \mathbb P \rightarrow M\oplus \mathbb P.
\]

\begin{proposition}\label{prop: cluster variety change of seed}
The maps $\varphi_v$ for $v\in \mathfrak T_{\mathbf s'}$ glue together to have an isomorphism of $\Bbbk[\mathbb P]$-schemes
\[
\varphi \colon \mathcal A_{\mathrm{scat}, \mathbf s'}^\mathrm{prin} \rightarrow \mathcal A_{\mathrm{scat}, \mathbf s}^\mathrm{prin}.
\]	
\end{proposition}

\begin{proof}
	Let $v$ and $v'$ be two vertices in $\mathfrak T_{\mathbf s'}$. Since each $\phi_v$ is an isomorphism, the statement is equivalent to the commutativity of the following diagram (for any $v$ and $v'$).
	\[ \begin{tikzcd}[sep = huge]
	T_{N, v\in \mathbf s'} \ar[d, dashed, "\mathfrak q_{v,v'}^{\mathbf s'}"] \ar[r, "\varphi_{v}" ] & T_{N, v\in \mathbf s} \ar[d, dashed, "\mathfrak q_{v,v'}^\mathbf s"]\\
	T_{N, v'\in \mathbf s'} \ar[r, "\varphi_{v'}"]	 & T_{N, v'\in \mathbf s}
	\end{tikzcd} \]
In terms of algebras, this amounts to show 
\[
T_k\mid_{\mathcal C_{v\in \mathbf s}^+} \circ \mathfrak p_{v,v'}^\mathbf s = \mathfrak p_{v,v'}^{\mathbf s'}\circ T_k\mid_{\mathcal C_{v'\in \mathbf s}^+} \colon \Bbbk[M\oplus \mathbb P] \dashrightarrow \Bbbk[M\oplus \mathbb P].
\]	
If the two chambers $\mathcal C_{v\in \mathbf s}^+$ and $\mathcal C_{v'\in \mathbf s}^+$ are on the same side of the hyperplane $e_k^\perp$, the above equality is just (\ref{1}) . If they are separated by $e_k^\perp$, it is the same as (\ref{eq: a mutation invariance equation}) and has been check in (\ref{eq: an identity for mutation invariance}).
\end{proof}

Combined with \Cref{prop: reconstruct cluster variety}, we see that the construction $\mathcal A_\mathbf s^\mathrm{prin}$ is independent of $\mathbf s$. In terms of the corresponding cluster algebra $\mathscr A^\mathrm{prin}(\mathbf s)$, once it has principal coefficients on some seed $\mathbf s$, it can be made to do so at any seed mutation equivalent to $\mathbf s$.

\subsection{Broken lines and theta functions}

This section is a recast of \cite[Section 3]{gross2018canonical} in the generalized situation. Recall the setting of scattering diagrams in \Cref{def: tropical vertex sd}.

\begin{definition}[Broken line, cf. {\cite[Definition 3.1]{gross2018canonical}}]
	Let $\mathfrak D$ be a scattering diagram over $\widehat{\Bbbk[P]}$ with a monoid map $r\colon P\rightarrow M$. Let $p_0\in P\setminus \ker(r)$ and $Q\in M_\mathbb R\setminus \mathrm{Supp}(\mathfrak D)$. A \emph{broken line} for $p_0$ with endpoint $Q$ is a piecewise linear continuous proper map $$\gamma\colon (-\infty, 0]\rightarrow M_\mathbb R\setminus \mathrm{Sing}(\mathfrak D)$$ with a finite number of domains of linearity $L_1, L_2, \dots, L_k$ (open intervals in $(-\infty, 0]$), where each $L = L_i\subset (-\infty, 0]$ is labeled by a monomial $c_L z^{p_L} \in \Bbbk[P]$ with $p_L\in P$. This data should satisfy:
	\begin{enumerate}
		\item $\gamma(0) = Q$.
		\item If $L = L_1$ is the first domain of linearity of $\gamma$, i.e. $L = (-\infty, t)$ for some $t\leq 0$, then $c_Lz^{p_L} = z^{p_0}$.
		\item For $t\in L$ any domain of linearity, $m_L\coloneqq r(p_L)  = - \gamma'(t)$.
		\item For two consecutive domains of linearity $L = (a, t)$ ($a$ can be $-\infty$) and $L' = (t, b)$, the monomial $c_{L'}z^{p_{L'}}$ is a term in the formal power series 
			\[ \mathfrak p_{\gamma(t), \mathfrak D}(c_Lz^{p_L}) = c_Lz^{p_L} \prod_{\substack{(\mathfrak d, f_\mathfrak d)\\ \gamma(t)\in \mathfrak d}} f_\mathfrak d^{-\langle n_0, m_L\rangle}. \]
			Here $n_0\in N$ is primitive, serving as a normal vector of every $\mathfrak d$ appearing in the product such that $\langle n_0, \gamma'(t) \rangle >0$. So the power $-\langle n_0, m_L\rangle$ is always a positive integer. 
	\end{enumerate}
\end{definition}

\begin{definition}[Theta function, {\cite[Definition 3.3]{gross2018canonical}}] Let $\mathfrak D$ be a scattering diagram over $\widehat{\Bbbk[P]}$. Let $p_0\in P \setminus \ker(r)$ and $Q\in M_\mathbb R\setminus \mathrm{Supp}(\mathfrak D)$. For a broken line $\gamma$ for $p_0$ with end point $Q$, define
\[
\mathrm{Mono}(\gamma) \coloneqq c_Q z^{p_Q}
\]
where (by abuse of notation) $Q$ stands for the last linear segment of $\gamma$.	
We define the \emph{theta function} for $p_0$ with endpoint $Q$ as the formal sum
\[
\vartheta_{Q, p_0} \coloneqq \sum_{\gamma} \mathrm{Mono}(\gamma)
\]
where the sum is over the set of all broken lines for $p_0$ with endpoint $Q$.

For $p_0 = \ker (r)$, we define for any endpoint $Q$
\[
\vartheta_{Q,p_0} = z^{p_0}.
\]
\end{definition}

We collect some important properties for theta functions from \cite{gross2018canonical}.

\begin{theorem}\label{thm: theta function consistency}
	\begin{enumerate}
		\item The theta function $\vartheta_{Q,p_0}$ is in $\widehat{\Bbbk[P]}$.
		\item Suppose that $\mathfrak D$ is consistent. Then for $Q, Q'\in M_\mathbb R\setminus \mathrm{Supp}(\mathfrak D)$ whose coordinates are linearly independent over $\mathbb Q$, and $p_0\in P$,
	\[ \vartheta_{Q', p_0} = \mathfrak p_{\gamma, \mathfrak D}(\vartheta_{Q, p_0}) \]
	where $\gamma$ is a path in $\mathfrak D$ from $Q$ to $Q'$ such that its path-ordered product is well-defined.
	\end{enumerate}
\end{theorem}

\begin{proof}
	Part (1) essentially follows from the proof of \cite[Proposition 3.4]{gross2018canonical}. We are using a different monoid $P$ here, but the same proof still works with $J \coloneqq \mathfrak m_P = P\setminus M$.
		
	Part (2), as pointed out in the proof of \cite[Theorem 3.5]{gross2018canonical}, is again a special case of \cite[Section 4]{carl2010tropical}. Here the generic condition on the coordinates of $Q$ and $Q'$ is just to make sure that any broken line does not cross any joint of $\mathfrak D$. Modulo $\mathfrak m_P^k$, the independence of $\vartheta_{Q,m_0}$ on $Q$ within one chamber follows from \cite[Lemma 4.7]{carl2010tropical}. The compatibility between $Q$ and $Q'$ in different chambers follows from \cite[Lemma 4.9]{carl2010tropical}. See also a more general discussion on the globalness of theta functions in \cite[Section 3.3]{gross2016theta}.
\end{proof}

In the case of generalized cluster scattering diagrams $\mathfrak D_{\mathbf s}$ (see \Cref{def: generalized cluster scattering}), the monoid $P$ is $M\oplus \bigoplus_{i\in I}\mathbb N^{r_i}$ (contained in $M\oplus \mathbb P$) with the natural projection $r$ to the direct summand $M$. We have the following properties of theta functions.

\begin{proposition}[Mutation invariance of broken line, cf. {\cite[Proposition 3.6]{gross2018canonical}}]\label{prop: mutation of broken line}
	The piecewise linear transformation $T_k\colon M_\mathbb R\rightarrow M_\mathbb R$ (with a lift on $M\oplus \mathbb P$) defines a one-to-one correspondence $\gamma\mapsto T_k(\gamma)$ between broken lines for $p_0$ with endpoint $Q$ for $\mathfrak D_\mathbf s$ and broken lines for $T_k(p_0)$ with endpoint $T_k(Q)$ for $\mathfrak D_{\mu_k(\mathbf s)}$. This correspondence satisfies, depending on whether $Q\in \mathcal H_{k,+}$ or $\mathcal H_{k,-}$,
\[ \mathrm{Mono}(T_k(\gamma)) = T_{k,\pm}(\mathrm{Mono}(\gamma)) \]
	where $T_{k,\pm}$ acts on a monomial as in \Cref{thm: mutation invariance}. In particular, we have
	\[ \vartheta_{T_k(Q), T_k(p_0)}^{\mu_k(\mathbf s)} = T_{k,\pm}(\vartheta_{Q,p_0}^{\mathbf s}). \]
\end{proposition}

\begin{proof}
	We use $T_k(\gamma)$ to denote the piecewise linear map $T_k \circ \gamma \colon (-\infty, 0]\rightarrow M_\mathbb R$. Suppose $L$ is a domain of linearity of $\gamma$ labeled with monomial $c_Lz^{p_L}$. If $\gamma(L)$ is contained in one of the half spaces $\mathcal H_{k,\pm}$, $L$ is also a domain of linearity for $T_k(\gamma)$. We apply the action of $T_{k,\pm}$ on the monomial $c_Lz^{p_L}$ (where the sign is chosen depending on which half space $L$ is in). If $\gamma(L)$ crosses $e_k^\perp$, split $L$ into $L^+$ and $L^-$, and apply $T_{k,\pm}$ respectively to the monomial $c_Lz^{p_L}$. One then needs to check the piecewise linear path $T_k \circ \gamma$ together with the new monomial data we just obtained is a broken line for $T_k(p_0)$ with endpoint $T_k(Q)$ in $\mathfrak D_{\mu_k(\mathbf s)}$ as in \cite[Proposition 3.6]{gross2018canonical}. The inverse of the operation $\gamma \mapsto T_k\circ \gamma$ is also clear. The rest of the statement follows easily.
\end{proof}

\begin{proposition}[cf. {\cite[Proposition 3.8]{gross2018canonical}}]\label{prop: only one broken line} Consider the scattering diagram $\mathfrak D_\mathbf s$.
	\begin{enumerate}
		\item Let $Q\in \mathrm{Int}(\mathcal C_\mathbf s^+) $ be an end point, and let $p\in P$ such that $r(p) \in \mathcal C_\mathbf s^+ \cap M$. Then $\vartheta_{Q,p} = z^p$. 
		\item Let $\mathcal C_{v}^+ \in \Delta_\mathbf s^+$ be a cluster chamber for some $v\in \mathfrak T_\mathbf s$, and $Q\in \mathrm{int}(\mathcal C_{v}^+)$ and $m\in \mathcal C_v^+ \cap M$. Then $\vartheta_{Q,p} = z^p$ if $r(p) = m$.
	\end{enumerate}
\end{proposition}

\begin{proof}
	 Part (1) is essentially \cite[Proposition 3.8]{gross2018canonical}, although there the scattering diagram is actually different from $\mathfrak D_\mathbf s$ in terms of wall-crossing functions. However, the bending behavior of a broken line on a wall is totally analogous, so the exact same argument still applies.
	 
	 Part (2) is the generalized version of \cite[Corollary 3.9]{gross2018canonical}. By \Cref{prop: mutation of broken line}, the transformation $T_{v_0, v}\colon M_\mathbb R \rightarrow M_\mathbb R$ defines a one-to-one correspondence between the broken lines for $p$ with $r(p)\in \mathcal C_v^+\cap M$ and $Q\in \mathrm{int}(C_v^+)$ in $\mathfrak D_\mathbf s$, and the ones for $T_{v_0,v}(p)$ with $r(T_{v_0,v}(p)) \in \mathcal C_{\mathbf s_v}^+ \cap M$ and $T_{v_0,v}(Q) \in \mathrm{int}(\mathcal C_{\mathbf s_v}^+)$. However the only broken lines of the later is labeled by the final monomial $z^{p'}$ for $p' = T_{v_0,v}(p)$ by part (1). The result follows. 
\end{proof}

\subsection{Cluster monomials as theta functions}

\begin{definition}
Let $\mathbf s$ be a generalized $\mathcal A$-seed with principal coefficients. Then for $v\in \mathfrak T_\mathbf s$, a cluster monomial in this seed is a monomial on the torus $T_{N,v}(R) \subset \mathcal A_\mathbf s^\mathrm{prin}$ of the form $z^m$ where $m$ is a non-negative $\mathbb N$-linear combination of $\{e_{i;v}^*\mid i\in I\}$. By the Laurent phenomenon, such a monomial extends to a regular function on the whole cluster variety $\mathcal A_\mathbf s^\mathrm{prin}$. 
\end{definition}

\begin{remark}
	One may regard a cluster monomial as a function on the initial torus $T_{N,v_0}(R)$. While being a monomial on the cluster variables ${x_{i;v}}$, it is also a Laurent polynomial in the initial cluster variables $x_{i}$ by the Laurent phenomenon.
\end{remark}

The following description of cluster monomials is a generalized version of \cite[Theorem 4.9]{gross2018canonical}. It proves the positivity (see \Cref{thm: positivity introduction}) of generalized cluster monomials.

\begin{theorem}\label{thm: cluster variable as theta function}
	Let $\mathfrak D_\mathbf s$ be the generalized cluster scattering diagram of a seed $\mathbf s$. Let $Q\in \mathrm{int}(\mathcal C_\mathbf s^+)$ a general end point and $m\in \mathcal C_v^+ \cap M$ for some $v\in \mathfrak T_\mathbf s$. Then the theta function $\vartheta_{Q,m}$ is an element in $z^m\cdot \mathbb N[P]$ which expresses the cluster monomial associated to $m$ of the algebra $\mathscr A^\mathrm{prin}(\mathbf s)$ in the initial seed $\mathbf s$.
\end{theorem}

\begin{proof}
	We first note that $m$ is regarded as a point in $P$ through the inclusion of $M$ in $P$. Let $Q'$ be a base point in $\mathrm{int}(\mathcal C_v^+)$ and $\gamma$ be a path going from $Q'$ to $Q$. By part (2) of \Cref{thm: theta function consistency}, we have
	\[
	\vartheta_{m, Q} = \mathfrak p_{\gamma} (\vartheta_{m, Q'}).
	\]
	As a theta function, $\vartheta_{m, Q}$ is a (formal) sum of monomials belonging to $z^m \widehat{\Bbbk[P]}$. By the positivity \Cref{thm: generalized positivity} of $\mathfrak D_\mathbf s$, $\vartheta_{m, Q}$ has positive integer coefficients, thus an element in $z^m \widehat {\mathbb N[P]}$. By part (2) of \Cref{prop: only one broken line}, $\vartheta_{m, Q'} = z^m$. We know that the cone $\mathcal C_v^+$ has integral generators $\{g_{i;v} \mid i\in I\}$ in $M$. Thus $m$ is a non-negative linear combination of these $g$-vectors.
	
	On the other hand, by \Cref{prop: cluster variable by path ordered product}, we have the following expression of a cluster variable
	\[
	x_{i;v} = \mathfrak p_\gamma(z^{g_{i;v}}).
	\]
	It follows immediately that $\vartheta_{m, Q}$ is a monomial of these $x_{i;v}$, thus expressing a cluster monomial. Finally by the generalized Laurent phenomenon \Cref{thm: generalized laurent}, we have $\vartheta_{m,Q}\in  z^m\cdot \mathbb N[P]$.
\end{proof}

Since $\vartheta_{m,Q}$ does not depend on $Q$ as long as it is chosen generally in the positive chamber, we simply write it as $\vartheta_m$. Consider the set of functions
\[
\{\vartheta_m \mid m\in \Delta_\mathbf s^+(\mathbb Z)\},
\]
where $\Delta_\mathbf s^+(\mathbb Z) = \bigcup_{v\in \mathfrak T_\mathbf s}\mathcal C_v^+ \cap M$. These are all cluster monomials. In general, they do not form an $R$-basis of the cluster algebra $\mathscr A^\mathrm{prin}(\mathbf s)$ or the upper cluster algebra $\overline{\mathscr A}^\mathrm{prin}(\mathbf s)$. But one can follow \cite[Section 7.1]{gross2018canonical} to define the set $\Theta\subset M$ such that for any $m\in \Theta$, $\vartheta_m$ is only a sum of monomials from finitely many broken lines. Consider the free $R$-module
\[
\mathrm{mid}(\mathcal A_\mathbf s^\mathrm{prin}) \coloneqq \bigoplus_{m\in \Theta} R\cdot \vartheta_m.
\]
It is shown in \cite[Theorem 7.5]{gross2018canonical} that in the ordinary case there are natural inclusions of $R$-modules
\[
\mathscr A^\mathrm{prin}(\mathbf s)\subset \mathrm{mid}(\mathcal A_\mathbf s^\mathrm{prin}) \subset \overline{\mathscr A}^\mathrm{prin}(\mathbf s)
\]
such that for the first inclusion, cluster monomials are sent to the corresponding theta functions, and for the second inclusion, any theta function is sent to the corresponding universal Laurent polynomials on $\mathcal A_\mathbf s^\mathrm{prin}$ (see \cite[Proposition 7.1]{gross2018canonical}). We expect that this is also true in the generalized case.

\subsection{More on positivity}\label{subsection: further on positivity}

In \cite{chekhov2011teichmller}, the authors proposed a positivity conjecture which is stronger than \Cref{thm: positivity introduction}. We formulate a version here.

A generalized cluster algebra in the sense of \cite{chekhov2011teichmller} (see \Cref{subsection: chekhov-shapiro}) is called \emph{reciprocal} if any of its exchange polynomials $\theta_i(u,v)$ is monic and palindromic, i.e. $\theta_i(u,v) = \theta_i(v,u)$ and has leading coefficient $1$. In this way, the exchange polynomials do not change under mutations. Note that $\theta_i(u,v)$ can have coefficients in $\mathbb {ZP}$ (rather than just in $\mathbb P$) in general.

\begin{conjecture}[c.f. {\cite[Conjecture 5.1]{chekhov2011teichmller}}]\label{conj: chekhov shapiro}
Any cluster variable of a reciprocal generalized cluster algebra whose exchange polynomials have coefficients in $\mathbb{P}$ (or more generally in $\mathbb {NP}$) is expressed as a positive Laurent polynomial in the initial cluster, i.e. an element in $\mathbb{NP}[x_1^\pm, \dots, x_n^\pm]$ where the $x_i$'s are the initial cluster variables.
\end{conjecture}

Chekhov and Shapiro pointed out that this conjecture is true for any generalized cluster algebra associated to a surface with arbitrary orbifold points \cite[Section 5]{chekhov2011teichmller} (see also \cite{banaian2020snake} for a proof using snake graphs). The rank two case of this conjecture has been resolved by Rupel in \cite{rupel2013greedy}.

We consider here a related situation where the reciprocal assumption is not required. Let $\mathbb P$ be an abelian group of finite rank. Consider an algebraic closure $\Bbbk = \overline{\mathbb {QP}}$ of the field of rational functions $\mathbb {QP}$. Let $\mathscr A^\mathrm{prin}(\Sigma)$ be a generalized cluster algebra with principal coefficients as of \Cref{def: principal coefficients}. The coefficients group is the tropical semifield $\mathrm{Trop}(\mathbf p)$. Recall that the initial exchange polynomials have the form
\[
    \theta_i(u,v) = \prod_{j=1}^{r_i} (p_{i,j}u+v).
\]
Let $\lambda\colon \mathrm{Trop}(\mathbf p) \rightarrow \Bbbk^*$ be an evaluation (as in \Cref{subsection: specialized coefficients}) such that each $\lambda(\theta_i(u,v))$ satisfies
\begin{enumerate}
	\item[(\textbf A)] All its coefficients are in $\mathbb{ZP}$ (in $\mathbb {NP}$ if assuming positivity);
	\item[(\textbf B)] $\lambda(\prod_{j=1}^{r_i} p_{i,j})$ is an element in $\mathbb P$.
\end{enumerate}
By the mutation formula of coefficients, the exchange polynomials after any steps of mutations still satisfy these two conditions. Therefore the cluster algebra with special coefficients $\mathscr A^\mathrm{prin}(\Sigma,\lambda)$ can be viewed as a generalized cluster algebra of \cite{chekhov2011teichmller} (with the coefficients group $\mathbb P$). Note that any reciprocal generalized cluster algebra can be obtained this way.

The scattering diagram $\lambda(\mathfrak D_\mathbf s)$ (see \Cref{subsection: sd with special coefficients}) is responsible for $\mathscr A^\mathrm{prin}(\Sigma, \lambda)$. It is over $\widehat{\Bbbk[M\oplus \prod_{i\in I}\mathbb N]}$ with formal parameters $t_i$. Note that by the generalized Laurent phenomenon, the cluster variables of $\mathscr A^\mathrm{prin}(\Sigma, \lambda)$ are all in $\mathbb {ZP}[x_1^\pm, \dots, x_n^\pm]$.

\begin{theorem}\label{thm: stronger positivity}
	Let $\mathscr A^\mathrm{prin}(\Sigma,\lambda)$ be a generalized cluster algebra as above assuming (\textbf A), (\textbf B), and that the initial exchange polynomials have coefficients in $\mathbb {NP}$. Let $\mathbf s$ be an $\mathcal A$-seed such that $\Sigma(\mathbf s) = \Sigma$. If there exists a representative of $\lambda(\mathfrak D_\mathbf s)$ such that every wall-crossing function is in $\widehat{\mathbb {NP}[M\oplus \prod_{i\in I}\mathbb N]}$, then any cluster variable is expressed as a positive Laurent polynomial in the initial cluster, i.e. an element in $\mathbb {NP}[x_1^\pm, \dots, x_n^\pm]$.
\end{theorem}

\begin{proof}
	As in \Cref{thm: cluster variable as theta function}, the positivity of cluster variables follows from the positivity of the scattering diagram $\lambda(\mathfrak D_\mathbf s)$ since every broken line ends with a monomial with coefficients in $\mathbb {NP}\subset \Bbbk$. Expressing a cluster variable as a theta function for $\lambda(\mathfrak D_\mathbf s)$ (and evaluated at $t_i=1$ where the $t_i$'s are the standard generators of $\prod_{i\in I}\mathbb N$), the result follows.
\end{proof}

If $\mathscr A^\mathrm{prin}(\Sigma,\lambda)$ is of finite type (i.e. there are only finitely many distinguished cluster variables), then the cluster complex $\Delta_\mathbf s^+$ is finite and complete in $M_\mathbb R$ by \Cref{cor: cluster variable parametrized by g-vector}. By \Cref{thm: wall crossing on cluster chamber}, we have that $\mathfrak D_\mathbf s = \mathfrak D(\Delta_\mathbf s^+)$ and the wall-crossing function on any facet of any cluster chamber has coefficients in $\mathbb {NP}$ under the evaluation $\lambda$ if assuming so for the initial ones. Then the positivity follows in this case from \Cref{thm: stronger positivity}. It is not hard to check that in \Cref{ex: example for k2} the expansion of the wall-crossing function $f_{\mathbb R_{\geq 0}(1, -1)}$ has every coefficient in $\mathbb N[s_1s_2, s_1 + s_2, t_1t_2, t_1 + t_2]$. By the description in \Cref{ex: example for k2} of all other walls, all wall-crossings functions in this scattering diagram are positive in this sense. This then implies all cluster variables are positive, i.e., have coefficients in $\mathbb N[s_1s_2, s_1 + s_2, t_1t_2, t_1 + t_2]$.

\bibliographystyle{amsalpha-fi-arxlast}

\providecommand{\bysame}{\leavevmode\hbox to3em{\hrulefill}\thinspace}
\providecommand{\MR}{\relax\ifhmode\unskip\space\fi MR }
\providecommand{\MRhref}[2]{%
  \href{http://www.ams.org/mathscinet-getitem?mr=#1}{#2}
}
\providecommand{\href}[2]{#2}

\end{document}